\documentclass[11pt]{article}
\usepackage[hscale=0.8,vscale=0.8]{geometry}

\usepackage{amsmath}
\usepackage{amsfonts}
\usepackage{latexsym}
\usepackage{graphicx}
\usepackage{caption}
\usepackage{subcaption}
\usepackage{multicol}
\usepackage{multirow}
\usepackage{amssymb}
\usepackage{mathtools}
\usepackage{hhline}
\usepackage{amsthm}
\usepackage{float}
\usepackage{cite}
\usepackage{color}
\usepackage{chngcntr}
\counterwithin{equation}{section}

\allowdisplaybreaks

\def\e{{\bf e}}
\def\u{{\bf u}}
\renewcommand{\v}{{\bf v}}
\def\q{{\bf q}}
\def\w{{\bf w}}

\def\n{{\bf n}}
\def\f{{\bf f}}

\def\x{{\bf x}}

\def\z{{\bf z}}

\def\btau{\boldsymbol{\tau}}
\def\bbeta{\boldsymbol{\eta}}
\def\bs{\boldsymbol{\sigma}}
\def\bxi{\boldsymbol{\xi}}
\def\bzeta{\boldsymbol{\zeta}}
\def\bphi{\boldsymbol{\phi}}
\def\bchi{\boldsymbol{\chi}}

\def\U{{\bf U}}
\def\V{{\bf V}}

\def\D{{\bf D}}
\def\X{{\bf X}}

\def\A{{\bf A}}
\def\B{{\bf B}}
\def\C{{\bf C}}

\def\E{{\bf E}}
\def\H{{\bf H}}

\def\bI{{\bf I}}

\def\d{\partial}

\def\lam{\lambda}
\def\R{{\bf R}}
\def\grad{\nabla}
\def\div{\grad\cdot}

\def\O{\Omega}
\def\Gam{\Gamma}

\def\<{\langle}
\def\>{\rangle}
\def\CO2{{CO$_{2}$}}

\def\dt{d_\tau}

\def\domf{\Omega_f}

\def\domp{\Omega_p}

\newtheorem{lemma}{Lemma}[section]
\newtheorem{theorem}{Theorem}[section]
\newtheorem{corollary}{Corollary}[section]

\begin{document}

\title{A Lagrange multiplier method for a Stokes-Biot fluid-poroelastic structure interaction model}

\author{Ilona Ambartsumyan\thanks{Department of Mathematics, University of
		Pittsburgh, Pittsburgh, PA 15260, USA;~{\tt  ila6@pitt.edu, elk58@pitt.edu, 
                yotov@math.pitt.edu}; partially supported by DOE grant
                DE-FG02-04ER25618 and NSF grant DMS 1418947.}~\and
		Eldar Khattatov\footnotemark[1]\and
		Ivan Yotov\footnotemark[1]\and
		Paolo Zunino\footnotemark[2]\thanks{Department of Mechanical Engineering 
                \& Materials Science, Pittsburgh, PA 15261, USA;~{\tt paz13@pitt.edu}; 
                partially supported by DOE grant DE-FG02-04ER25618.}
                \footnotemark[3]\thanks{MOX, Department of Mathematics, Politecnico di Milano, 
                Italy.}
}

\date{\today}
\maketitle

\begin{abstract}
We study a finite element computational model for solving the coupled
problem arising in the interaction between a free fluid and a fluid in
a poroelastic medium. The free fluid is governed by the Stokes
equations, while the flow in the poroelastic medium is modeled using
the Biot poroelasticity system. Equilibrium and kinematic conditions
are imposed on the interface. A mixed Darcy formulation is employed,
resulting in continuity of flux condition of essential type. A
Lagrange multiplier method is employed to impose weakly this
condition. A stability and error analysis is performed for the
semi-discrete continuous-in-time and the fully discrete
formulations. A series of numerical experiments is presented to
confirm the theoretical convergence rates and to study the
applicability of the method to modeling physical phenomena and the
sensitivity of the model with respect to its parameters.
\end{abstract}

\section{Introduction}
In this paper we study the interaction of a free incompressible
viscous Newtonian fluid with a fluid within a poroelastic medium.
This is a challenging multiphysics problem with applications to
predicting and controlling processes arising in groundwater flow in
fractured aquifers, oil and gas extraction, arterial flows, and
industrial filters. In these applications, it is important to model
properly the interaction between the free fluid with the fluid within
the porous medium, and to take into account the effect of the
deformation of the medium. For example, geomechanical effects play an
important role in hydraulic fracturing, as well as in modeling
phenomena such as subsidence and compaction.

We adopt the Stokes equations to model the free fluid and the Biot
system \cite{Biot} for the fluid in the poroelastic media. In the
latter, the volumetric deformation of the elastic porous matrix is
complemented with the Darcy equation that describes the average
velocity of the fluid in the pores. The model features two
different kinds of coupling across the interface: Stokes-Darcy
coupling \cite{DMQ,Gir-Riv-NS-D,LSY,Mo-Xu,RivYot,VasWangYot,Xie-Xu-Xue} and
fluid-structure interaction (FSI)
\cite{Badia-Quaini-Quart,deparis2,nobile2008effective,quaini2007semi,Xu-Yang}.

The well-posedness of the mathematical model based on the Stokes-Biot
system for the coupling between a fluid and a poroelastic structure is
studied in \cite{Showalter-Biot-Stokes}. A numerical study of the
problem, using the Navier-Stokes equations for the fluid, is presented
in \cite{Badia-Quaini-Quart}, utilizing a variational multiscale
approach to stabilize the finite element spaces. The problem is solved
using both a monolithic and a partitioned approach, with the latter
requiring subiterations between the two problems.  The reader is also
referred to \cite{BYZ}, where a non-iterative operator-splitting method
for a coupled Navier-Stokes-Biot model is developed.

An alternative partitioned approach for the coupled Stokes-Biot
problem based on the Nitsche's method is developed in
\cite{Bukac-Nitsche}. The resulting method is loosely coupled and
non-iterative with conditional stability. Unlike the method in
\cite{BYZ}, which is suitable for the pressure formulation of Darcy
flow, the Nitsche's method can handle the mixed Darcy formulation.  It
does, however, suffer from a reduced convergence, due to the splitting
across the interface.  This is typical for Nitsche's splittings, see
e.g. \cite{burman2009stabilization} for modeling of FSI. Possible
approaches to alleviate this problem include iterative correction
\cite{Burman-explicit} and the use of the split method as a
preconditioner for the monolithic scheme \cite{Bukac-Nitsche}.

In applications to flow in fractured poroelastic media, an alternative
modeling approach is based on a reduced-dimension fracture model. We
mention recent work using the Reynolds lubrication equation
\cite{Ganis-2014,Girault-2015,Lee-phase-field,Mikelic-phase-field} as well as an averaged Brinkman
equation \cite{Bukac-reduced-fracture}. Earlier works that do not
account for elastic deformation of the media include averaged Darcy
models
\cite{Martin20051667,Frih20121043,Morales-2010,DAngelo2012465,Fumagalli-2012},
Forchheimer models \cite{Frih200891}, and Brinkman models
\cite{lesinigo2011multiscale}.

In this work we focus on the monolithic scheme for the
full-dimensional Stokes-Biot problem with the approximation of the
continuity of normal velocity condition through the use of a Lagrange
multiplier. We consider the mixed formulation for Darcy flow in the
Biot system, which provides a locally mass conservative flow
approximation and an accurate Darcy velocity. However, this
formulation results in the continuity of normal velocity condition
being of essential type, which requires weak enforcement through
either a penalty or a Lagrange multiplier formulation. Here we study
the latter, as an alternative to the previously developed Nitsche
formulation \cite{Bukac-Nitsche}. The advantage of the Lagrange
multiplier method is that it doesn't involve a penalty parameter and
it can enforce the the continuity of normal velocity with machine
precision accuracy on matching grids
\cite{ambartsumyan2014simulation}. The method is also 
convergent on non-matching grids.  After deriving a finite element
based numerical approximation scheme for the Stokes-Biot problem, we
provide a detailed theoretical analysis of stability and error
estimates. A critical component of the analysis is the construction of
a finite element interpolant in the space of velocities with weakly
continuous normal components. This interpolant is shown to have
optimal approximation properties, even for grids that do not match
across the interface.  The numerical tests confirm the theoretical
convergence rates and illustrate that the method is applicable for
simulating real world phenomena with a wide range of realistic
physical parameters.

An additional advantage of the Lagrange multiplier formulation is that
it is suitable for efficient parallel domain decomposition algorithms
for the solution of the coupled problem, via its reduction to an
interface problem, see, e.g. \cite{VasWangYot} for the Stokes-Darcy
problem. It can also lead to multiscale approximations through the use
of a coarse-scale Lagrange multiplier or mortar space
\cite{APWY,ganis2009implementation,GirVasYot}. However, this topic is
beyond the scope of the paper and it will be investigated in the
future.

The remainder of the manuscript is organized as follows. In Section 2
we present the mathematical model. Section 3 is devoted to the
semi-discrete continuous-in-time numerical scheme and the uniqueness
and existence of its solution, followed by its stability analysis in
Section 4. A detailed error analysis is presented in Section 5, which
gives insight on the expected convergence rates with different choice
of finite element spaces. Section 6 and the Appendix present the
analysis for the fully discrete scheme. Extensive numerical
experiments are discussed in Section 7, while Section 8 sums up our
findings.

\section{Stokes-Biot model problem}

We consider a multiphysics model problem for free fluid's interaction with a flow in a 
deformable porous media, where the simulation domain $\Omega \subset
\R^d$, $d = 2,3$, is a union of non-overlapping regions $\O_f$ and $\O_p$. 
Here $\O_f$ is a free fluid region
with flow governed by the Stokes equations and $\O_p$ is a poroelastic
material governed by the Biot system. For simplicity of notation, we assume 
that each region is connected. The extension to non-connected regions is 
straightforward. Let $\Gamma_{fp} = \d \O_f \cap
\d \O_p$.  Let $(\u_\star,p_\star)$ be the velocity-pressure pair in
$\O_\star$, $\star = f$, $p$, and let $\bbeta_p$ be the displacement
in $\O_p$.  Let $\mu > 0$ be the fluid viscosity,
let $\f_\star$ be the body force terms, and let
$q_\star$ be external source or sink terms.  Let $\D(\u_f)$ and
$\bs_f(\u_f,p_f)$ denote, respectively, the deformation rate tensor
and the stress tensor:
$$
\D(\u_f) = \frac 12 (\grad \u_f + \grad\u_f^T), \qquad
\bs_f(\u_f,p_f) = -p_f \bI + 2\mu \D(\u_f).
$$
In the free fluid region $\O_f$, $(\u_f,p_f)$ satisfy 
the Stokes equations
\begin{align}
	-\div \bs_f(\u_f,p_f)  = \f_f \quad \mbox{in } \O_f \times (0,T] \label{stokes1} \\
	 \div \u_f =  q_f  \quad \mbox{in } \O_f \times (0,T], \label{stokes2}
\end{align}
where $T > 0$ is the final time.
Let $\bs_e(\bbeta_p)$ and $\bs_p(\bbeta_p,p_p)$ be the elastic and poroelastic stress tensors, 
respectively:
\begin{equation}\label{stress-defn}
	\bs_e(\bbeta_p) = \lambda_p (\div \bbeta_p) \bI + 2\mu_p \D(\bbeta_p), \qquad
	\bs_p(\bbeta_p,p_p) = \bs_e(\bbeta_p) - \alpha p_p \bI,
\end{equation}
where $ 0 < \lambda_{min} \le \lambda_p(\x) \le \lambda_{max}$ and 
$0 < \mu_{min} \le \mu_p(\x) \le \mu_{max} $ are the Lam\'{e} parameters and 
$0 \le \alpha \le 1$ is the Biot-Willis constant. 
The poroelasticity region $\O_p$ is governed by the quasi-static 
Biot system \cite{Biot}
\begin{eqnarray}
	- \div \bs_p(\bbeta_p,p_p) = \f_p, \quad \mu K^{-1}\u_p + \grad p_p = 0, \, 
        &&\mbox{in } \O_p\times (0,T], \label{eq:biot1}\\ 
	\frac{\d}{\d t} \left(s_0 p_p + \alpha \div \bbeta_p\right) + \div \u_p = q_p 
	\quad &&\mbox{in } \O_p\times (0,T], \label{eq:biot2}
\end{eqnarray}
where $s_0 \ge 0$ is a storage coefficient and $K$ the symmetric and uniformly 
positive definite rock permeability tensor, satisfying, for some constants 
$0 < k_{min} \le k_{max}$,
$$
\forall \, \bxi \in \R^d, \,\,  
k_{min} \bxi^T\bxi \le \bxi^T K(\x) \bxi \le k_{max} \bxi^T\bxi, \,\, \forall \, 
\x \in \Omega_p.
$$

Following \cite{Showalter-Biot-Stokes,Badia-Quaini-Quart}, the 
{\it interface conditions} on the 
fluid-poroelasticity interface $\Gamma_{fp}$ are {\it mass conservation},
{\it balance of stresses}, and the Beavers-Joseph-Saffman (BJS) condition
\cite{BJ,Saffman} modeling {\it slip with friction}: 
\begin{align}
	\label{eq:mass-conservation}
	& \u_f\cdot\n_f + \left(\frac{\d \bbeta_p}{\d t} + \u_p\right)\cdot\n_p = 0 
&& \mbox{on } \Gamma_{fp}\times (0,T],
	\\
        \label{balance-stress}
	& - (\bs_f \n_f)\cdot\n_f =  p_p, \qquad 	\bs_f\n_f + \bs_p\n_p = 0
&& \mbox{on } \Gamma_{fp}\times (0,T],
	\\
	\label{Gamma-fp-1}
	& - (\bs_f\n_f)\cdot\btau_{f,j} = \mu\alpha_{BJS}\sqrt{K_j^{-1}}
	\left(\u_f - \frac{\d \bbeta_p}{\d t}\right)\cdot\btau_{f,j} 
&& \mbox{on } \Gamma_{fp}\times (0,T],
\end{align}
%
%as well as {\it conservation of momentum}:
%
%\begin{equation}\label{Gamma-fp-2} 
%	\bs_f\n_f + \bs_p\n_p = 0
%	\quad \mbox{on } \Gamma_{fp}\times [0,T],
%\end{equation}
%
where $\n_f$ and $\n_p$ are the outward unit normal vectors to $\d
\O_f$, and $\d \O_p$, respectively, $\btau_{f,j}$, $1 \le j \le d-1$,
is an orthogonal system of unit tangent vectors on $\Gamma_{fp}$, $K_j
= (K \btau_{f,j}) \cdot \btau_{f,j}$, and $\alpha_{BJS} \ge 0$ is an
experimentally determined friction coefficient.  We note that the
continuity of flux constrains the normal velocity of the solid
skeleton, while the BJS condition accounts for its tangential
velocity.

The above system of equations needs to be complemented by a set of
boundary and initial conditions. Let $\Gamma_f = \partial \O_f \cap \d\Omega$
and $\Gamma_p = \partial \O_p \cap \d\Omega$. Let $\Gamma_p = \Gamma_p^D \cup \Gamma_p^N$.
We assume for simplicity homogeneous boundary conditions:
$$
\u_f = 0 \mbox{ on } \Gamma_f \times (0,T], 
\quad \bbeta_p = 0 \mbox{ on } \Gamma_p\times (0,T], \quad
p_p = 0 \mbox{ on } \Gamma_p^D\times (0,T], \,\,
\u_p \cdot \n_p = 0 \mbox{ on } \Gamma_p^N\times (0,T].
$$
To avoid the issue with restricting the mean value of the pressure, we assume that 
$|\Gamma_p^D| > 0$. We also assume that $\Gamma_p^D$ is not adjacent to the 
interface $\Gamma_{fp}$, i.e., $\mbox{dist}(\Gamma_p^D,\Gamma_{fp}) \ge s > 0$.
Non-homogeneous displacement and velocity conditions can be handled in 
a standard way by adding suitable extensions of the boundary data. The pressure boundary 
condition is natural in the mixed Darcy formulation, so non-homogeneous pressure
data would lead to an additional boundary term. We further set the initial conditions
$$
p_p(\mathbf{x},0) = p_{p,0}(\mathbf{x}), \,\, 
\bbeta_p(\mathbf{x},0) = \bbeta_{p,0}(\mathbf{x}) \mbox{ in } \O_p.
$$
The solvability of the Stokes-Biot system
\eqref{stokes1}--\eqref{Gamma-fp-1} was discussed in
\cite{Showalter-Biot-Stokes}, see also \cite{Showalter-Deg-Evo}. In
the following we derive a Lagrange multiplier type weak formulation of
the system, which will be the basis for our finite element
approximation.  Let $(\cdot,\cdot)_S$, $S \subset \R^d$, be the
$L^2(S)$ inner product and let $\<\cdot,\cdot\>_F$, $F \subset
\R^{d-1}$, be the $L^2(F)$ inner product or duality pairing. We will
use the standard notation for Sobolev spaces, see, e.g.
\cite{ciarlet1978finite}. Let
\begin{align}
\V_{f} &= \{ \v_f \in H^1(\O_{f})^d : \v_f = 0 \text{ on }\Gamma_f\}, && W_{f} = L^2(\O_{f}), \nonumber
\\
\V_{p} &= \{ \v_p \in H({\rm div}; \O_{p}) : \v_p \cdot \n_p = 0 \text{ on }
\Gamma_p^N\},&& W_{p} = L^2(\O_{p}), \nonumber
\\
\X_{p} &= \{ \bxi_p \in H^1(\O_{p})^d : \bxi_p= 0 \text{ on } \Gamma_p \}, \label{spaces}
\end{align}
where $H({\rm div}; \O_{p})$ is the space of $L^2(\O_p)^d$-vectors with
divergence in $L^2(\O_p)$ with a norm
$$
\|\v\|_{H({\rm div}; \O_{p})}^2 = \|\v\|_{L^2(\Omega_p)}^2 + \|\div\v\|_{L^2(\Omega_p)}^2.
$$
We define the global velocity and pressure spaces as 
$$
\V = \{\v = (\v_{f},\v_{p}) \in \V_{f} \times \V_{p}\}, \quad
W = \{w = (w_{f},w_{p}) \in W_{f} \times W_{p}\},
$$
with norms 
$$
\|\v\|_{\V}^2 = \|\v_{f}\|_{H^1(\Omega_f)}^2 + \|\v_{p}\|_{H({\rm div}; \O_{p})}^2,
\quad \|w\|_W^2 = \|w_{f}\|_{L^2(\Omega_f)}^2 + \|w_{p}\|_{L^2(\Omega_p)}^2.
$$
The weak formulation is obtained by multiplying the equations in each
region by suitable test functions, integrating by parts the second
order terms in space, and utilizing the interface and boundary
conditions. Let
\begin{align*}
a_{f}(\u_{f},\v_{f}) &= (2\mu\D(\u_{f}),\D(\v_{f}))_{\O_{f}},
\\
a^d_{p}(\u_{p},\v_{p}) &= (\mu K^{-1}\u_{p},\v_{p})_{\O_{p}},
\\
a^e_{p}(\bbeta_p,\bxi_p) &= (2\mu_p\D(\bbeta_p),\D(\bxi_p))_{\O_{p}}+(\lambda_p\nabla\cdot\bbeta_p,\nabla\cdot\bxi_p)_{\O_{p}}
\end{align*}
be the bilinear forms related to Stokes, Darcy and the elasticity
operators, respectively.  Let 
$$
b_{\star}(\v,w) = -(\div \v,w)_{\O_\star}.
$$
Integration by parts in \eqref{stokes1} and the two equations in \eqref{eq:biot1} 
leads to the interface term
$$
I_{\Gamma_{fp}} = - \<\bs_f\n_f,\v_f\>_{\Gamma_{fp}} - \<\bs_p\n_p,\bxi_p\>_{\Gamma_{fp}}
+ \<p_p,\v_p\cdot\n_p\>_{\Gamma_{fp}}.
$$
Using the first condition for balance of normal stress in \eqref{balance-stress} we set
$$
\lam = -(\bs_f\n_f)\cdot\n_f = p_p \mbox{ on } \Gamma_{fp}, 
$$
which will be used as a Lagrange multiplier to impose the mass conservation 
interface condition \eqref{eq:mass-conservation}.
Utilizing the BJS condition \eqref{Gamma-fp-1} and the second
condition for balance of stresses in \eqref{balance-stress}, we obtain
$$
I_{\Gamma_{fp}} = a_{BJS}(\u_{f},\d_t\bbeta_{p};\v_{f},\bxi_{p}) + 
b_{\Gamma}(\v_{f},\v_{p},\bxi_{p};\lam),
$$
where
\begin{align*}
a_{BJS}(\u_f,\bbeta_p;\v_f,\bxi_p) &= 
\sum_{j=1}^{d-1} 
\big\<\mu\alpha_{BJS}\sqrt{K_j^{-1}}(\u_f - \bbeta_p)
\cdot\btau_{f,j},(\v_f - \bxi_p)\cdot\btau_{f,j} \big\>_{\Gam_{fp}}, \\
b_{\Gamma}(\v_f,\v_p,\bxi_p;\mu) &= \<\v_f\cdot\n_f 
+ (\bxi_p + \v_p)\cdot\n_p,\mu\>_{\Gam_{fp}}.
\end{align*}
For the well-posedness of $b_{\Gamma}$ we require that
$\lam \in \Lambda = (\V_p \cdot \n_p|_{\Gamma_{fp}})'$. According to 
the normal trace theorem, since
$\v_p \in \V_p \subset H({\rm div}; \O_{p})$, then $\v_p \cdot \n_p \in H^{-1/2}(\d\Omega_p)$.
Furthermore, since 
$\v_p \cdot \n_p = 0 \text{ on } \Gamma_p^N$ and 
$\mbox{dist}(\Gamma_p^D,\Gamma_{fp}) \ge s > 0$, then
$\v_p \cdot \n_p \in H^{-1/2}(\Gamma_{fp})$, see, e.g. \cite{Galvis-Sarkis}. 
Therefore we take $\Lambda = H^{1/2}(\Gamma_{fp})$.

The Lagrange multiplier variational formulation is: for $t \in (0,T]$, find
 $\u_f(t) \in \V_f$, $p_f(t) \in W_f$, 
$\u_p(t) \in \V_p$, $p_p(t) \in W_p$, 
$\bbeta_p(t) \in \X_p$, and $\lam(t) \in \Lambda$, such that 
$p_p(0) = p_{p,0}$, $\bbeta_p(0) = \bbeta_{p,0}$, 
and for all
$\v_f \in \V_f$, $w_f \in W_f$, $\v_p \in \V_p$, $w_p \in W_p$, 
$\bxi_p \in \X_p$, and $\mu \in \Lambda$,
\begin{align}
&  
a_{f}(\u_{f},\v_{f}) + a^d_{p}(\u_{p},\v_{p}) + a^e_{p}(\bbeta_{p},\bxi_{p})
+ a_{BJS}(\u_{f},\d_t \bbeta_{p};\v_{f},\bxi_{p}) + b_f(\v_{f},p_{f})+ b_p(\v_{p},p_{p})\nonumber
\\
& \qquad\quad   + 
\alpha b_p(\bxi_{p},p_{p}) + b_{\Gamma}(\v_{f},\v_{p},\bxi_{p};\lam)  = (\f_{f},\v_{f})_{\O_f} + (\f_{p},\bxi_{p})_{\O_p}, \label{h-cts-1}\\
& \left( s_0 \d_t p_{p},w_{p}\right)_{\O_p} 
- \alpha b_p\left(\d_t\bbeta_{p},w_{p}\right) - b_p(\u_{p},w_{p}) - b_f(\u_{f},w_{f}) \nonumber
\\
& \qquad\quad 
= (q_{f},w_{f})_{\O_f} + (q_{p},w_{p})_{\O_p}, \label{h-cts-2} \\
& b_{\Gamma}\left(\u_{f},\u_{p},\d_t\bbeta_{p};\mu\right) = 0. \label{h-cts-gamma}
\end{align}
where we used the notation $\d_t = \frac{\d}{\d t}$.
We note that the balance of normal stress, BJS, and conservation of
momentum interface conditions
\eqref{balance-stress}--\eqref{Gamma-fp-1} are natural and have been
utilized in the derivation of the weak formulation, while the
conservation of mass condition \eqref{eq:mass-conservation} is
essential and it is imposed weakly in \eqref{h-cts-gamma}.  The weak
formulation \eqref{h-cts-1}--\eqref{h-cts-gamma} is suitable for
multiscale numerical approximations and efficient parallel domain
decomposition algorithms
\cite{VasWangYot,APWY,ganis2009implementation,GirVasYot}.

%>>>>>>>>>>>>>>>>>>>>>>>>>>>>>>>>>>>>>>>>>>>>>>>>>>>>>>>>>>>>>>>>>>>>>>>>>>>>>>>>>>>>>>>>>>>>>>>>>>>>>>>>>>>>>>>>>>>>>>>>>>>>>>>>>

\section{Semi-discrete formulation}
Let $\mathcal{T}^f_h$ and $\mathcal{T}^p_h$ be shape-regular and
quasi-uniform partitions \cite{ciarlet1978finite} of $\O_f$ and
$\O_p$, respectively, both consisting of affine elements with maximal
element diameter $h$.  The two partitions may be non-matching at the
interface $\Gamma_{fp}$. For the discretization of the fluid velocity
and pressure we choose finite element spaces $\V_{f,h} \subset \V_f$
and $W_{f,h} \subset W_f$, which are assumed to be inf-sup stable. 
Examples of such spaces include the MINI elements,
the Taylor-Hood elements and the conforming Crouzeix-Raviart elements.
For the discretization of the porous medium problem we choose
$\V_{p,h} \subset \V_p$ and $W_{p,h}\subset W_p$ to be any of
well-known inf-sup stable mixed finite element spaces, such as the
Raviart-Thomas or the Brezzi-Douglas-Marini spaces. The reader is
referred to \cite{BoffiBrezziFortin} for an overview of stable Stokes
and Darcy mixed finite element spaces. The global spaces are
$$
\V_h = \{\v_h = (\v_{f,h},\v_{p,h}) \in \V_{f,h} \times \V_{p,h}\}, \quad
W_h = \{w_h = (w_{f,h},w_{p,h}) \in W_{f,h} \times W_{p,h}\}.
$$
We employ a conforming Lagrangian finite element space $\X_{p,h} \subset \X_h$ to approximate 
the structure displacement. Note that the finite element 
spaces $\V_{f,h}$, $\V_{p,h}$, and $\X_{p,h}$ satisfy the prescribed 
homogeneous boundary conditions on the external
boundaries.  For the discrete Lagrange multiplier space we take
$$
\Lambda_h = \V_{p,h}\cdot\n_p|_{\Gamma_{fp}}.
$$
The semi-discrete continuous-in-time problem reads: given $p_{p,h}(0)$ and $\bbeta_{p,h}(0)$,
for $t \in (0,T]$, 
find $\u_{f,h}(t) \in \V_{f,h}$, $p_{f,h}(t) \in W_{f,h}$,
$\u_{p,h}(t) \in \V_{p,h}$, $p_{p,h}(t) \in W_{p,h}$, $\bbeta_{p,h}(t) \in \X_{p,h}$, and $\lambda_h(t) \in \Lambda_h$
such that for all $\v_{f,h} \in \V_{f,h}$, $w_{f,h} \in W_{f,h}$,
$\v_{p,h} \in \V_{p,h}$, $w_{p,h} \in W_{p,h}$, $\bxi_{p,h} \in \X_{p,h}$, and $\mu_h \in \Lambda_h$,
\begin{align}
& a_{f}(\u_{f,h},\v_{f,h}) + a^d_{p}(\u_{p,h},\v_{p,h}) + a^e_{p}(\bbeta_{p,h},\bxi_{p,h})+ a_{BJS}(\u_{f,h},\d_t\bbeta_{p,h};\v_{f,h},\bxi_{p,h})  + b_f(\v_{f,h},p_{f,h})  \nonumber
	\\
&\qquad\quad
	 + b_p(\v_{p,h},p_{p,h})+ 
	\alpha b_p(\bxi_{p,h},p_{p,h})  
	+ b_{\Gamma}(\v_{f,h},\v_{p,h},\bxi_{p,h};\lam_h) = (\f_{f},\v_{f,h})_{\O_f} 
+ (\f_{p},\bxi_{p,h})_{\O_p}, \label{h-weak-1} \\
&	( s_0 \d_t p_{p,h},w_{p,h})_{\O_p} 
	- \alpha b_p(\d_t\bbeta_{p,h},w_{p,h}) 
- b_p(\u_{p,h},w_{p,h}) - b_f(\u_{f,h},w_{f,h}) \nonumber
	\\
	&\qquad\quad
= (q_{f},w_{f,h})_{\O_f} + (q_{p},w_{p,h})_{\O_p}, \label{h-weak-2} \\
&
	b_{\Gamma}(\u_{f,h},\u_{p,h},\d_t\bbeta_{p,h};\mu_h) = 0. \label{h-b-gamma}
\end{align}
We will take $p_{p,h}(0) $ and $\bbeta_{p,h}(0)$ to be suitable projections of 
the initial data $p_{p,0}$ and $\bbeta_{p,0}$.

The assumptions on the fluid viscosity $\mu$ and the material coefficients 
$K$, $\lambda_p$, and $\mu_p$ imply that the bilinear forms $a_f(\cdot,\cdot)$,
$a_p^d(\cdot,\cdot)$, and $a_p^e(\cdot,\cdot)$ are coercive and continuous in 
the appropriate norms. In particular, there exist positive constants 
$c^{f}$, $c^{p}$, $c^{e}$, $C^{f}$, $C^{p}$, $C^{e}$ such that 
\begin{align}
& c^f \|\v_f\|^2_{H^1(\O_f)} \leq a_f(\v_f,\v_f), \,\,\,
a_f(\v_f,\q_f) \leq C^f \|\v_f\|_{H^1(\O_f)}\|\q_f\|_{H^1(\O_f)}, && \hspace{-.2cm}
\forall \v_f, \q_f \in \V_{f}, \label{C1}\\
& c^p \|\v_p\|^2_{L^2(\O_p)} \leq a^d_p(\v_p,\v_p), \,\,\, 
a^d_p(\v_p,\q_p) \leq C^p \|\v_p\|_{L^2(\O_p)}\|\q_p\|_{L^2(\O_p)}, && \hspace{-.2cm}
\forall \v_p, \q_p \in \V_{p},  \label{C2}\\
& c^e \|\bxi_p\|^2_{H^1(\O_p)} \leq a^e_p(\bxi_p,\bxi_p), \,\,\, 
a^e_p(\bxi_p,\bzeta_p) \leq C^e \|\bxi_p\|_{H^1(\O_p)}\|\bzeta_p\|_{H^1(\O_p)}, && \hspace{-.2cm}
\forall \bxi_p, \bzeta_p \in \X_{p}, \label{C3}
\end{align}
where \eqref{C1} and \eqref{C3} hold true thanks to Poincare
inequality and \eqref{C3} also relies on Korn's inequality, see
\cite{ciarlet1978finite} or \cite{Ern-Guermond} for more details.
We further define, for $\v_f \in \V_f$, $\bxi_p \in \X_p$,
$$
|\v_f - \bxi_p|_{a_{BJS}}^2 = a_{BJS}(\v_{f},\bxi_{p};\v_{f},\bxi_{p}) 
= \sum_{j=1}^{d-1} \mu\alpha_{BJS}
\|K_j^{-1/4}(\v_{f} - \bxi_{p}) \cdot\btau_{f,j}\|^2_{L^2(\Gam_{fp})}.
$$
We next state a discrete inf-sup condition, which will be utilized to 
control the pressure in the two regions and the Lagrange
multiplier. Following \cite{Galvis-Sarkis}, we define a seminorm in 
$\Lambda_h$, 
\begin{equation}\label{norm-lh}
|\mu_h|_{\Lambda_h}^2 = a_p^d(\u_{p,h}^*(\mu_h),\u_{p,h}^*(\mu_h)),
\end{equation}
where $(\u_{p,h}^*(\mu_h),p_{p,h}^*(\mu_h)) \in \V_{p,h}\times W_{p,h}$ is the mixed 
finite element solution to the Darcy problem with Dirichlet data $\mu_h$ 
on $\Gamma_{fp}$:
\begin{align*}
& a_p^d(\u_{p,h}^*(\mu_h),\v_{p,h}) + b_p(\v_{p,h},p_h^*(\mu_h)) = 
- \<\v_{p,h}\cdot\n_p,\mu_h\>_{\Gamma_{fp}}, \quad \forall \, \v_{p,h} \in \V_{p,h},\\
& b_p(\u_{p,h}^*(\mu_h),w_{p,h}) = 0, \quad \forall \, w_{p,h}  \in W_{p,h}.
\end{align*}
We equip $\Lambda_h$ with the norm 
$\|\mu_h\|_{\Lambda_h}^2 = \|\mu_h\|_{L^2(\Gamma_{fp})}^2 + |\mu_h|_{\Lambda_h}^2$. 
This norm can be considered as a discrete version of the 
$H^{1/2}(\Gamma_{fp})$-norm \cite{Galvis-Sarkis}. For convenience of notation
we define the composite norms
$$
\|(\v_h,\bxi_{p,h})\|_{\V\times\X_p}^2 = \|\v_h\|_{\V}^2 + \|\bxi_{p,h}\|_{H^1(\Omega_p)}^2,
\quad
\|(w_h,\mu_h)\|_{W\times\Lambda_h}^2 = \|w_h\|_W^2 + \|\mu_h\|_{\Lambda_h}^2,
$$
as well as 
\begin{align*}
b(\v_h,\bxi_{p,h};w_h) & = b_f(\v_{f,h},w_{f,h})+ b_p(\v_{p,h},w_{p,h})
+ \alpha b_p(\bxi_{p,h},w_{p,h}), \\
b_{\Gamma}(\v_h,\bxi_{p,h};\mu_h) & = b_{\Gamma}(\v_{f,h},\v_{p,h},\bxi_{p,h};\mu_h).
\end{align*}
The next result establishes the 
Ladyzhenskaya-Babuska-Brezzi (LBB) condition for the mixed Stokes-Darcy problem,
where it is understood that the zero functions are excluded from the inf-sup.  
\begin{lemma}\label{GS-inf-sup}
	There exists a constant $\beta >0$ independent of $h$ such that
	\begin{align} 
	\inf_{(w_h,\mu_h)\in W_h\times\Lambda_h} \sup_{\v_h \in \V_h}
	\frac{b_f(\v_{f,h};w_{f,h})+b_p(\v_{p,h};w_{p,h}) 
+ \langle \v_{f,h}\cdot \n_f + \v_{p,h}\cdot \n_p,\mu_h\rangle}
	{\|\v_h\|_{\V} \|(w_h,\mu_h)\|_{W\times\Lambda_h}}
	\geq \beta. 
	\end{align}
\end{lemma}
\begin{proof}
The result is proven in \cite{Galvis-Sarkis} in the case of velocity boundary conditions 
on $\d\Omega$ by restricting the mean value of $W_h$. It can be easily verified that, since
$|\Gamma_p^D| > 0$, the result holds with no restriction on $W_h$.
\end{proof}
This result implies the inf-sup condition for the formulation 
\eqref{h-weak-1}-\eqref{h-b-gamma}.
\begin{corollary}
There exists a constant $\beta >0$ independent of $h$ such that
\begin{align} 
\inf_{(w_h,\mu_h)\in W_h\times\Lambda_h} \sup_{(\v_h,\bxi_{p,h}) \in \V_h\times \X_{p,h}}
\frac{b(\v_h,\bxi_{p,h};w_h) + b_{\Gamma}(\v_h,\bxi_{p,h};\mu_h)}
{\|(\v_h,\bxi_{p,h})\|_{\V\times\X_p} \|(w_h,\mu_h)\|_{W\times\Lambda_h}}
\geq \beta. \label{inf-sup}
\end{align}
\end{corollary}
\begin{proof}
The statement follows from Lemma~\ref{GS-inf-sup} by simply taking $\bxi_{p,h} = 0$.
\end{proof}

%>>>>>>>>>>>>>>>>>>>>>>>>>>>>>>>>>>>>>>>>>>>>>>>>>>>>>>>>>>>>>>>>>>>>>>>>>>>>>>>>>>>>>>>>>>>>>>>>>>>>>>>>>>>>>>>>>>>>>>>>>>>>>>>>>

\subsection{Existence and uniqueness of the solution}
In this section we show that the Stokes-Biot system is
well-posed. For the existence of the solution we adopt the theory of 
differential-algebraic equations (DAEs) \cite{brenan-dae}.

Let $\{\bphi_{\u_{f},i}\}, \{\bphi_{\u_{p},i}\},
\{\bphi_{\bbeta_{p},i}\}, \{\phi_{p_{f},i}\}, \{\phi_{p_{p},i}\}$ and
$\{\phi_{\lambda,i}\}$ be bases of $\V_{f,h}, \V_{p,h}, \X_{p,h},
W_{f,h}, W_{p,h}$ and $\Lambda_h$, respectively. Let $M_p, A_f, A_p,
A_e, B^T_{ff}, B^T_{pp}$ and $B^T_{ep}$ denote the matrices whose
$(i,j)$-entries are, respectively,
$(\phi_{p_{p},j}, \phi_{p_{p},i})_{\Omega_p}$,
$a_f(\bphi_{\u_{f},j},\bphi_{\u_{f},i})$, $a_p^d(\bphi_{\u_{p},j},\bphi_{\u_{p},i})$, 
$a_p^e(\bphi_{\bbeta_{p},j},\bphi_{\bbeta_{p},i})$, 
$b_f(\nabla\cdot \bphi_{\u_{f},j}, \phi_{p_{f},i})$, 
$b_p(\nabla\cdot \bphi_{\u_{p},j}, \phi_{p_{p},i})$, and 
$b_p(\nabla\cdot \bphi_{\bbeta_{p},j}, \phi_{p_{p},i})$.
We also introduce matrices $A_{ff}^{BJS}, A_{fe}^{BJS}$ and
$A_{ee}^{BJS}$ whose $(i,j)$-entries are, respectively,
$a_{BJS}(\bphi_{\u_{f},j},0;\bphi_{\u_{f},i},0)$, $a_{BJS}(\bphi_{\u_{f},j},0;0,\bphi_{\bbeta_{p},i})$,
and $a_{BJS}(0,\bphi_{\bbeta_{p},j};0,\bphi_{\bbeta_{p},j})$.
Finally, let $B_{f,\Gamma}^T,B_{p,\Gamma}^T$ and $B_{e,\Gamma}^T$ stand for the matrices 
with $(i,j)$-entries defined by
$b_{\Gamma}(\bphi_{\u_{f},j},0,0;\phi_{\lambda,i})$, 
$b_{\Gamma}(0,\bphi_{\u_{p},j},0;\phi_{\lambda,i})$, and 
$b_{\Gamma}(0,0,\bphi_{\bbeta_{p},j};\phi_{\lambda,i})$, respectively.

Taking in \eqref{h-weak-1}-\eqref{h-b-gamma} $\u_{f,h}(t,\x) = \sum_i
u_{f,i}(t)\phi_{\u_{f},i}$, $\u_{p,h}(t,\x) = \sum_i
u_{p,i}(t)\phi_{\u_{p},i}$, $\bbeta_{p,h}(t,\x) = \sum_i
\eta_{p,i}(t)\phi_{\bbeta_{p},i}$, $p_{f,h}(t,\x) = \sum_i
p_{f,i}(t)\phi_{p_{f},i}$, $p_{p,h}(t,\x) = \sum_i
p_{p,i}(t)\phi_{p_{p},i}$ and $\lambda_{h}(t,\x) = \sum_i
\lambda_{i}(t)\phi_{\lambda,i}$ with (time-dependent) coefficients
$\overline\u_f, \overline\u_p, \overline\bbeta_p, \overline p_f,
\overline p_p, \overline\lambda$, leads to the matrix-vector system
\begin{align} 
&  A_{f}\,\overline\u_f + A_{p}\,\overline\u_p + A_{e}\,\overline\bbeta_p + A^{BJS}_{ff}\,
\overline\u_f+ A^{BJS}_{fe}\,\d_t\overline\bbeta_p+ B^T_{ff}\,\overline p_f 
	   + \left(B^T_{pp} + \alpha B^T_{ep}\right)\overline p_p \nonumber \\
&\qquad\qquad\qquad\qquad\qquad
+\left(B^T_{f,\Gamma} + B^T_{p,\Gamma} + B^T_{e,\Gamma}\right)\overline\lam = \mathcal{F}_{\u_f} + \mathcal{F}_{\u_p} \label{dae1} \\
& M_{p}\,\d_t \overline p_p - \alpha B_{ep}\,\d_t\overline\bbeta_p
- B_{pp}\,\overline\u_p - B_{ff}\,\overline\u_f  + A^{BJS,T}_{fe}\,\overline\u_f + A^{BJS}_{ee}\,\d_t\overline\bbeta_p  = \mathcal{F}_{p_f} + \mathcal{F}_{p_p},  \label{dae2}\\
& B_{\f,\Gamma}\overline\u_f + B_{p,\Gamma}\overline\u_p + B_{e,\Gamma}\d_t\overline\bbeta_p = 0, \label{dae3}
\end{align}
which can be written in the DAE system form
\begin{align}
    \E\,\d_t X(t) + \H\,X(t) = L(t) \label{DAE},
\end{align}
where
\begingroup
\renewcommand*{\arraystretch}{1.5}
\begin{align}\label{XLE}
        X(t) &= 
        \begin{pmatrix}
            \overline\u_f(t) \\
            \overline\u_p(t) \\
            \overline\bbeta_p(t) \\
            \overline p_f(t) \\
            \overline p_p(t) \\
            \overline \lam(t) 
        \end{pmatrix},
        \quad
        L(t) = 
        \begin{pmatrix}
            \mathcal{F}_{\u_f} \\
            0 \\
            \mathcal{F}_{\bbeta_p} \\
            \mathcal{F}_{p_f} \\
            \mathcal{F}_{p_p} \\
            0
        \end{pmatrix},
        \quad
    \E = 
        \begin{pmatrix}
        	0 & 0 & A^{BJS}_{fe}  & 0 & 0       & 0 \\
        	0   & 0 & 0             & 0 & 0       & 0 \\
        	0   & 0 & A^{BJS}_{ee}  & 0 & 0       & 0 \\
        	0   & 0 & 0             & 0 & 0       & 0 \\
        	0   & 0 & -\alpha B_{ep}& 0 & s_0 M_p & 0 \\
        	0   & 0 & -B_{e,\Gamma} & 0 & 0       & 0
        \end{pmatrix},
\end{align}
\begin{align}     
     \H &=
     \begin{pmatrix}
     	A_f + A^{BJS}_{ff} & 0             & 0   & B^T_{ff}  & 0                & B_{f,\Gamma}^T \\
     	0                  & A_p           & 0   & 0         & B^T_{pp}         & B_{p,\Gamma}^T \\
     	A^{BJS,T}_{fe}     & 0             & A_e & 0         & \alpha B^T_{ep}  & B_{e,\Gamma}^T \\
     	-B_{ff}            & 0             & 0   & 0         & 0                & 0              \\
     	0                  & -B_{pp}       & 0   & 0         & 0                & 0              \\
     	-B_{f,\Gamma}      & -B_{p,\Gamma} & 0   & 0         & 0                & 0
     \end{pmatrix}.
\end{align}
We note that the matrix 
\begin{align*}
     \E+\H = \begin{pmatrix}
     	 A_f + A^{BJS}_{ff} & 0             & A^{BJS}_{fe}     & B^T_{ff} & 0                & B_{f,\Gamma}^T \\
     	0                        & A_p           & 0                & 0         & B^T_{pp}        & B_{p,\Gamma}^T \\
     	A^{BJS,T}_{fe}           & 0             & A_e+A^{BJS}_{ee} & 0         & \alpha B^T_{ep} & B_{e,\Gamma}^T \\
     	-B_{ff}                   & 0             & 0                & 0         & 0                & 0              \\
     	0                        & -B_{pp}        & -\alpha B_{ep}    & 0         & s_0 M_p          & 0              \\
     	-B_{f,\Gamma}            & -B_{p,\Gamma} & -B_{e,\Gamma}    & 0         & 0                & 0
     \end{pmatrix}
\end{align*}
can be written as a block $2\times2$ matrix
\begin{align*}
\E + \H &=
    \begin{pmatrix}
        \A  & \B^T \\
        -\B & \C
    \end{pmatrix},
\end{align*}
where
\begin{align*}
    &\A = \begin{pmatrix} A_f + A_{ff}^{BJS} & 0 & A_{fe}^{BJS} \\ 0 & A_p & 0 \\ A_{fe}^{BJS,T} & 0 & A_e + A_{ee}^{BJS} \end{pmatrix}, \,
    \B^T = \begin{pmatrix} B_{ff}^T & 0 & B_{f,\Gamma}^T \\ 0 & B_{pp}^T & B_{p,\Gamma}^T  \\ 0 & \alpha B^T_{ep} & B_{e,\Gamma}^T  \end{pmatrix}, \,
    \C = \begin{pmatrix} 0 & 0 & 0 \\ 0 & s_0 M_p & 0 \\ 0 & 0 & 0 \end{pmatrix}.
\end{align*}

The following result can be found in \cite{yi-biot}.
\begin{lemma} \label{nonsingDAE1}
    If $\A$ and $\C$ are positive semi-definite and 
$\ker(\A)\cap \ker(\B) = \ker(\C)\cap \ker(\B^T) = \{0\}$, then $\E+\H$ is invertible.
\end{lemma}
It is convenient to associate with matrices $\A$, $\B$, and $\C$ the bilinear forms
$\phi_{\A}(\cdot,\cdot),\, \phi_{\B}(\cdot,\cdot)$ and
$\phi_{\C}(\cdot,\cdot)$ on $\left(\V_h\times\X_h\right) \times
\left(\V_h\times\X_h\right)$,
$\left(\V_h\times\X_h\right)\times\left(W_h\times\Lambda_h\right)$
and
$\left(W_h\times\Lambda_h\right)\times\left(W_h\times\Lambda_h\right)$, respectively:
\begin{align*}
\phi_{\A}((\u_{h},\bbeta_{p,h}),(\v_{h},\bxi_{p,h})) &= a_f(\u_{f,h},\v_{f,h}) + a_p^d(\u_{p,h},\v_{p,h}) + a_p^e(\bbeta_{p,h},\bxi_{p,h}) \nonumber\\
&+ a_{BJS}(\u_{f,h},\bbeta_{p,h};\v_{v,h},\bxi_{p,h}) \\
\phi_{\B}((\u_{h},\bbeta_{p,h}),(w_{h},\mu_h)) &= b_f(\u_{f,h},w_{f,h}) + b_p(\u_{p,h},w_{p,h}) \nonumber\\
&+ \alpha b_p(\bbeta_{p,h},w_{p,h}) + b_{\Gamma}(\u_{f,h},\u_{p,h},\bbeta_{p,h};\mu_h) \\
\phi_{\C}((p_{h},\lam_h),(w_{h},\mu_h)) &= (s_0p_{p,h},w_{p,h})_{\Omega_p}.
\end{align*}
By identifying functions in the finite element spaces with algebraic vectors of their degrees 
of freedom, we note that $\ker(\phi_{\A}) = \ker(\A)$, $\ker(\phi_{\B}) = \ker(\B)$, and
$\ker(\phi_{\C}) = \ker(\C)$. Also, for $\phi_{\B^T}((w_h,\mu_h),(\v_{h},\bxi_{p,h})) = 
\phi_{\B}((\v_{h},\bxi_{p,h}),(w_{h},\mu_h))$, we have that $\ker(\phi_{\B^T}) = \ker(\B^T)$.
We next show that the conditions of the Lemma~\ref{nonsingDAE1} are satisfied.
\begin{lemma}\label{nonsingDAE2}
The bilinear forms $\phi_{\A},\,\phi_{\B}$ and $\phi_{\C}$ satisfy
\begin{align*}
   \ker(\phi_{\A})\cap \ker(\phi_{\B}) &= \{(0,0)\}, \\
   \ker(\phi_{\C})\cap \ker(\phi_{\B^T}) &= \{(0,0)\}.
\end{align*}
Moreover, $\phi_{\A}$ and $\phi_{\C}$ are positive definite and semi-definite, respectively.
\end{lemma}
\begin{proof}
The coercivity of $a_f(\cdot,\cdot)$, $a_p^d(\cdot,\cdot)$, and $a_p^e(\cdot,\cdot)$,
\eqref{C1}--\eqref{C3}, and the non-negativity of $a_{BJS}(\cdot,\cdot)$ imply that 
$\phi_{\A}(\cdot,\cdot)$ is coercive and $\ker(\phi_{\A}) = 0$, hence the 
first statement of the lemma follows. We next note that 
$\ker(\phi_{\B^T})$ consists of $(w_{h},\mu_h)\in W_h\times\Lambda_h$ such that 
$$
\phi_{\B^T}((w_{h},\mu_h),(\v_{h},\bxi_{p,h})) = 0,\quad 
\forall \, (\v_{h},\bxi_{p,h})\in \V_h\times\X_{p,h},
$$
therefore the inf-sup condition (\ref{inf-sup}) implies that $\ker(\phi_{\B^T}) = \{(0,0)\}$,
which gives the second statement of the lemma. The positive semi-definiteness of 
$\phi_{\C}(\cdot,\cdot)$ is straightforward.
\end{proof}
To state the desired result, we will first introduce Bochner spaces equipped with norms:
\begin{align}
\|\phi\|_{L^2(0,T;X)}:=\left(\int_0^T \|\phi(t)\|^2_{X}\, ds\right)^{1/2},\qquad \|\phi\|_{L^{\infty}(0,T;X)}:= \text{ess sup}_{t\in [0,T]}\|\phi(t)\|_X \nonumber \\
\|\phi\|_{W^{1,\infty}(0,T;X)}:=\text{ess sup}_{t\in [0,T]}\{ \|\phi(t)\|_X, \|\d_t\phi(t)\|_X \}. 
\label{Bochner-norms}
\end{align}
\begin{theorem}\label{well-posed}
	There exists a unique solution $(\u_{f,h}, p_{f,h}, \u_{p,h},
        p_{p,h}, \bbeta_{p,h}, \lam_h)$ in \\ $L^{\infty}(0,T;
        \V_{f,h}) \times$ $ L^{\infty}(0,T; W_{f,h}) \times
        L^{\infty}(0,T; \V_{p,h}) \times W^{1,\infty}(0,T; W_{p,h})
        \times W^{1,\infty}(0,T;\X_{p,h}) \times
        L^{\infty}(0,T;\Lambda_h) $ of the weak formulation
        \eqref{h-weak-1}-\eqref{h-b-gamma}.
\end{theorem}
\begin{proof}
According to the DAE theory, see Theorem 2.3.1 in \cite{brenan-dae},
if the matrix pencil $s\E + \H$ is nonsingular for some $s\neq 0$ and
the initial data is consistent, then \eqref{DAE} has a solution.
Lemma \ref{nonsingDAE2} guarantees that in our case the pencil with
$s=1$ is invertible. Also, the initial data $p_{p,h}(0)$ and
$\bbeta_{p,h}(0)$ does not lead to consistency issues. In particular,
the only algebraic constraints in the DAE system \eqref{DAE} are the
second and fourth equations, see the definition of $\E$ in
\eqref{XLE}. The second equation is the discretized Darcy's law, and
the initial value $\u_{p,h}(0)$ can be chosen to satisfy it for any
given $p_{p,h}(0)$, while the fourth equation is the discretized
incompressibility constraint for Stokes, which does not involve the
initial data. Furthermore, the initial data can be assumed to satify 
the boundary conditions. As a result, Theorem 2.3.1 in \cite{brenan-dae}
implies existence of a solution of the weak semi-discrete 
formulation \eqref{h-weak-1}-\eqref{h-b-gamma}.

To show uniqueness, we assume that there are two solutions satisfying
these equations with the same initial conditions. Then their
difference
$(\tilde{\u}_{f,h},\tilde{p}_{f,h},\tilde{\u}_{p,h},\tilde{p}_{p,h},\tilde{\bbeta}_{p,h},\tilde{\lam}_h)$
satisfies \eqref{h-weak-1}-\eqref{h-b-gamma} with zero data.  By
taking $(\v_{f,h},w_{f,h},\v_{p,h},w_{p,h},\bxi_{p,h},\mu_h) =
(\tilde{\u}_{f,h},\tilde{p}_{f,h},\tilde{\u}_{p,h},\tilde{p}_{p,h},\d_t\tilde{\bbeta}_{p,h},\tilde{\lam}_h)$
in \eqref{h-weak-1}-\eqref{h-b-gamma}, we obtain the energy equality
\begin{align*}
 a_{f}(\tilde{\u}_{f,h},\tilde{\u}_{f,h}) + a^d_{p}(\tilde{\u}_{p,h},\tilde{\u}_{p,h}) + a^e_{p}\left(\tilde{\bbeta}_{p,h},\d_t \tilde{\bbeta}_{p,h}\right)+ \left(s_0\d_t \tilde{p}_{p,h},\tilde{p}_{p,h}\right) + \left|\tilde{\u}_{f,h} - \d_t \tilde{\bbeta}_{p,h} \right|^2_{a_{BJS}} 
= 0
\end{align*}
Using the algebraic identity
\begin{align}
\int_{S}\phi \frac{\d \phi}{\d t} = \frac12 \frac{\d }{\d t}\|\phi\|^2_{L^2(S)}  
\label{int-parts}
\end{align}
we write the energy equality as
\begin{align*}
\frac12 \d_t\left( s_0 \|\tilde{p}_{p,h}\|^2_{L^2(\O_p)} + a^e_{p}(\tilde{\bbeta}_{p,h},\tilde{\bbeta}_{p,h})\right)
+ a_{f}(\tilde{\u}_{f,h},\tilde{\u}_{f,h}) + a^d_{p}(\tilde{\u}_{p,h},\tilde{\u}_{p,h}) + \left|\tilde{\u}_{f,h} - \d_t\tilde{\bbeta}_{p,h} \right|^2_{a_{BJS}} = 0
\end{align*}
Integrating in time over $[0,t]$ for arbitrary $t \in (0,T]$, we obtain 
\begin{align}
&\frac12 \left( s_0 \|\tilde{p}_{p,h}(t)\|^2_{L^2(\O_p)} 
+ a^e_{p}(\tilde{\bbeta}_{p,h}(t),\tilde{\bbeta}_{p,h}(t))\right) \nonumber\\
& \qquad\qquad +
\int_0^t \left[\left|\tilde{\u}_{f,h} - \d_t\tilde{\bbeta}_{p,h}\right|^2_{a_{BJS}}   +
a_{f}(\tilde{\u}_{f,h},\tilde{\u}_{f,h}) + a^d_{p}(\tilde{\u}_{p,h},\tilde{\u}_{p,h})\right]\, ds  = 0. 
\label{diff-energy-eq}
\end{align}
Due to the coercivity of bilinear forms, we conclude that
$\tilde{\u}_{f,h}(t)=0,\, \tilde{\u}_{p,h}(t)=0,\,\tilde{\bbeta}_{p,h}(t) =0,\, \forall
t\in [0,T]. $ If $s_0 \neq 0$, we also have that $\tilde{p}_{p,h}(t)=0$, but
we can also obtain uniqueness for both pressure variables and
the Lagrange multiplier simultaneously and independently of parameters. In particular,
from the inf-sup condition \eqref{inf-sup} and \eqref{h-weak-1}, 
we have for $(\tilde{p}_h,\tilde{\lam}_h)$
\begin{align*}
\beta\|(\tilde{p}_{h},\tilde{\lam}_h)\|_{W\times \Lambda_h}\leq \sup_{(\v_h,\bxi_{p,h}) \in \V_h\times \X_{p,h}} \frac{b_f(\v_{f,h},\tilde{p}_{f,h}) + b_p(\v_{p,h},\tilde{p}_{p,h}) + 
	\alpha b_p(\bxi_{p,h},\tilde{p}_{p,h})+b_{\Gamma}(\v_{f,h},\v_{p,h},\bxi_{p,h};\tilde{\lam}_h)} {\|(\v_h,\bxi_h)\|_{\V\times \X_p}} \\
 = \sup_{(\v_h,\bxi_{p,h}) \in \V_h\times \X_{p,h}}\bigg[ \frac{-a_f(\tilde{\u}_{f,h},\v_{f,h})-a_p^d(\tilde{\u}_{p,h},\v_{p,h})-a_p^e(\tilde{\bbeta}_{p,h},\bxi_{p,h})-a_{BJS}(\tilde{\u}_{f,h},\d_t\tilde{\bbeta}_{p,h};\v_{f,h},\bxi_{p,h})} {\|(\v_h,\bxi_h)\|_{\V\times \X_p}}  \bigg] = 0.
\end{align*}
Therefore, we conclude that $\tilde{p}_{f,h}(t)=0,\,\tilde{p}_{p,h}(t)=0,\,\tilde{\lam}_h(t)=0,\,\forall t\in (0,T]$ 
and the solution of \eqref{h-weak-1}-\eqref{h-b-gamma} is unique.
\end{proof}
The next two sections are devoted to the stability and error analysis
of the semi-discrete problem.

\section{Stability analysis of the semi-discrete formulation}
We will make use of the following well-known inequalities:

$\bullet$ (Cauchy-Schwarz) For any $u,v \in L^2(S)$,
\begin{equation}
(u,v)_S \leq \|u\|_{L^2(S)}\|v\|_{L^2(S)},
 \label{CS}
\end{equation}

$\bullet$ (Trace) For any $v \in H^1(S)$,
\begin{align}
& \|v\|_{L^2(\partial S)} \leq C\|v\|_{H^1(S)},
 \label{Trace}
\end{align}

$\bullet$ (Young's) For any real numbers $a,b$ and $\epsilon >0$,
\begin{align}
& ab \leq \frac{\epsilon a^2}{2} + \frac{b^2}{2\epsilon}, \label{Young}
\end{align}

$\bullet$ (Gronwall's) Let $g(t) \geq 0$ and
$ u(t) \leq f(t) + \int_{s}^{t}g(\tau)u(\tau) d\tau $, then
\begin{align}
&  u(t) \leq f(t) 
+  \int_{s}^{t}f(\tau)g(\tau)\exp\left(\int_{\tau}^{t}g(r)dr\right) d\tau. \label{Gronwall}
\end{align}
For the sake of simplicity, throughout the analysis, $C$ will denote a
generic positive constant independent of the mesh size. We will also
abuse notation by denoting $\epsilon$ as an arbitrary constant
with different values at different occurrences, arising from the usage
of inequality \eqref{Young}.

By taking $(\v_{f,h},w_{f,h},\v_{p,h},w_{p,h},\bxi_{p,h},\mu_h) 
= \left(\u_{f,h},p_{f,h},\u_{p,h},p_{p,h},\d_t \bbeta_{p,h},\lam_h\right)$
in \eqref{h-weak-1}--\eqref{h-b-gamma} and proceeding as in the uniqueness proof,
Theorem~\ref{well-posed}, we obtain
\begin{align}
&\frac12 \left(s_0 \|p_{p,h}(t)\|^2_{L^2(\O_p)} + a^e_{p}(\bbeta_{p,h}(t),\bbeta_{p,h}(t))\right) 
\nonumber\\
&\quad+
\int_0^t \left[ \left|\u_{f,h} - \d_t\bbeta_{p,h}\right|^2_{a_{BJS}} 
+ a_{f}(\u_{f,h},\u_{f,h}) + a^d_{p}(\u_{p,h},\u_{p,h})\right]\, ds   \nonumber\\
&\quad = \frac12 \left(s_0 \|p_{p,h}(0)(t)\|^2_{L^2(\O_p)} 
+ a^e_{p}(\bbeta_{p,h}(0),\bbeta_{p,h}(0))\right)
+\int_0^t \mathcal{F}\left(t;\u_{f,h},\d_t \bbeta_{p,h},p_{f,h},p_{p,h}\right)\, ds,  
\label{energy-eq}
\end{align}
where $\mathcal{F}\left(t;\u_{f,h},\d_t \bbeta_{p,h},p_{f,h},p_{p,h}\right)$ 
denotes the total forcing term:
\begin{align*}
\mathcal{F}\left(t;\u_{f,h},\d_t \bbeta_{p,h},p_{f,h},p_{p,h}\right) &= (\f_{f},\u_{f,h})_{\O_f} + \left(\f_{p},\d_t \bbeta_{p,h}\right)_{\O_p} + (q_{f},p_{f,h})_{\O_f} + (q_{p},p_{p,h})_{\O_p} 
\end{align*}
Using integration by parts in time, we write the forcing term as
\begin{align*}
 \mathcal{F}\left(t;\u_{f,h},\d_t \bbeta_{p,h},p_{f,h},p_{p,h}\right) 
=  (\f_{f},\u_{f,h})_{\O_f} + \d_t\left(\f_{p}, \bbeta_{p,h}\right)_{\O_p} 
- \left(\d_t \f_{p}, \bbeta_{p,h}\right)_{\O_p} 
+ (q_{f},p_{f,h})_{\O_f} + (q_{p},p_{p,h})_{\O_p}.
\end{align*}
Therefore, for any $\epsilon_1 >0$, we have
\begin{align}
& \int_0^t \mathcal{F}\left(t;\u_{f,h},\d_t \bbeta_{p,h},p_{f,h},p_{p,h}\right)ds
\nonumber \\
& \qquad \leq \frac{1}{2}
\|\bbeta_{p,h}(0)\|^2_{L^2(\O_p)} + \frac12\|\f_{p}(0)\|^2_{L^2(\O_p)}
+ \frac12\int_0^t \left(\|\bbeta_{p,h}\|^2_{L^2(\O_p)}
+ \left\|\d_t \f_{p}\right\|^2_{L^2(\O_p)} 
\right) ds 
\nonumber\\
&\qquad\qquad  + \frac{\epsilon_1}{2}\bigg(\|\bbeta_{p,h}(t)\|^2_{L^2(\O_p)} 
+   \int_0^t \left( \|\u_{f,h}\|^2_{L^2(\O_f)} 
+ \|p_{f,h}\|^2_{L^2(\O_f)}+\|p_{p,h}\|^2_{L^2(\O_p)}\right) ds\bigg) \nonumber\\
&\qquad\qquad + \frac{1}{2\epsilon_1}\left(\left\|\f_{p}(t)\right\|^2_{L^2(\O_p)} + \int_0^t \left( \|\f_{f}\|^2_{L^2(\O_f)}
+ \|q_{f}\|^2_{L^2(\O_f)}+\|q_{p}\|^2_{L^2(\O_p)}\right) ds\right). 
\label{forcing-term}
\end{align}
Combining \eqref{energy-eq}, \eqref{forcing-term} and \eqref{C1}--\eqref{C3}, 
and taking $\epsilon_1$ small enough, we obtain
\begin{align}
& s_0 \|p_{p,h}(t)\|^2_{L^2(\O_p)} + \|\bbeta_{p,h}(t)\|^2_{H^1(\O_p)} +
\int_0^t \left(\left|\u_{f,h} - \d_t\bbeta_{p,h} \right|^2_{a_{BJS}}  
+ \|\u_{f,h}|^2_{H^1(\O_f)}  
+ \|\u_{p,h}\|^2_{L^2(\O_p)} \right)\, ds \nonumber \\
&\leq C \epsilon_1\int_0^t \left( 
\|p_{f,h}\|^2_{L^2(\O_f)}+\|p_{p,h}\|^2_{L^2(\O_p)}\right) ds 
+ C \int_0^t \|\bbeta_{p,h}\|^2_{L^2(\O_p)} ds
\nonumber \\
&\qquad +C \left(s_0 \|p_{p,h}(0)\|^2_{L^2(\O_p)} 
+ \|\bbeta_{p,h}(0)\|^2_{H^1(\O_p)} + \|\f_{p}(0)\|^2_{L^2(\O_p)}
+ \int_0^t \left\| \d_t \f_{p}\right\|^2_{L^2(\O_p)} ds \right)\nonumber\\
& \qquad + C \epsilon_1^{-1}\left( \left\|\f_{p}(t)\right\|^2_{L^2(\O_p)} 
+ \int_0^t \left( \|\f_{f}\|^2_{L^2(\O_f)}
+ \|q_{f}\|^2_{L^2(\O_f)}+\|q_{p}\|^2_{L^2(\O_p)}\right) ds\right). 
\label{energy-ineq1}
\end{align}
Finally, from the inf-sup condition \eqref{inf-sup} and \eqref{h-weak-1}, 
we have 
\begin{align*}
&\|(p_{h},\lam_h)\|_{W\times \Lambda_h} \\
&\qquad \leq C\sup_{(\v_h,\bxi_{p,h}) \in \V_h\times \X_{p,h}} \frac{b_f(\v_{f,h},p_{f,h}) + b_p(\v_{p,h},p_{p,h}) + 
	\alpha b_p(\bxi_{p,h},p_{p,h})+b_{\Gamma}(\v_{f,h},\v_{p,h},\bxi_{p,h};\lam_h)} {\|(\v_h,\bxi_{p,h})\|_{\V\times \X_p}} \\
	& \qquad = C\sup_{(\v_h,\bxi_{p,h}) \in \V_h\times \X_{p,h}}\left[ 
\frac{-a_f(\u_{f,h},\v_{f,h})-a_p^d(\u_{p,h},\v_{p,h})
-a_p^e(\bbeta_{p,h},\bxi_{p,h})} {\|(\v_h,\bxi_{p,h})\|_{\V\times \X_p}} \right. \\
	& \left. \qquad\qquad+\frac{-a_{BJS}(\u_{f,h},\d_t\bbeta_{p,h};\v_{f,h},\bxi_{p,h})+(\f_{f},\v_{f,h})+(\f_{p},\bxi_{p,h})} {\|(\v_h,\bxi_{p,h})\|_{\V\times \X_p}} \right],
\end{align*}
which, combined with \eqref{C1}--\eqref{C3}, gives
\begin{align}
& \epsilon_2 \int_0^t \left(  \|p_{f,h}\|^2_{L^2(\O_f)} +\|p_{p,h}\|^2_{L^2(\O_p)}
+ \|\lambda\|_{\Lambda_h}^2 \right)\, ds 
\leq C \epsilon_2\int_0^t \left(  \left\|  \u_{f,h}\right\|^2_{H^1(\O_f)}+\left\|  \u_{p,h}\right\|^2_{L^2(\O_p)}+\left\|  \bbeta_{p,h}\right\|^2_{H^1(\O_p)} \right. \nonumber \\
&  \qquad\qquad\qquad\qquad \left.
+ |\u_{f,h}-\d_t\bbeta_{p,h}|^2_{a_{BJS}}
+ \|\f_{f}\|^2_{L^2(\O_f)}
+\|\f_{p}\|^2_{L^2(\O_p)}\right)\, ds. \label{p-bound}
\end{align}
Adding \eqref{energy-ineq1} and \eqref{p-bound} and taking $\epsilon_2$ 
small enough, and then $\epsilon_1$ small enough, implies
\begin{align}
& s_0 \|p_{p,h}(t)\|^2_{L^2(\O_p)} + \|\bbeta_{p,h}(t)\|^2_{H^1(\O_p)}+
\int_0^t \bigg(\left|\u_{f,h} - \d_t\bbeta_{p,h}\right|^2_{a_{BJS}} \nonumber\\
& \quad + \|\lam_h\|^2_{\Lambda_h}+\|p_{f,h}\|^2_{L^2(\O_f)}
+\|p_{p,h}\|^2_{L^2(\O_p)}  + \|\u_f\|^2_{H^1(\O_f)}  + \|\u_p\|^2_{L^2(\O_p)} \bigg)\, ds 
\nonumber \\
&\leq C\left(  \int_0^t  \|\bbeta_{p,h}\|^2_{H^1(\O_p)}\, ds 
+ s_0 \|p_{p,h}(0)\|^2_{L^2(\O_p)} + \|\bbeta_{p,h}(0)\|^2_{H^1(\O_p)}
+ \|\f_{p}(0)\|^2_{L^2(\O_p)} \right. \nonumber \\
& \left. \quad + \int_0^t \left( \|\f_{f}\|^2_{L^2(\O_f)}+ \|\f_{p}\|^2_{L^2(\O_p)}+\left\| \d_t \f_{p}\right\|^2_{L^2(\O_p)} + \|q_{f}\|^2_{L^2(\O_f)}+\|q_{p}\|^2_{L^2(\O_p)}\right)ds \right).
\label{energy-ineq4}
\end{align}
The use of Gronwall's inequality \eqref{Gronwall} implies the following stability result.
\begin{theorem}
The solution of the semi-discrete problem \eqref{h-weak-1}--\eqref{h-b-gamma} satisfies
\begin{align}\label{stability-semi-discrete}
& \sqrt{s_0} \|p_{p,h}\|_{L^{\infty}(0,T;L^2(\O_p))} + \|\bbeta_{p,h}\|_{L^{\infty}(0,T;H^1(\O_p))} + \|\u_f\|_{L^2(0,T;H^1(\O_f))}   + \|\u_p\|_{L^2(0,T;L^2(\O_p))} \nonumber\\
&\quad+\|p_{f,h}\|_{L^2(0,T;L^2(\O_f))}+\|p_{p,h}\|_{L^2(0,T;L^2(\O_p))} +\|\lam_h\|_{L^2(0,T;\Lambda_h)}+
 \left|\u_{f,h} - \d_t\bbeta_{p,h}\right|_{L^2(0,T;a_{BJS})} \nonumber\\
&\leq C\sqrt{\exp(T)}\left( \sqrt{s_0}\|p_{p,h}(0)\|_{L^2(\O_p)} +\|\bbeta_{p,h}(0)\|_{H^2(\O_p)}+ \left\|\f_{p}\right\|_{L^{\infty}(0,T;L^2(\O_p))} +\left\|\f_{p}\right\|_{L^{2}(0,T;L^2(\O_p))} \right. 
\nonumber\\
&\quad \left. 
 + \|\f_{f}\|_{L^{2}(0,T;L^2(\O_f))}+ \left\| \d_t \f_{p}\right\|_{L^{2}(0,T;L^2(\O_p))}+ \|q_{f}\|_{L^{2}(0,T;L^2(\O_f))} +\|q_{p}\|_{L^{2}(0,T;L^2(\O_p))} \right). 
\end{align}
\end{theorem}

%>>>>>>>>>>>>>>>>>>>>>>>>>>>>>>>>>>>>>>>>>>>>>>>>>>>>>>>>>>>>>>>>>>>>>>>>>>>>>>>>>>>>>>>>>>>>>>>>>>>>>>>>>>>>>>>>>>>>>>>>>>>>>>>>>

\section{Error analysis}
In this section, we analyze the error arising due to discretization in
space. We denote by $k_f$ and $s_f$ the degrees of polynomials in
the spaces $\V_{f,h}$ and $W_{f,h}$ respectively. Let $k_p$ and $s_p$ be the
degrees of polynomials in the spaces $\V_{p,h}$ and $W_{p,h}$
respectively. Finally, let $k_s$ be the polynomial degree in $\X_{p,h}$.

%>>>>>>>>>>>>>>>>>>>>>>>>>>>>>>>>>>>>>>>>>>>>>>>>>>>>>>>>>>>>>>>>>>>>>>>>>>>>>>>>>>>>>>>>>>>>>>>>>>>>>>>>>>>>>>>>>>>>>>>>>>>>>>>>>

\subsection{Approximation error}
Let $Q_{f,h}$, $Q_{p,h}$, and $Q_{\lambda,h}$ be the $L^2$-projection operators onto $W_{f,h}$,
$W_{p,h}$, and $\Lambda_h$ respectively, satisfying:
\begin{align}
& (p_{f}-Q_{f,h}p_{f},w_{f,h})_{\Omega_f}=0,&& \forall \, w_{f,h} \in W_{f,h} 
\label{fluid-pressure-int}\\
& (p_{p}-Q_{p,h}p_{p},w_{p,h})_{\Omega_p}=0,&& \forall \, w_{p,h} \in W_{p,h} 
\label{darcy-pressure-int}\\
&  \langle\lambda-Q_{\lambda,h}\lambda,\mu_h\rangle_{\Gamma_{fp}}=0,&& 
\forall \, \mu_h \in \Lambda_h \label{l-multuplier-int}
\end{align}
These operators satisfy the approximation properties \cite{ciarlet1978finite}:
 	\begin{eqnarray}
 	&& \|p_f - Q_{f,h}p_f\|_{L^2(\O_f )} \leq Ch^{s_f+1}\|p_f\|_{H^{s_f+1}(\O_f)}, \label{stokesPresProj}\\
 	&& \|p_p - Q_{p,h}p_p\|_{L^2(\O_p )} \leq Ch^{s_p+1}\|p_p\|_{H^{s_p+1}(\O_p)}, \label{darcyPresProj}\\
 	&&  \|\lambda - Q_{\lambda,h} \lambda\|_{L^2(\Gamma_{fp} )} \leq Ch^{k_p+1} \|\lambda\|_{H^{k_p+1}(\Gamma_{fp})}. \label{LMProj}
 	\end{eqnarray}	
Since the discrete Lagrange multiplier space is chosen as 
$\Lambda_h = \V_{p,h}\cdot \n_p|_{\Gam_{fp}}$, we have
\begin{align*}
\langle\lambda - Q_{\lambda,h}\lambda, \v_{p,h}\cdot \n_p\rangle_{\Gam_{fp}} =0, && 
\forall \, \v_{p,h}\in \V_{p,h}.
\end{align*}
We note that the discrete seminorm \eqref{norm-lh} in $\Lambda_h$ 
is well defined for any function in $L^2(\Gamma_{fp})$. It is easy to see that 
$|\lambda - Q_{\lambda,h}\lambda|_{\Lambda_h} =0$, hence
\begin{align}
\|\lambda - Q_{\lambda,h}\lambda\|_{\Lambda_h} = \|\lambda - Q_{\lambda,h}\lambda\|_{L^2(\Gam_{fp})}. 
\label{lam-proj-prop}
\end{align}
Next, we consider a Stokes-like projection operator 
$(S_{f,h}, R_{f,h}) : \V_f \rightarrow \V_{f,h} \times W_{f,h}$, defined for all $\v_f \in \V_f$ by
\begin{align}
& a_f(S_{f,h}\textbf{v}_f, \v_{f,h} ) - b_f( \v_{f,h},R_{f,h}\textbf{v}_f) = a_f (\v_f, \v_{f,h} ), 
&&\forall \v_{f,h} \in \V_{f,h}, \\
& b_f(S_{f,h}\textbf{v}_f,w_{f,h})= b_f(\textbf{v}_f,w_{f,h}), &&\forall w_{f,h} \in W_{f,h}. 
\label{stokes-proj}
\end{align}
The operator $S_{f,h}$ satisfies the approximation property \cite{fernandez2011incremental}:
\begin{align}
\|\v_f-S_{f,h}\v_f\|_{H^1(\O_f)} \leq Ch^{k_f}\|\v_f\|_{H^{k_f+1}(\O_f)}. 
\label{eq:stokes-like-approx-prop}
\end{align}
Let $\Pi_{p,h}$ be the MFE interpolant onto $\V_{p,h}$ satisfying for any 
$\theta>0$ and for all $\v_p \in \V_{p}\cap H^{\theta}(\Omega_p)$,
\begin{align}
& (\nabla \cdot \Pi_{p,h}\v_p,w_{p,h}) = (\nabla \cdot \v_p, w_{p,h}), && 
\forall w_{p,h} \in W_{p,h}, \label{mfe-operator}
\\
& \langle \Pi_{p,h} \v_p\cdot\n_p, \v_{p,h}\cdot\n_p \rangle_{\Gamma_{fp}} = 
\langle \v_p\cdot\n_p, \v_{p,h}\cdot\n_p \rangle_{\Gamma_{fp}}, && \forall 
\v_{p,h} \in \V_{p,h}. \label{mfe-operator-bis}
\end{align}
We will make use of the following bounds on $\Pi_{p,h}$ \cite{ciarlet1978finite,Mathew_1989}:
\begin{align}
& \|\v_p-\Pi_{p,h}\v_p\|_{L^2(\O_p)} \leq Ch^{k_p+1}\|\v_p\|_{H^{k_p+1}(\O_p)}, 
\label{eq:mfe-approx-prop1} \\
& \|\Pi_{p,h}\v_p\|_{H(\text{div};\O_p)} \leq C\left(\|\v_p\|_{H^{\theta}(\O_p)} 
+ \|\nabla \cdot\v_p\|_{L^{2}(\O_p)} \right). \label{eq:mfe-approx-prop2}
\end{align}
Finally, let $S_{s,h}$ be the Scott-Zhang interpolant 
onto $\X_{p,h}$, satisfying for all $\bxi_p \in H^{k_s+1}(\O_p)$ \cite{Scott_Zhang_1990}:
\begin{align}
\|\bxi_p-S_{s,h}\bxi_p\|_{L^2(\O_p)} + h|\bxi_p-S_{s,h}\bxi_p|_{H^1(\O_p)}\leq Ch^{k_s+1}\|\bxi_p\|_{H^{k_s+1}(\O_p)} \label{eq:scott-zhang-approx-prop1} 
\end{align}

\subsubsection{Construction of a weakly-continuous interpolant}
In this section we use the operators defined above to build an operator onto the space that
satisfies the weak continuity of normal velocity condition \eqref{h-b-gamma}. Let 
$$
\U = \{ (\v_f,\v_p,\bxi_p) \in \V_f \times \V_p \cap H^{\theta}(\Omega_p) \times \X_p :
\v_f\cdot\n_f + \v_p\cdot\n_p + \bxi_p\cdot\n_p = 0 \}.
$$
Consider its discrete analog
$$
\U_h = \left\{ (\v_{f,h},\v_{p,h},\bxi_{p,h}) \in \V_{f,h} \times \V_{p,h} \times \X_{p,h} 
: b_{\Gamma}\left(\v_{f,h},\v_{p,h},\bxi_{p,h};\mu_h\right)=0, \forall \mu_h \in \Lambda_h \right\}.
$$
We will construct an interpolation operator $I_h: \U \to \U_h$ as a triple
$$
I_h(\v_f,\v_p,\bxi_p) = \left(I_{f,h}\v_f , I_{p,h}\v_p,I_{s,h}\bxi_p\right),
$$
with the following properties:
\begin{align}
& b_{\Gamma}\left(I_{f,h}\v_f,I_{p,h}\v_p, I_{s,h}\bxi_p;\mu_h\right) =0, 
&& \forall  \mu_h \in \Lambda_h, \label{darcy-velocity-int1}\\
& b_f(I_{f,h}\v_f  - \v_f, w_{f,h}) =0, && \forall w_{f,h} \in W_{f,h}, 
\label{stokes-velocity-int2}\\ 
& b_p(I_{p,h}\v_p  - \v_p, w_{p,h}) =0 ,&& \forall w_{p,h} \in W_{p,h}. \label{darcy-velocity-int2} 
\end{align}
We let $I_{f,h} :=S_{f,h}$ and $I_{s,h} := S_{s,h}$. To construct $I_{p,h}$, we first consider 
an auxiliary problem:
\begin{align}
	\begin{cases}
	\nabla \cdot \nabla\phi = 0 &\textrm{ in } \O_p, \\
        \phi = 0 &\textrm{ on } \Gamma_p^D, \\
	\nabla \phi \cdot \n_p=0 &\textrm{ on } \Gamma_p^N, \\
	\nabla \phi \cdot \n_p=(\v_f-I_{f,h}\v_f)\cdot\n_f 
        + (\bxi_p-I_{s,h}\bxi_p)\cdot \n_p &\textrm{ on } \Gamma_{fp}.
	\end{cases} \label{eq:axil-problem}
\end{align}
Let $\z=\nabla \phi$ and define $\w=\z+\v_p$. By construction,
\begin{align}
\nabla \cdot \w & = \nabla \cdot \z + \nabla \cdot \v_p = \nabla \cdot \v_p 
&&\textrm{ in }\O_p, \label{div-w}\\
\w \cdot \n_p & = \z_p\cdot\n_p + \v_p\cdot\n_p = \v_f\cdot\n_f -I_{f,h}\v_f\cdot\n_f 
+ \bxi_p\cdot \n_p- I_{s,h}\bxi_p\cdot \n_p + \v_p \cdot \n_p \nonumber \\
& = -I_{f,h}\v_f\cdot\n_f - I_{s,h}\bxi_p\cdot \n_p &&\textrm{ on }\Gamma_{fp}. \label{propW}
\end{align}
We now let
\begin{equation}
I_{p,h}\v_p = \Pi_{p,h}\w. \label{operator}
\end{equation}
Next, we verify that the operator $I_h=(I_{f,h},I_{p,h},I_{s,h})$ satisfies 
\eqref{darcy-velocity-int1}--\eqref{darcy-velocity-int2}. Property \eqref{stokes-velocity-int2}
follows immediately from \eqref{stokes-proj}, while, using \eqref{mfe-operator} and 
\eqref{div-w}, property \eqref{darcy-velocity-int2} follows from
$$
 (\nabla \cdot I_{p,h}\v_p, w_{p,h})_{\O_p} =(\nabla \cdot \Pi_{p,h}\w, w_{p,h})_{\O_p}
= (\nabla \cdot \w, w_{p,h})_{\O_p} = (\nabla \cdot \v_p, w_{p,h})_{\O_p}, \quad 
\forall \, w_{p,h} \in W_{p,h}.
$$
Using \eqref{propW} and \eqref{mfe-operator-bis}, we have for all $\mu_h \in \Lambda_h$,
\begin{align*}
\langle I_{p,h}\v_p \cdot \n_p,\mu_h \rangle_{\Gamma_{fp}} = 
\langle \Pi_{p,h}\w \cdot \n_p,\mu_h \rangle_{\Gamma_{fp}} = 
\langle \w \cdot \n_p,\mu_h \rangle_{\Gamma_{fp}} =  
\langle -I_{f,h}\v_f\cdot\n_f - I_{s,h}\bxi_p\cdot \n_p,\mu_h \rangle_{\Gamma_{fp}},
\end{align*}
which implies \eqref{darcy-velocity-int1}.

The approximation properties of the components of $I_h$ are the following.
\begin{lemma}\label{l:interpolation}
	For all $(\v_f,\v_p,\bxi_p) \in \U\cap (H^{k_f+1}(\O_f),H^{k_p+1}(\O_p), H^{k_s+1}(\O_p))$, 
	\begin{align}
	&\|\v_f - I_{f,h}\v_f\|_{H^1(\O_f )} \leq Ch^{k_f}\|\v_f\|_{H^{k_f+1}(\O_f)}, \label{stokesVel}\\
	&\|\bxi_p - I_h^s\bxi_p\|_{H^m(\O_p )} \leq Ch^{k_s+1-m}\|\bxi_p\|_{H^{k_s+1}(\O_p)}, 
        \quad m=0,1, \label{displ-bound}\\
	&\|\v_p - I_{p,h}\v_p\|_{L^2(\O_p )} \leq C\left(h^{k_p+1}\|\v_p\|_{H^{k_p+1}(\O_p )}
        + h^{k_f}\|\v_f\|_{H^{k_f+1}(\O_f )} + h^{k_s}\|\bxi_p\|_{H^{k_s+1}(\O_p)}\right).
	\label{darcy-bound} 
	\end{align}
	\end{lemma}
	\begin{proof}
The bounds \eqref{stokesVel} and \eqref{displ-bound} follow immediately from
\eqref{eq:stokes-like-approx-prop} and \eqref{eq:scott-zhang-approx-prop1}.
Next, using \eqref{operator}, we have
\begin{align*}
\|\v_p - I_{p,h}\v_p\|_{L^2(\O_p )} = \|\v_p - \Pi_{p,h} \v_p - \Pi_{p,h} \z\|_{L^2(\O_p )} 
\leq  \|\v_p - \Pi_{p,h} \v_p \|_{L^2(\O_p )} + \| \Pi_{p,h} \z\|_{L^2(\O_p )}.
\end{align*}
Elliptic regularity for \eqref{eq:axil-problem} \cite{Grisvard,Lions_Magenes_1972} implies
$$
\|\z\|_{H^{\theta}(\O_p)} \leq C \left( 
\|(\v_f-I_{f,h}\v_f)\cdot\n_f \|_{H^{\theta-1/2}(\Gam_{fp})} 
+ \|\left(\bxi_p-I_{s,h}\bxi_p\right)\cdot \n_p\|_{H^{\theta-1/2}(\Gam_{fp})}\right), \,\, 
0\leq \theta \leq 1/2.
$$
Since $\nabla \cdot \z = 0$ by construction, choosing $\theta =1/2$ 
and using \eqref{eq:mfe-approx-prop2}, \eqref{Trace}, \eqref{stokesVel} and 
\eqref{displ-bound}, we get
\begin{align*}
\| \Pi_{p,h} \z\|_{L^2(\O_p)} &\leq C\|\z\|_{H^{1/2}(\O_p )} \\
&\leq C \left(\|(\v_f-I_{f,h}\v_f)\cdot\n_f \|_{L^{2}(\Gam_{fp})} + \|\left(\bxi_p-I_{s,h}\bxi_p\right)\cdot \n_p\|_{L^{2}(\Gam_{fp})}\right) \\
&\leq C \left(\|\v_f-I_{f,h}\v_f \|_{H^{1}(\O_{f})} + \|\bxi_p-I_{s,h}\bxi_p\|_{H^{1}(\O_{p})}\right) \\
& \leq C\left( h^{k_f}\|\v_f\|_{H^{k_f+1}(\O_f )} + h^{k_s}\|\bxi_p\|_{H^{k_s+1}(\O_p)}\right).
\end{align*}
Finally, by \eqref{eq:mfe-approx-prop1},
\begin{align*}
\|\v_p - I_{p,h}\v_p\|_{L^2(\O_p )} & 
\leq  \|\v_p - \Pi_{p,h} \v_p \|_{L^2(\O_p )} + \| \Pi_{p,h} \z\|_{L^2(\O_p )} \\
&\leq C\left(h^{k_p+1}\|\v_p\|_{H^{k_p+1}(\O_p )}+ h^{k_f}\|\v_f\|_{H^{k_f+1}(\O_f )} + h^{k_s}\|\bxi_p\|_{H^{k_s+1}(\O_p)}\right).
\end{align*}
\end{proof}

%>>>>>>>>>>>>>>>>>>>>>>>>>>>>>>>>>>>>>>>>>>>>>>>>>>>>>>>>>>>>>>>>>>>>>>>>>>>>>>>>>>>>>>>>>>>>>>>>>>>>>>>>>>>>>>>>>>>>>>>>>>>>>>>>>

\subsection{Error estimates}

In this section we derive a priori error estimate for the semi-discrete formulation
\eqref{h-weak-1}-\eqref{h-b-gamma}. We recall that, due to 
\eqref{h-cts-gamma}, $(\u_f,\u_p,\d_t\bbeta_{p}) \in \U$ and 
we can apply the interpolant 
$I_h(\u_f,\u_p,\d_t\bbeta_{p}) = (I_{f,h}\u_f , I_{p,h}\u_p,I_{s,h}\d_t\bbeta_{p}) \in \U_h$
for any $t \in (0,T]$. We introduce the errors for all variables
and split them into approximation and discretization errors:
\begin{align}
\e_f &:= \u_f-\u_{f,h} = (\u_f-I_{f,h}\u_f) + (I_{f,h}\u_f-\u_{f,h}) 
:= \bchi_f+\bphi_{f,h}, \nonumber \\
\e_p &:= \u_p-\u_{p,h} = (\u_p-I_{p,h}\u_p) + (I_{p,h}\u_p-\u_{p,h}) := \bchi_p+\bphi_{p,h}, 
\nonumber\\
\e_s &:= \bbeta_p-\bbeta_{p,h} = (\bbeta_p-I_{s,h}\bbeta_p) + (I_{s,h}\bbeta_p-\bbeta_{p,h}) 
:= \bchi_s+\bphi_{s,h}, \nonumber\\
e_{fp} &:= p_f-p_{f,h} = (p_f- Q_{f,h}p_f)+( Q_{f,h}p_f-p_{f,h}) :=\chi_{fp}+\phi_{fp,h}, \nonumber\\
e_{pp} &:= p_p-p_{p,h} = (p_p- Q_{p,h}p_p)+(Q_{p,h}p_p- p_{p,h}) :=\chi_{pp}+\phi_{pp,h}, \nonumber\\
e_{\lam} &:= \lam-\lam_h = (\lam-Q_{\lam,h}\lam) +(Q_{\lam,h}\lam-\lam_h) := \chi_{\lam}
+\phi_{\lam,h}. \label{errors-def}
\end{align}
Subtracting \eqref{h-weak-1}--\eqref{h-weak-2} from \eqref{h-cts-1}--\eqref{h-cts-2} and summing 
the two equations, we obtain the error equation
\begin{align}
&  a_{f}(\e_f,\v_{f,h})
+ a^d_{p}(\e_p,\v_{p,h})
+ a^e_{p}(\e_s,\bxi_{p,h}) 
+a_{BJS}(\e_f,\d_t\e_{s};\v_{f,h},\bxi_{p,h})+ b_f(\v_{f,h},e_{fp}) \nonumber \\ 
&\quad 
+ b_p(\v_{p,h},e_{pp})
+\alpha b_p(\bxi_{p,h},e_{pp})  
+ b_{\Gamma}(\v_{f,h},\v_{p,h},\bxi_{p,h};e_\lam) +\left(s_0\,\d_t e_{pp},w_{p,h}\right)  \nonumber \\
&\quad
- \alpha b_p(\d_t e_{s},w_{p,h})
- b_p(\e_{p},w_{p,h}) 
- b_f(\e_{f},w_{f,h})
  = 0,  
\end{align}
Setting $\displaystyle \v_{f,h}=\bphi_{f,h}, \v_{p,h}=\bphi_{p,h},\bxi_{p,h}=\d_t\bphi_{s,h}, w_{f,h}=\phi_{fp,h}$, and $w_{p,h}=\phi_{pp,h}$, we have
\begin{align}
&  a_{f}(\bchi_{f},\bphi_{f,h}) + a_{f}(\bphi_{f,h},\bphi_{f,h}) 
+ a^d_{p}(\bchi_{p},\bphi_{p,h}) + a^d_{p}(\bphi_{p,h},\bphi_{p,h}) + a^e_{p}\left(\bchi_{s},\d_t\bphi_{s,h}\right)+a^e_{p}\left(\bphi_{s,h},\d_t\bphi_{s,h}\right) \nonumber \\ 
&\quad    
+ a_{BJS}\left(\bchi_{f},\d_t\bchi_{s};\bphi_{f,h},\d_t\bphi_{s,h}\right) 
+ a_{BJS}\left(\bphi_{f,h},\d_t\bphi_{s,h};\bphi_{f,h},\d_t\bphi_{s,h}\right) 
+ b_f(\bphi_{f,h},\chi_{fp}) +  b_f(\bphi_{f,h},\phi_{fp,h}) 
\nonumber \\
&\quad  
+ b_p(\bphi_{p,h},\chi_{pp}) +  b_p(\bphi_{p,h},\phi_{pp,h})
+ \alpha b_p\left(\d_t\bphi_{s,h},\chi_{pp}\right) 
+ \alpha b_p\left(\d_t \bphi_{s,h},\phi_{pp,h}\right) \nonumber \\ 
&\quad
+ b_{\Gamma}\left(\bphi_{f,h},\bphi_{p,h},\d_t\bphi_{s,h};\chi_\lam\right)
+b_{\Gamma}\left(\bphi_{f,h},\bphi_{p,h},\d_t\bphi_{s,h};\phi_{\lam,h}\right) 
+\left( s_0\,\d_t \chi_{pp},\phi_{pp,h}\right) + \left( s_0\, \d_t \phi_{pp,h},\phi_{pp,h}\right)   
    \nonumber \\
&\quad - \alpha b_p\left(\d_t\bchi_s,\phi_{pp,h}\right) 
- \alpha b_p\left(\d_t\bphi_{s,h},\phi_{pp,h}\right) 
-b_p(\bchi_{p},\phi_{pp,h}) -  b_p(\bphi_{p,h},\phi_{pp,h}) \nonumber \\ 
&\quad - b_f(\bchi_{f},\phi_{fp,h}) - b_f(\bphi_{f,h},\phi_{fp,h})  = 0. \label{er-eq}
\end{align}
The following terms simplify, due
to the properties of projection operators
\eqref{darcy-pressure-int},\eqref{l-multuplier-int}, \eqref{stokes-velocity-int2}, 
and \eqref{darcy-velocity-int2}:
\begin{equation} \label{orthogonaity}
b_f(\bchi_{f},\phi_{fp,h}) = b_p(\bchi_{p},\phi_{pp,h}) = b_p(\bphi_{p,h},\chi_{pp}) = 0, \quad 
 \left(s_0 \, \d_t \chi_{pp},\phi_{pp,h}\right)=
 \langle\bphi_{p,h} \cdot\n_p,\chi_{\lambda}\rangle_{\Gam_{fp}}=0,
\end{equation}
where we also used that $\Lambda_h = \V_{p,h}\cdot \n_p|_{\Gamma_{fp}}$ for the last equality.
We also have
$$
b_{\Gamma}\left(\bphi_{f,h},\bphi_{p,h},\d_t\bphi_{s,h};\phi_{\lam,h}\right) = 0, \quad
b_{\Gamma}\left(\bphi_{f,h},\bphi_{p,h},\d_t\bphi_{s,h};\chi_\lam\right) 
= \left\langle\bphi_{f,h} \cdot \n_f + \d_t\bphi_{s,h} \cdot \n_p,\chi_{\lambda}
\right\rangle_{\Gamma_{fp}},
$$
where we have used \eqref{darcy-velocity-int1} and \eqref{h-b-gamma} for the 
first equality and the last equality in \eqref{orthogonaity} for the 
second equality. Using \eqref{int-parts}, we write
\begin{align*}
\left( s_0 \, \d_t \phi_{pp,h},\phi_{pp,h}\right)  = \frac12 s_0\, \d_t\|\phi_{pp,h}\|^2_{L^2(\Omega_p)}, \quad a^e_{p}\left(\bphi_{s,h},\d_t\bphi_{s,h}\right) = \frac12 \d_ta^e_{p}\left(\bphi_{s,h},\bphi_{s,h}\right).
\end{align*}
Rearranging terms and using the results above, the error equation 
\eqref{er-eq} becomes
\begin{align}
 & a_{f}(\bphi_{f,h},\bphi_{f,h})+ a^d_{p}(\bphi_{p,h},\bphi_{p,h})
+\frac12\d_t\left(a^e_{p}(\bphi_{s,h},\bphi_{s,h}) + s_0\|\phi_{pp,h}\|^2_{L^2(\Omega_p)} \right)   
+\left|\bphi_{f,h}-\d_t\bphi_{s,h} \right|^2_{a_{BJS}}  
\nonumber \\
&\quad  
=   a_{f}(\bchi_{f},\bphi_{f,h}) 
+  a^d_{p}(\bchi_{p},\bphi_{p,h}) +  a^e_{p}\left(\bchi_{s},\d_t\bphi_{s,h}\right) \nonumber \\
&\quad + \sum_{j=1}^{d-1}\left\langle\mu\alpha_{BJS}\sqrt{K_j^{-1}}(\bchi_f - \d_t\bchi_s)
\cdot\btau_{f,j},(\bphi_{f,h} -\d_t\bphi_{s,h}) \cdot\btau_{f,j} \right\rangle_{\Gam_{fp}} 
+ b_f(\bphi_{f,h},\chi_{fp})  \nonumber \\
&\quad  +  \alpha b_p(\d_t\bphi_{s,h},\chi_{pp}) + \alpha b_p(\d_t\bchi_{s},\phi_{pp,h}) 
+ \langle\bphi_{f,h} \cdot \n_f + \d_t\bphi_{s,h} \cdot \n_p,\chi_{\lambda}\rangle_{\Gamma_{fp}}.
\label{er-eq-1}  
\end{align}
We proceed with bounding the terms on the right-hand side in \eqref{er-eq-1}. 
Using the continuity of the bilinear forms \eqref{C1} and \eqref{C2} 
and inequalities \eqref{CS} and \eqref{Young}, we have
\begin{equation}\label{rhs1}
 a_{f}(\bchi_{f},\bphi_{f,h}) +  a^d_{p}(\bchi_{p},\bphi_{p,h}) 
\leq C\epsilon_1^{-1}\left( \|\bchi_{f}\|^2_{H^1(\Omega_f)} + \|\bchi_{p}\|^2_{L^2(\Omega_p)} \right)
+ \epsilon_1\left( \|\bphi_{f,h}\|^2_{H^1(\Omega_f)}  + \|\bphi_{p,h}\|^2_{L^2(\Omega_p)}\right).
\end{equation}
Similarly, using inequalities \eqref{CS}, \eqref{Trace} and \eqref{Young}, we obtain
\begin{align}
     & \sum_{j=1}^{d-1}\left\langle\mu\alpha_{BJS}\sqrt{K_j^{-1}}(\bchi_f - \d_t\bchi_s)
     \cdot\btau_{f,j},(\bphi_{f,h} - \d_t\bphi_{s,h}) \cdot\btau_{f,j} \right\rangle_{\Gam_{fp}}  
\nonumber \\
     &\qquad\qquad \leq \epsilon_1 \left|\bphi_{f,h}-\d_t\bphi_{s,h} \right|_{a_{BJS}}^2    
+ C\epsilon_1^{-1}\left(\|\bchi_{f}  \|^2_{H^1(\O_{f})} + \|\d_t\bchi_{s} \|^2_{H^1(\O_{p})} \right).
\label{rhs2}
\end{align}
Finally, using \eqref{CS},\eqref{Trace} and \eqref{Young}, we bound the rest of the terms 
that do not involve $\d_t\bphi_{s,h}$:
\begin{align}
    &b_f(\bphi_{f,h},\chi_{fp})+\alpha b_p\left(\d_t\bchi_{s},\phi_{pp,h}\right)
+ \langle\bphi_{f,h} \cdot \n_f ,\chi_{\lambda}\rangle_{\Gamma_{fp}}  \nonumber \\
 & \quad \leq   C\epsilon_1^{-1}\left(\|\chi_{fp}\|^2_{L^2(\Omega_f)} 
+ \left\|\nabla \cdot\d_t\bchi_{s}\right\|^2_{L^2(\Omega_p)}
+ \|\chi_{\lambda}\|^2_{L^2(\Gamma_{fp})}\right) \nonumber \\
& \qquad
+ \epsilon_1\left(\|\nabla \cdot\bphi_{f,h}\|^2_{L^2(\Omega_f)} +\|\phi_{pp,h}\|^2_{L^2(\Omega_p)} 
+ \|\bphi_{f,h}\cdot \n_f\|^2_{L^2(\Gamma_{fp})} \right) \nonumber \\
    & \quad \leq C\epsilon_1^{-1}\left(  \|\chi_{fp}\|^2_{L^2(\Omega_f)} 
+ \left\|\d_t\bchi_{s}\right\|^2_{H^1(\Omega_p)}  
+  \|\chi_{\lambda}\|^2_{L^2(\Gamma_{fp})}\right)  
+\epsilon_1 \left(\|\bphi_{f,h}\|^2_{H^1(\Omega_f)}
+ \|\phi_{pp,h}\|^2_{L^2(\Omega_p)}\right). \label{rhs3}
\end{align}
Combining \eqref{er-eq-1}--\eqref{rhs3}, integrating 
over $[0,t]$, where $0< t\leq T$, using the coercivity of the bilinear forms 
\eqref{C1}--\eqref{C3}, and taking $\epsilon_1$ small enough, we obtain
\begin{align}
 & 
\|\bphi_{s,h}(t)\|^2_{H^1(\Omega_f)}
+ s_0 \|\phi_{pp,h}(t)\|^2_{L^2(\Omega_p)}  
+ \|\bphi_{f,h}\|^2_{L^2(0,t;H^1(\Omega_f))} 
\nonumber \\
& \qquad\qquad\qquad
+ \|\bphi_{p,h}\|^2_{L^2(0,t;L^2(\Omega_p))}
+ \left|\bphi_{f,h}-\d_t\bphi_{s,h} \right|^2_{L^2(0,t;a_{BJS})} 
\nonumber \\
& \qquad\leq 
\epsilon_1 \|\phi_{pp,h}\|^2_{L^2(0,t;L^2(\Omega_p))} +
C\epsilon_1^{-1}\left(\|\d_t \bchi_{s}\|^2_{L^2(0,t;H^1(\Omega_p))} 
+ \|\chi_{fp}\|^2_{L^2(0,t;L^2(\Omega_f))} 
\right. \nonumber \\
& \qquad\qquad\ \left.
+ \|\bchi_{f}\|^2_{L^2(0,t;H^1(\Omega_f))} 
+ \|\bchi_{p}\|^2_{L^2(0,t;L^2(\Omega_p))} 
+ \|\chi_{\lambda}\|^2_{L^2(0,t;L^2(\Gamma_{fp}))} 
\right)
\nonumber \\
&\qquad\qquad\ 
+  C \int_{0}^{t} \left( a^e_{p}(\bchi_{s}, \d_t\bphi_{s,h})
+ \alpha b_p(\d_t\bphi_{s,h},\chi_{pp})
+ \langle \d_t\bphi_{s,h} \cdot \n_p,\chi_{\lambda}\rangle_{\Gamma_{fp}} \right) ds
\nonumber \\
&\qquad\qquad\ 
+ C\left(\|\bphi_{s,h}(0)\|^2_{H^1(\Omega_f)} + s_0\|\phi_{pp,h}(0)\|^2_{L^2(\Omega_p)}\right).
  \label{eq:error-ineq}
\end{align}
For the initial conditions, we set $p_{p,h}(0)=Q_{p,h} p_{p,0}$ and 
$\bbeta_{p,h}(0)=  I_{s,h} \bbeta_{p,0}$, implying
\begin{align}
\bphi_{s,h}(0)=0, \qquad \phi_{pp,h}(0)=0 \label{initial-error}
\end{align}
We next bound the terms on the right involving $\d_t\bphi_{s,h}$.
Using integration by parts in time, \eqref{CS}, \eqref{Young}, \eqref{C3} and 
\eqref{initial-error}, we obtain
\begin{align}
& \int_{0}^{t}a^e_{p}\left(\bchi_{s},\d_t\bphi_{s,h}\right)ds = a^e_{p}(\bchi_{s},\bphi_{s,h}) \big|_0^t - \int_{0}^{t}a^e_{p}\left(\d_t\bchi_{s},\bphi_{s,h}\right)ds \nonumber \\
& \quad\leq 
C\left(\epsilon_1^{-1} \|\bchi_{s}(t)\|^2_{H^1(\Omega_p)} + \left\|\d_t\bchi_{s}\right\|^2_{L^2(0,t;H^1(\Omega_p))} \right)  
+\epsilon_1\|\bphi_{s,h}(t)\|^2_{H^1(\Omega_p)} +\|\bphi_{s,h}\|^2_{L^2(0,t;H^1(\Omega_p))}.
\end{align}
Similarly, using \eqref{CS}, \eqref{Trace}, \eqref{Young} and \eqref{initial-error}, we have
\begin{align}
&  \int_{0}^{t}\left\langle\d_t\bphi_{s,h} \cdot \n_p,\chi_{\lambda}\right\rangle_{\Gamma_{fp}}ds +\int_{0}^{t}\alpha b_p\left(\d_t\bphi_{s,h},\chi_{pp}\right)ds
\nonumber \\
&\quad = \langle \bphi_{s,h} \cdot \n_p,\chi_{\lambda}\rangle_{\Gamma_{fp}}\big|^t_0 
-  \int_{0}^{t}\left\langle \bphi_{s,h} \cdot \n_p,\d_t \chi_{\lambda}\right\rangle_{\Gamma_{fp}}ds 
+\alpha b_p\left(\bphi_{s,h},\d_t\chi_{pp}\right) \big|_0^t -\int_{0}^{t}\alpha b_p\left(\bphi_{s,h},\d_t\chi_{pp}\right)ds \nonumber \\
&\quad \leq
 \epsilon_1\|\bphi_{s,h}(t) \cdot \n_p\|^2_{L^2(\Gamma_{fp})}
+ \|\bphi_{s,h} \cdot \n_p\|^2_{L^2(0,t;L^2(\Gamma_{fp}))}  
+ \epsilon_1\|\nabla \cdot \bphi_{s,h}(t)\|^2_{L^2(\O_{p})}
+ \|\nabla \cdot \bphi_{s,h}\|^2_{L^2(0,t;L^2(\O_{p}))} 
\nonumber \\
&\qquad +C \bigg(\epsilon_1^{-1}\|\chi_{\lambda}(t) \|^2_{L^2(\Gamma_{fp})}   
+ \left\|\d_t\chi_{\lambda} \right\|^2_{L^2(0,t;L^2(\Gamma_{fp}))}
+ \epsilon_1^{-1}\|\chi_{pp}(t)\|^2_{L^2(\O_{p})} 
+ \|\d_t\chi_{pp}\|^2_{L^2(0,t;L^2(\O_{p}))}\bigg)\nonumber \\
&\qquad \leq\epsilon_1 \|\bphi_{s,h}(t) \|^2_{H^1(\O_{p})} 
+ \|\bphi_{s,h} \|^2_{L^2(0,t;H^1(\O_{p}))} \nonumber \\
&\qquad 
+C \bigg(\epsilon_1^{-1}\|\chi_{\lambda}(t) \|^2_{L^2(\Gamma_{fp})}   
+ \left\|\d_t\chi_{\lambda} \right\|^2_{L^2(0,t;L^2(\Gamma_{fp}))}
+ \epsilon_1^{-1}\|\chi_{pp}(t)\|^2_{L^2(\O_{p})} 
+ \|\d_t\chi_{pp}\|^2_{L^2(0,t;L^2(\O_{p}))}\bigg). \label{dt-phi-bound}
\end{align}
Using \eqref{initial-error}--\eqref{dt-phi-bound} and taking $\epsilon_1$ small enough, 
we obtain from \eqref{eq:error-ineq},
\begin{align}
 &  
\|\bphi_{s,h}(t)\|^2_{H^1(\O_p)} + s_0 \|\phi_{pp,h}(t)\|^2_{L^2(\Omega_p)}  
+ \|\bphi_{f,h}\|^2_{L^2(0,t;H^1(\O_f))} 
\nonumber \\
 &\qquad\qquad\qquad 
+ \|\bphi_{p,h}\|^2_{L^2(0,t;L^2(\O_p))}
+ \left|\bphi_{f,h}-\d_t\bphi_{s,h} \right|^2_{L^2(0,t;a_{BJS})}  
\nonumber \\
 &\qquad
\leq \epsilon_1\|\phi_{pp,h}\|^2_{L^2(0,t;L^2(\Omega_p))} 
+ \|\bphi_{s,h} \|^2_{L^2(0,t;H^1(\O_{p}))}    \nonumber \\
 &\qquad\qquad  + C\epsilon_1^{-1}
\left( 
\|\chi_{fp}\|^2_{L^2(0,t;L^2(\Omega_f))}  + \|\bchi_{f}\|^2_{L^2(0,t;H^1(\Omega_f))} 
+ \|\bchi_{p}\|^2_{L^2(0,t;L^2(\Omega_p))} 
\right. \nonumber \\
&\qquad\qquad\qquad \left.
+ \|\chi_{\lambda}(t) \|^2_{L^2(\Gamma_{fp})} 
+ \|\chi_{pp}(t)\|^2_{L^2(\O_{p})}
+ \|\chi_{\lambda}\|^2_{L^2(0,t;L^2(\Gamma_{fp}))}
+ \|\bchi_{s}(t)\|^2_{H^1(\Omega_p)} 
\right)
\nonumber \\
&\qquad\qquad 
+ C \left(
\|\d_t \bchi_{s}\|^2_{L^2(0,t;H^1(\Omega_p))}  
+ \|\d_t\chi_{\lambda}\|^2_{L^2(0,t;L^2(\Gamma_{fp}))} 
+ \|\d_t\chi_{pp}\|^2_{L^2(0,t;L^2(\O_{p}))}
\right). \label{eq:error-ineq1}
\end{align}
Next, we use the inf-sup condition \eqref{inf-sup} with the choice 
$(w_h,\mu_h) =((\phi_{fp,h},\phi_{pp,h}), \phi_{\lam, h})$ 
and the error equation obtained by subtracting \eqref{h-weak-1} from \eqref{h-cts-1}:
\begin{align*}
    &\|((\phi_{fp,h},\phi_{pp,h}),\phi_{\lam,h})\|_{W\times\Lambda_h}   \nonumber\\
    & \quad\leq C \sup_{(\v_h,\bxi_{p,h}) \in \V_h\times \X_{p,h}}
\frac{b_f(\v_{f,h},\phi_{fp,h})+b_p(\v_{p,h},\phi_{pp,h})
+\alpha b_p(\bxi_{p,h}, \phi_{pp,h}) 
+ b_{\Gamma}(\v_{f,h},\v_{p,h},\bxi_{p,h};\phi_{\lam,h})}
{\|(\v_h,\bxi_{p,h})\|_{\V\times\X_p}} \nonumber\\
    &\quad =  \sup_{(\v_h,\bxi_{p,h}) \in \V_h\times \X_{p,h}}
\bigg(\frac{-a_{f}(\e_f,\v_{f,h})- a^d_{p}(\e_p,\v_{p,h})
- a^e_{p}(\e_s,\bxi_{p,h}) -a_{BJS}(\e_f,\d_t\e_{s};\v_{f,h},\bxi_{p,h})}
{\|(\v_h,\bxi_{p,h})\|_{\V\times\X_p}} \nonumber \\
    &\qquad\quad
+ \frac{-b_f(\v_{f,h},\chi_{fp})-b_p(\v_{p,h},\chi_{pp})
-\alpha b_p(\bxi_{p,h}, \chi_{pp}) 
- b_{\Gamma}(\v_{f,h},\v_{p,h},\bxi_{p,h};\chi_{\lam})} 
{\|(\v_h,\bxi_{p,h})\|_{\V\times\X_p}}\bigg).
\end{align*}
Due to \eqref{darcy-pressure-int} and \eqref{l-multuplier-int}, 
$b_p(\v_{p,h},\chi_{pp})=\langle \v_{p,h}\cdot \n_p,\chi_{\lam}\rangle_{\Gam_{fp}}=0$. 
Then, integrating over $[0,t]$ and using the continuity of the bilinear forms 
\eqref{C1}--\eqref{C3} and the trace inequality \eqref{Trace}, we get
\begin{align}
&\epsilon_2 (\|\phi_{fp,h}\|^2_{L^2(0,t;L^2(\Omega_f))} + 
\|\phi_{pp,h}\|^2_{L^2(0,t;L^2(\Omega_p))}
+ \|\phi_{\lam,h}\|^2_{L^2(0,t;L^2(\Gamma_{fp}))})
\nonumber\\
&\quad\leq C \epsilon_2\left( \|\bphi_{f,h}\|^2_{L^2(0,t;H^1(\O_f))}
+ \|\bphi_{p,h}\|^2_{L^2(0,t;L^2(\O_p))}
+ \|\bphi_{s,h}\|^2_{L^2(0,t;H^1(\O_p))}
+\left|\bphi_{f,h}-\d_t\bphi_{s,h} \right|^2_{L^2(0,t;a_{BJS})}  
\right.\nonumber \\
&\qquad\qquad\qquad 
+\|\bchi_{f}\|^2_{L^2(0,t;H^1(\O_f))} 
+ \|\bchi_{p}\|^2_{L^2(0,t;L^2(\O_p))} 
+ \left\|\d_t\bchi_{s}  \right\|^2_{L^2(0,t;H^1(\O_{p}))} 
\nonumber \\
&\qquad\qquad\qquad \left.
+ \|\chi_{fp}\|^2_{L^2(0,t;L^2(\O_f))} + \|\chi_{pp}\|^2_{L^2(0,t;L^2(\O_p))}
+ \|\chi_{\lam}\|_{L^2(0,t;L^2(\Gamma_{fp}))}
\right).
\label{eq-error-ineq-inf-sup}
\end{align}
Adding \eqref{eq:error-ineq1} and \eqref{eq-error-ineq-inf-sup} and taking $\epsilon_2$ 
small enough, and then $\epsilon_1$ small enough, gives
\begin{align}
 &  
 \|\bphi_{s,h}(t)\|^2_{H^1(\O_p)} + s_0 \|\phi_{pp,h}(t)\|^2_{L^2(\Omega_p)}  
+ \|\bphi_{f,h}\|^2_{L^2(0,t;H^1(\O_f))} + \|\bphi_{p,h}\|^2_{L^2(0,t;L^2(\O_p))}
\nonumber \\
 &\qquad  + \left|\bphi_{f,h}-\d_t\bphi_{s,h} \right|^2_{L^2(0,t;a_{BJS})} +  \|\phi_{fp,h}\|^2_{L^2(0,t;L^2(\O_f))}+\|\phi_{pp,h}\|^2_{L^2(0,t;L^2(\O_p))} +\|\phi_{\lam,h} \|^2_{L^2(0,t;\Lambda_h)}\nonumber \\
 &\quad  \leq C\bigg(  \|\bphi_{s,h} \|^2_{L^2(0,t;H^1(\O_{p}))} + \left\|\d_t \bchi_{s}\right\|^2_{L^2(0,t;H^1(\Omega_p))} + \|\chi_{fp}\|^2_{L^2(0,t;L^2(\Omega_f))}+ \|\bchi_{f}\|^2_{L^2(0,t;H^1(\Omega_f))}  \nonumber \\
&\qquad   + \|\bchi_{p}\|^2_{L^2(0,t;L^2(\Omega_p))} + \|\chi_{\lambda}(t) \|^2_{L^2(\Gamma_{fp})} + \|\chi_{pp}(t)\|^2_{L^2(\O_{p})} + \|\chi_{\lambda}\|^2_{L^2(0,t;L^2(\Gamma_{fp}))} + \left\|\d_t\chi_{\lambda} \right\|^2_{L^2(0,t;L^2(\Gamma_{fp}))} \nonumber \\
&\qquad   + \|d_t\chi_{pp}\|^2_{L^2(0,t;L^2(\O_{p}))}
+ \|\d_t\bchi_{s}(t)\|^2_{H^1(\Omega_p)} \bigg).
\label{eq:error-ineq2}
\end{align}  
Applying Gronwall's inequality \eqref{Gronwall} and using the triangle
inequality and the approximation properties 
\eqref{stokesPresProj}--\eqref{LMProj}, \eqref{lam-proj-prop} and
\eqref{stokesVel}--\eqref{darcy-bound}, results in the following theorem.
\begin{theorem}\label{thm:error-semi-discrete}
Assuming sufficient smoothness for the solution of \eqref{h-cts-1}--\eqref{h-cts-gamma},
then the solution of the semi-discrete problem \eqref{h-weak-1}--\eqref{h-b-gamma}
with $p_{p,h}(0)=Q_{p,h} p_{p,0}$ and $\bbeta_{p,h}(0)=  I_{s,h} \bbeta_{p,0}$
satisfies
\begin{align}
&
 \|\bbeta_{p} -\bbeta_{p,h}\|_{L^{\infty}(0,T;H^1(\O_p))} 
+ \sqrt{s_0} \|p_p -p_{p,h}\|_{L^{\infty}(0,T;L^2(\Omega_p))} \nonumber \\
& \qquad
 +  \|\u_f - \u_{f,h}\|_{L^2(0,T;H^1(\O_f))} + \|\u_p-\u_{p,h}\|_{L^2(0,T;L^2(\O_p))}
 + \left|(\u_{f}-\d_t\bbeta_{p})-(\u_{f,h}-\d_t\bbeta_{p,h}) \right|_{L^2(0,T;a_{BJS})}
\nonumber \\
& \qquad
+\|p_f-p_{f,h}\|_{L^2(0,T;L^2(\O_f))}+\|p_p-p_{p,h}\|_{L^2(0,T;L^2(\O_p))} 
+\|\lam-\lam_h \|_{L^2(0,T;\Lambda_h)} \nonumber \\
& \quad
   \leq C\sqrt{\exp(T)}\bigg(  h^{k_f}\| \u_{f}\|_{L^{2}(0,T;H^{k_f+1}(\Omega_f))} + h^{s_f+1}\|p_f\|_{L^2(0,T;H^{s_f+1}(\Omega_f))}  \nonumber \\
& \qquad
 + h^{k_p+1}\left(\|\u_{p}\|_{L^2(0,T;H^{k_p+1}(\Omega_p))} 
+ \|\lambda\|_{L^2(0,T;H^{k_p+1}(\Gamma_{fp}))}
\right. \nonumber\\
& \qquad\qquad\qquad \left.
+ \|\lambda\|_{L^{\infty}(0,T;H^{k_p+1}(\Gamma_{fp}))}+ \left\|\d_t\lambda \right\|_{L^2(0,T;H^{k_p+1}(\Gamma_{fp}))} \right)  \nonumber \\
& \qquad
 + h^{s_p+1}\left(\|p_p\|_{L^{\infty}(0,T;H^{s_p+1}(\O_{p}))}+ \|p_p\|_{L^2(0,T;H^{s_p+1}(\O_{p}))}
+ \|\d_tp_p\|_{L^2(0,T;H^{s_p+1}(\O_{p}))} \right)  
\nonumber \\
& \qquad
 +h^{k_s}\left(\|\bbeta_{p}\|_{L^{\infty}(0,T;H^{k_s+1}(\Omega_p))}  +\left\|\bbeta_{p}\right\|_{L^2(0,T;H^{k_s+1}(\Omega_p))}+\left\|\d_t\bbeta_{p}\right\|_{L^2(0,T;H^{k_s+1}(\Omega_p))} \right)\bigg). 
\label{error-estimate}
\end{align}
\label{error-semi-discrete}
\end{theorem}
%>>>>>>>>>>>>>>>>>>>>>>>>>>>>>>>>>>>>>>>>>>>>>>>>>>>>>>>>>>>>>>>>>>>>>>>>>>>>>>>>>>>>>>>>>>>>>>>>>>>>>>>>>>>>>>>>>>>>>>>>>>>>>>>>>
\section{Fully discrete formulation}
For the time discretization we employ the backward Euler method. Let $\tau$ be the time 
step, $T = N \tau$, and let $t_n = n \tau$, $0 \le n \le N$. Let 
$\dt u^n := \tau^{-1}(u^n-u^{n-1})$ be the first order
(backward) discrete time derivative, where $u^n:= u(t_n)$.
Then the fully discrete model reads: given $p^0_{p,h} = p_{p,h}(0)$ and 
$\bbeta^0_{p,h} = \bbeta_{p,h}(0)$, find $\u^n_{f,h} \in \V_{f,h}$, $p^n_{f,h} \in W_{f,h}$,
$\u^n_{p,h} \in \V_{p,h}$, $p^n_{p,h} \in W_{p,h}$, $\bbeta^n_{p,h} \in \X_{p,h}$, 
and $\lambda^n_h \in \Lambda_h$, $1\leq n\leq N$,
such that for all $\v_{f,h} \in \V_{f,h}$, $w_{f,h} \in W_{f,h}$,
$\v_{p,h} \in \V_{p,h}$, $w_{p,h} \in W_{p,h}$, $\bxi_{p,h} \in \X_{p,h}$, 
and $\mu_h \in \Lambda_h$,
\begin{align}
& a_{f}(\u^n_{f,h},\v_{f,h}) + a^d_{p}(\u^n_{p,h},\v_{p,h}) 
+ a^e_{p}(\bbeta^n_{p,h},\bxi_{p,h}) + a_{BJS}(\u^n_{f,h},\dt\bbeta^n_{p,h};\v_{f,h},\bxi_{p,h})
+ b_f(\v_{f,h},p^n_{f,h})  \nonumber \\
&\qquad\quad
+ b_p(\v_{p,h},p^n_{p,h}) + \alpha b_p(\bxi_{p,h},p^n_{p,h})  
	+ b_{\Gamma}(\v_{f,h},\v_{p,h},\bxi_{p,h};\lam^n_h) = (\f^n_{f},\v_{f,h})_{\O_f} 
	+ (\f^n_{p},\bxi_{p,h})_{\O_p}, \label{disc-h-weak-1} \\
	&	( s_0 \dt p^n_{p,h},w_{p,h})_{\O_p} 
	- \alpha b_p(\dt\bbeta^n_{p,h},w_{p,h}) 
	- b_p(\u^n_{p,h},w_{p,h}) - b_f(\u^n_{f,h},w_{f,h}) \nonumber
	\\
	&\qquad\quad
	= (q^n_{f},w_{f,h})_{\O_f} + (q^n_{p},w_{p,h})_{\O_p}, \label{disc-h-weak-2} \\
	&
	b_{\Gamma}(\u^n_{f,h},\u^n_{p,h},\dt\bbeta^n_{p,h};\mu_h) = 0. \label{disc-h-b-gamma}
\end{align}
We introduce the discrete-in-time norms
\begin{align*}
\|\phi\|^2_{l^2(0,T;X)} :=\left( \tau \sum_{n=1}^N \|\phi^n\|^2_X\right)^{1/2}, 
\qquad \|\phi\|^2_{l^{\infty}(0,T;X)} := \max_{0\leq n\leq N}\|\phi^n\|_X.
\end{align*}
Next, we state the main results for the formulation
\eqref{disc-h-weak-1}-\eqref{disc-h-b-gamma}. The proofs follow the framework in the 
semi-discrete case. Details can be found in the Appendix.
\begin{theorem}\label{stability-fully-discrete}
The solution of fully discrete problem \eqref{disc-h-weak-1}-\eqref{disc-h-b-gamma} 
satisfies
	\begin{align*}
& \sqrt{s_0} \|p_{p,h}\|_{l^{\infty}(0,T;L^2(\O_p))} + \|\bbeta_{p,h}\|_{l^{\infty}(0,T;H^1(\O_p))}+\|\u_{f,h}\|_{l^2(0,T;H^1(\O_{f}))} +\|\u_{p,h}\|_{l^2(0,T;L^2(\O_{p}))}  \\
& \qquad +|\u_{f,h} - \dt \bbeta_{p,h}|_{l^2(0,T;a_{BJS})} +\|p_{p,h}\|_{l^2(0,T;L^2(\O_p))} +\|p_{f,h}\|_{l^2(0,T;L^2(\O_f))}+\|\lam_{h}\|_{l^2(0,T;\Lambda_h)}\\
& \qquad
	+ \tau \left( \sqrt{s_0}\|\dt p_{p,h}\|_{l^2(0,T;L^2(\O_p))} + \|\dt \bbeta_{p,h}\|_{l^2(0,T;H^1(\O_p))}\right) \\
& \quad
	\leq C\sqrt{\exp(T)}\Big(\sqrt{s_0} \|p^0_{p,h}\|_{L^2(\O_p)}  + \|\bbeta^0_{p,h}\|_{H^1(\O_p)} +  \|\f_{p}\|_{l^{\infty}(0,T;L^2(\O_p))} + \|\d_t\f_p\|_{L^2(0,T;L^2(\O_p))} \\
& \qquad
	+   \|\f_{f}\|_{l^2(0,T;L^2(\O_f))} + \|q_{f}\|_{l^2(0,T;L^2(\O_f))} + \|q_{p}\|_{l^2(0,T;L^2(\O_p))}+\|\f_{p}\|_{l^2(0,T;L^2(\O_p))}  \Big).
	\end{align*}
\end{theorem}
\begin{theorem}\label{error-fully-discrete}
Assuming sufficient smoothness for the solution of \eqref{h-cts-1}--\eqref{h-cts-gamma},
then the solution of the fully discrete problem 
\eqref{disc-h-weak-1}-\eqref{disc-h-b-gamma} satisfies
	\begin{align*}
& \sqrt{s_0} \|p_p-p_{p,h}\|_{l^{\infty}(0,T;L^2(\O_p))} + \|\bbeta-\bbeta_{p,h}\|_{l^{\infty}(0,T;H^1(\O_p))}\\
& \qquad
	+ \sqrt{\tau}\Big(  \sqrt{s_0} \|d_{\tau}(p_p- p_{p,h})\|_{l^{2}(0,T;L^2(\O_p))} 
	+\|d_{\tau}(\bbeta_p-\bbeta_{p,h})\|_{l^2(0,T;H^1(\O_p))}	 \Big) \\
& \qquad
	+\Big( \|\u_f-\u_{f,h}\|_{l^2(0,T;H^1(\O_f))} + \|\u_p-\u_{p,h}\|_{l^2(0,T;L^2(\O_p))} 
        +  |\u_f-\dt \bbeta_p -(\u_{f,h} - \dt\bbeta_{p,h}) |_{l^2(0,T;a_{BJS})}\\ 
& \qquad
+\|p_f-p_{f,h}\|_{l^2(0,T;L^2(\O_f))} +\|p_p-p_{p,h}\|_{l^2(0,T;L^2(\O_p))} +\|\lam-\lam_{h}\|_{l^2(0,T;\Lambda_h)}\Big)  \\
& \quad
	\leq C\sqrt{\exp( T)}\Big( h^{k_f}\|\u_{f}\|_{l^2(0,T;H^{k_f+1}(\O_f))} +h^{s_f+1}\|p_{f}\|_{l^2(0,T;H^{s_f+1}(\O_f))} \\
& \qquad
	 + h^{k_p+1}\left(\|\u_{p}\|_{l^2(0,T;H^{k_p+1}(\O_p))} + \|\lam\|_{l^2(0,T;H^{k_p+1}(\Gam_{fp}))}+\|\lam\|_{l^{\infty}(0,T;H^{k_p+1}(\Gam_{fp}))}+\|\d_t \lam\|_{L^2(0,T;H^{k_p+1}(\O_p))}\right)\\
& \qquad
	+h^{s_p+1}\left( \|p_{p}\|^2_{l^2(0,T;H^{k_p+1}(\O_p))}+\|p_{p}\|_{l^{\infty}(0,T;H^{s_p+1}(\O_p))}+\|\d_t p_{p}\|_{L^2(0,T;H^{s_p+1}(\O_p))} \right)\\
& \qquad
	+h^{k_s}\left( \|\bbeta_{p}\|_{l^2(0,T;H^{k_s+1}(\O_p))} +\|\bbeta_{p}\|_{l^{\infty}(0,T;H^{k_s+1}(\O_p))}+\|\d_t\bbeta_{p}\|_{L^2(0,T;H^{k_s+1}(\O_p))}\right)\\
& \qquad
	 +\tau\left(\sqrt{s_0}\|\d_{tt}p_p\|_{L^{2}(0,T;L^{2}(\O_p))}+\|\d_{tt}\bbeta_p\|_{L^{2}(0,T;H^{1}(\O_p))}  \right) \Big).
	\end{align*}
\end{theorem}
%>>>>>>>>>>>>>>>>>>>>>>>>>>>>>>>>>>>>>>>>>>>>>>>>>>>>>>>>>>>>>>>>>>>>>>>>>>>>>>>>>>>>>>>>>>>>>>>>>>>>>>>>>>>>>>>>>>>>>>>>>>>>>>>>>

\section{Numerical results}
In this section, we present results from several computational
experiments in two dimensions. The fully discrete method
\eqref{disc-h-weak-1}--\eqref{disc-h-b-gamma} has been implemented using 
the finite element package FreeFem++ \cite{freefem}. 
The first test confirms the theoretical
convergence rates for the problem using an analytical solution. 
The second and third examples show the applicability of the
method to modeling fluid flow in an irregularly shaped fractured
reservoir with physical parameters, while the last one performs a
sensitivity analysis for the method with respect to various
parameters.
\begin{figure}[h]
	\centering
	\begin{subfigure}[b]{0.3\textwidth}
	\includegraphics[height=0.25\textheight]{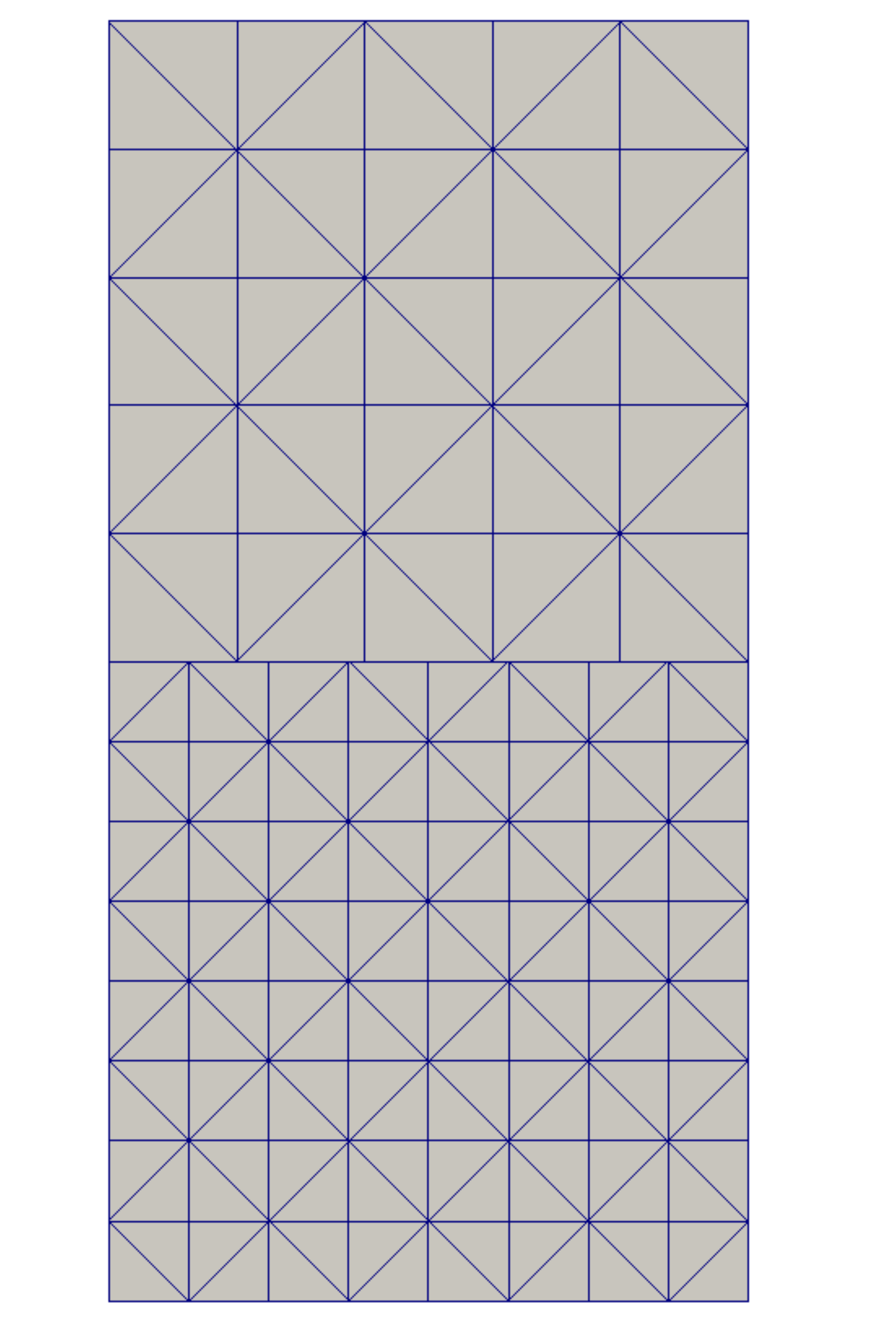}
		\caption{Computational domain $\Omega$ in Example 1, non-matching grids}
		\label{fig:1_1}
	\end{subfigure}
	\begin{subfigure}[b]{0.3\textwidth}
	\includegraphics[height=0.25\textheight]{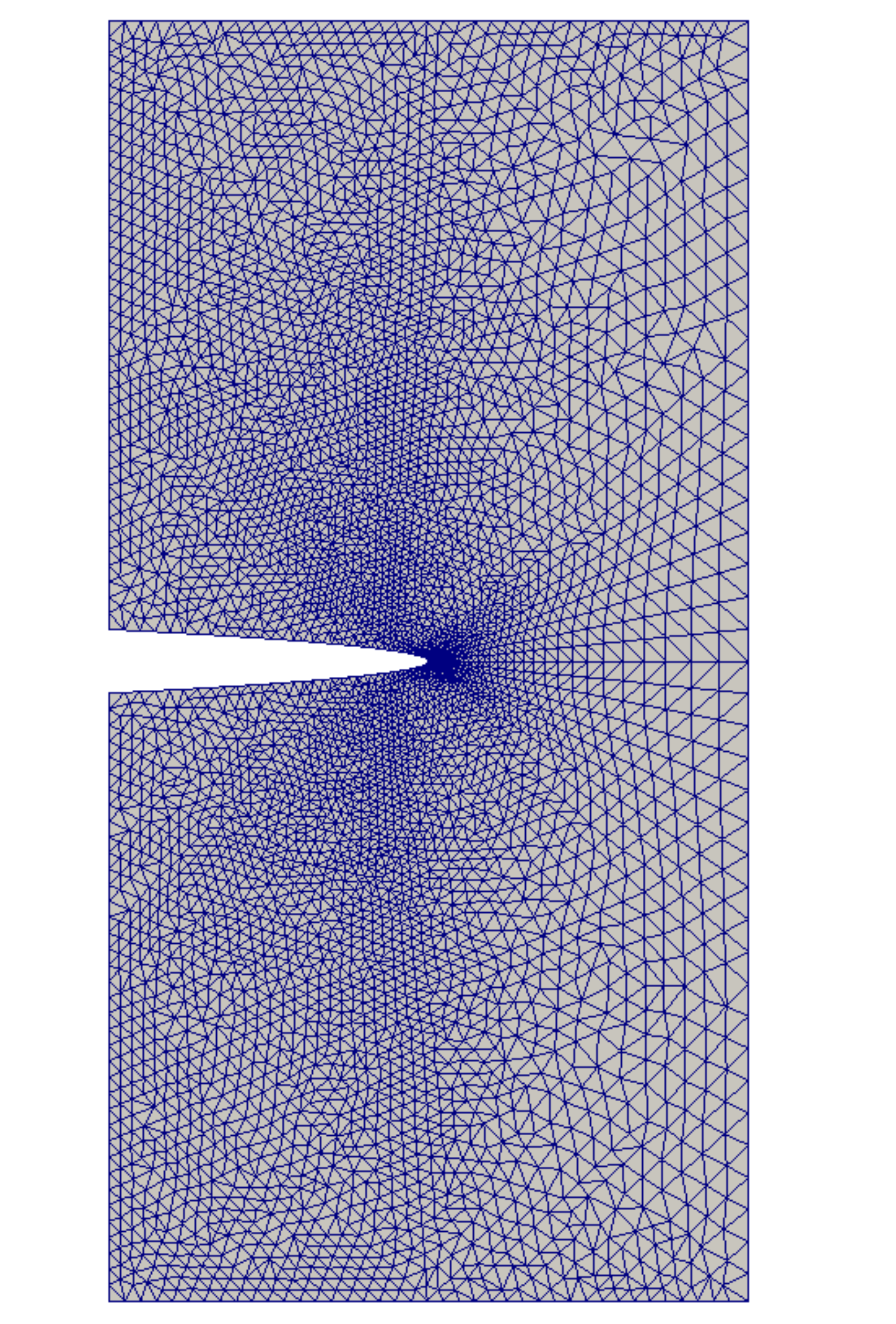}
		\caption{Reference domain $\hat{\Omega}$ in \\Examples 2, 3, and 4}
		\label{fig:1_2}
	\end{subfigure}
	\begin{subfigure}[b]{0.3\textwidth}
	\includegraphics[width=\textwidth]{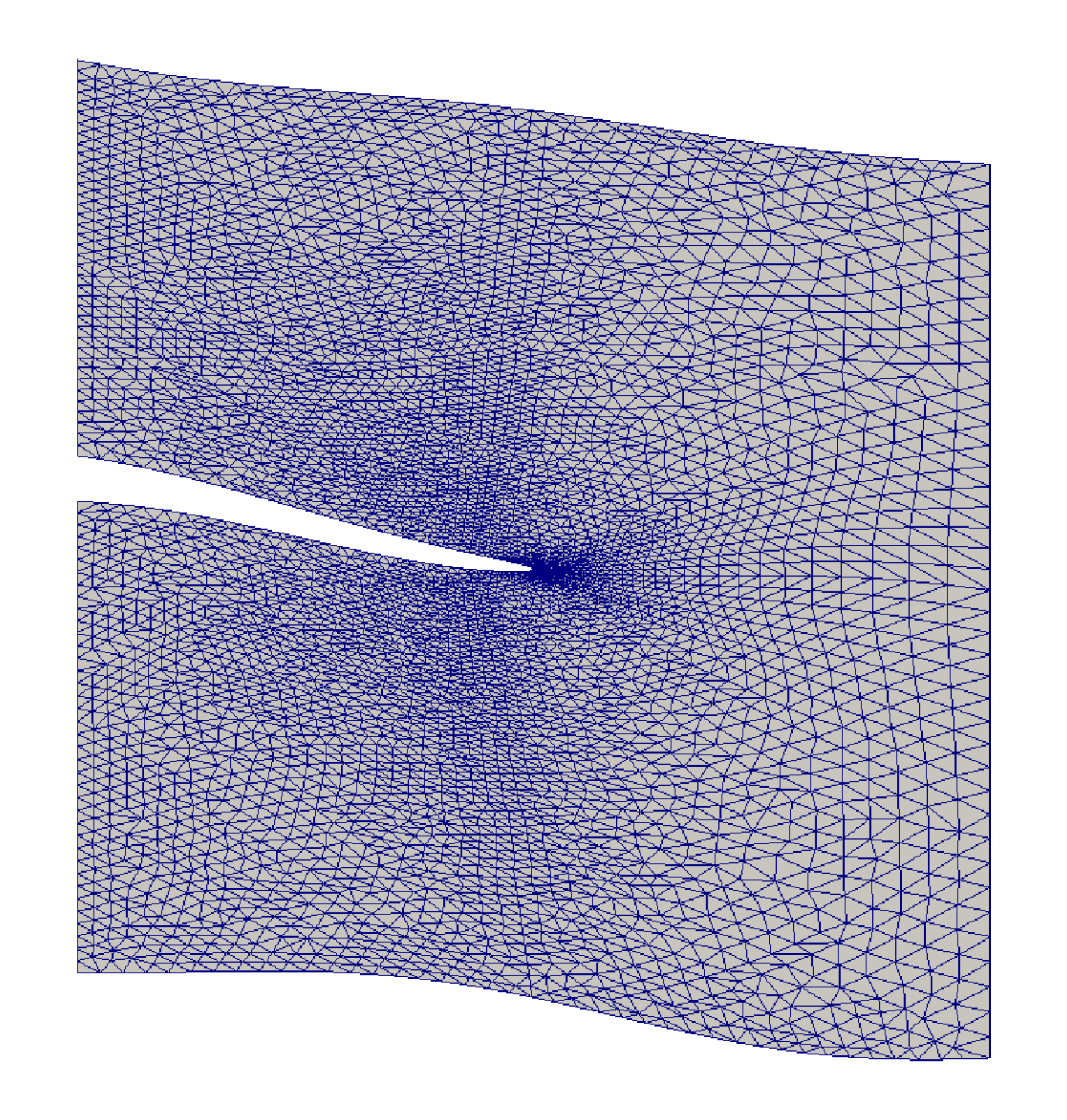}
		\caption{Physical domain $\Omega$ in \\ Examples 2 and 4}
		\label{fig:1_3}
	\end{subfigure}
	\caption{Simulation domains.}\label{fig:1}
\end{figure}

\subsection{Convergence test}
In this test we study the convergence for the space discretization
using an analytical solution.  The domain is $\Omega =
[0,1]\times[-1,1]$, see Figure \ref{fig:1_1}.  We associate the upper
half with the Stokes flow, while the lower half represents the flow in
the poroelastic structure governed by the Biot system.  The appropriate
interface conditions are enforced along the interface $y = 0$.  The
solution in the Stokes region is
\begingroup
\def\arraystretch{1.1}
\begin{align*}
    &\u_f = \pi\cos(\pi t)\begin{pmatrix}-3x+\cos(y) \\ 
y+1 \end{pmatrix}, \quad p_f = \e^t\sin(\pi x)\cos(\frac{\pi y}{2}) + 2\pi \cos(\pi t).
\end{align*}
The Biot solution is chosen accordingly to satisfy the interface
conditions \eqref{eq:mass-conservation}-\eqref{Gamma-fp-1}:
\begin{align*}
    &\u_p = \pi\e^t \begin{pmatrix} \cos(\pi x)\cos(\frac{\pi y}{2}) \\ \frac12\sin(\pi x)\sin(\frac{\pi y}{2}) \end{pmatrix}, \quad p_p = \e^t\sin(\pi x)\cos(\frac{\pi y}{2}), \quad \bbeta_p = \sin(\pi t) \begin{pmatrix}-3x+\cos(y) \\ y+1 \end{pmatrix}.
\end{align*}
\endgroup
The right hand side functions $\f_f,\, q_f,\, \f_p$ and $q_p$ are
computed from \eqref{stokes1}--\eqref{eq:biot2} using the above
solution. The model problem is then complemented with the appropriate
Dirichlet boundary conditions and initial data. The total simulation
time for this test case is $T=0.01$s and the time step is $\Delta t =
10^{-3}$s. The time step is sufficiently small, so that the time discretization
error does not affect the convergence rates.

We study the convergence for two choices of finite element spaces. The
lower order choice is the MINI elements \cite{arnold1984stable} 
$\mathcal{P}^b_1 - \mathcal{P}_1$ for Stokes, the Raviart-Thomas \cite{raviart1977mixed}
$\mathcal{RT}_0-\mathcal{P}_0$ and continuous Lagrangian
$\mathcal{P}_1$ elements for the Biot system, and piecewise constant
Lagrange multiplier $\mathcal{P}_0$. In this case $k_f = 1$, $s_f = 1$, 
$k_p = 0$, $s_p = 0$, and $k_s = 1$, so Theorem
\ref{error-fully-discrete} implies first order of convergence for all
variables. The higher order choice is the Taylor-Hood
\cite{taylor1973numerical} $\mathcal{P}_2-\mathcal{P}_1$ for Stokes,
the Raviart-Thomas $\mathcal{RT}_1-\mathcal{P}_1^{dc}$ and
$\mathcal{P}_2$ for Biot, and $\mathcal{P}^{dc}_1$ for the Lagrange
multiplier, with $k_f = 2$, $s_f = 1$, $k_p = 1$, $s_p = 1$, and $k_s = 2$,
in which case second order convergence rate for all
variables is expected. These theoretical results are verified by the
rates shown in the Table \ref{T1}, where the errors were computed on a
sequence of refined meshes, which are matching along the interface.

\begingroup
\def\arraystretch{1.1}
\begin{table}[ht!]
\begin{center}
\begin{tabular}{c|cc|cc|cc|cc|cc}
\hline
\multicolumn{11}{c}{$\mathcal{P}^b_1 - \mathcal{P}_1,\, \mathcal{RT}_0-\mathcal{P}_0,\, \mathcal{P}_1$ and $\mathcal{P}_0$} \\
\hline
& \multicolumn{2}{c|}{$\|\e_{f}\|_{l^{2}(H^1(\domf))}$} & \multicolumn{2}{c|}{$\|e_{fp}\|_{l^{2}(L^2(\domf))} $} & \multicolumn{2}{c|}{$\|\e_{p}\|_{l^{2}(L^2(\domp))}$} & \multicolumn{2}{c|}{$ \|e_{pp}\|_{l^{\infty}(L^2(\domp))} $} & \multicolumn{2}{c}{$\|{\boldsymbol \e}_{s}\|_{l^{\infty}(H_1(\domp))}$} \\ 
$h$	&	error	&	rate	&	error	&	rate	&	error	&	rate	&	error	&	rate	&	error	&	rate
	\\ \hline
1/8	&	8.96E-03	&	--	&	2.61E-03	&	--	&	1.05E-01	&	--	&	1.03E-01	&	--	&	5.09E-02	&	--	\\
1/16	&	4.47E-03	&	1.0	&	8.33E-04	&	1.6	&	5.23E-02	&	1.0	&	5.17E-02	&	1.0	&	1.34E-02	&	1.9	\\
1/32	&	2.24E-03	&	1.0	&	2.76E-04	&	1.6	&	2.61E-02	&	1.0	&	2.59E-02	&	1.0	&	3.94E-03	&	1.8	\\
1/64	&	1.12E-03	&	1.0	&	9.43E-05	&	1.6	&	1.31E-02	&	1.0	&	1.29E-02	&	1.0	&	1.43E-03	&	1.5	\\
1/128	&	5.59E-04	&	1.0	&	3.28E-05	&	1.5	&	6.53E-03	&	1.0	&	6.47E-03	&	1.0	&	6.32E-04	&	1.2	\\
\hline
%\hline
\multicolumn{11}{c}{$\mathcal{P}_2 - \mathcal{P}_1,\, \mathcal{RT}_1-\mathcal{P}_1^{dc},\, \mathcal{P}_2$ and $\mathcal{P}_1^{dc}$} \\
\hline
& \multicolumn{2}{c|}{$\|\e_{f}\|_{l^{2}(H^1(\domf))}$} & \multicolumn{2}{c|}{$\|e_{fp}\|_{l^{2}(L^2(\domf))} $} & \multicolumn{2}{c|}{$\|\e_{p}\|_{l^{2}(L^2(\domp))}$} & \multicolumn{2}{c|}{$ \|e_{pp}\|_{l^{\infty}(L^2(\domp))} $} & \multicolumn{2}{c}{$\|{\boldsymbol \e}_{s}\|_{l^{\infty}(H_1(\domp))}$} \\ 
$h$	&	error	&	rate	&	error	&	rate	&	error	&	rate	&	error	&	rate	&	error	&	rate
	\\ \hline
1/8	&	1.25E-04	&	--	&	1.31E-03	&	--	&	1.82E-02	&	--	&	1.60E-02	&	--	&	1.54E-01	&	--	\\
1/16	&	2.90E-05	&	2.1	&	3.25E-04	&	2.0	&	4.38E-03	&	2.1	&	4.01E-03	&	2.0	&	3.82E-02	&	2.0	\\
1/32	&	7.06E-06	&	2.0	&	8.07E-05	&	2.0	&	1.08E-03	&	2.0	&	1.00E-03	&	2.0	&	9.51E-03	&	2.0	\\
1/64	&	1.77E-06	&	2.0	&	1.97E-05	&	2.0	&	2.67E-04	&	2.0	&	2.51E-04	&	2.0	&	2.37E-03	&	2.0	\\
1/128	&	4.73E-07	&	1.9	&	4.51E-06	&	2.1	&	6.47E-05	&	2.0	&	6.23E-05	&	2.0	&	5.89E-04	&	2.0	\\
\hline
\end{tabular}
\end{center}
\caption{Example 1: relative numerical errors and convergence rates on matching grids.}
\label{T1}
\end{table}
\endgroup

We also perform a convergence test with the lower order choice of
finite elements on non-matching grids along the interface. We
prescribe the ratio between mesh characteristic sizes to be
$h_{Stokes} = \frac58 h_{Biot}$ as shown in Figure
\ref{fig:1_1}. According to the results shown in Table \ref{T2}, first
order convergence is observed for all variables, which agrees with 
Theorem \ref{error-fully-discrete}.

\begingroup
\def\arraystretch{1.1}
\begin{table}[ht!]
\begin{center}
\begin{tabular}{c|cc|cc|cc|cc|cc}
\hline
\multicolumn{11}{c}{$\mathcal{P}^b_1 - \mathcal{P}_1,\, \mathcal{RT}_0-\mathcal{P}_0,\, \mathcal{P}_1$ and $\mathcal{P}_0$} \\
\hline
& \multicolumn{2}{c|}{$\|\e_{f}\|_{l^{2}(H^1(\domf))}$} & \multicolumn{2}{c|}{$\|e_{fp}\|_{l^{2}(L^2(\domf))} $} & \multicolumn{2}{c|}{$\|\e_{p}\|_{l^{2}(L^2(\domp))}$} & \multicolumn{2}{c|}{$ \|e_{pp}\|_{l^{\infty}(L^2(\domp))} $} & \multicolumn{2}{c}{$\|{\boldsymbol \e}_{s}\|_{l^{\infty}(H_1(\domp))}$} \\ 
$h_{Biot}$	&	error	&	rate	&	error	&	rate	&	error	&	rate	&	error	&	rate	&	error	&	rate
	\\ \hline
1/8	&	1.43E-02	&	--	&	6.06E-03	&	--	&	1.05E-01	&	--	&	1.03E-01	&	--	&	5.09E-02	&	--	\\
1/16	&	7.16E-03	&	1.0	&	1.79E-03	&	1.8	&	5.23E-02	&	1.0	&	5.17E-02	&	1.0	&	1.34E-02	&	1.9	\\
1/32	&	3.58E-03	&	1.0	&	5.81E-04	&	1.6	&	2.61E-02	&	1.0	&	2.59E-02	&	1.0	&	3.94E-03	&	1.8	\\
1/64	&	1.79E-03	&	1.0	&	1.95E-04	&	1.6	&	1.31E-02	&	1.0	&	1.29E-02	&	1.0	&	1.43E-03	&	1.5	\\
1/128	&	8.94E-04	&	1.0	&	6.77E-05	&	1.5	&	6.53E-03	&	1.0	&	6.47E-03	&	1.0	&	6.32E-04	&	1.2	\\
\hline
\end{tabular}
\end{center}
\caption{Example 1: relative numerical errors and convergence rates on non-matching grids.}
\label{T2}
\end{table}
\endgroup

\subsection{Application to flow through fractured reservoirs}
For the rest of the cases, we introduce the reference domain
$\hat{\Omega}$ given by the rectangle $[0, 1] \times [-1,1]$m, see
Figure \ref{fig:1_2}. A fracture, which represents the reference fluid
domain $\hat{\Omega}_f$ is then positioned in the middle of the
rectangle, with the boundary defined by
$$
\hat{x}^2 = 200(0.05 - \hat{y})(0.05 + \hat{y}), \quad \hat{y} \in [-0.05,0.05].
$$
Furthermore, the physical domain $\Omega$, see Figure \ref{fig:1_3},
with more realistic geometry, is defined as a transformation of the
reference domain $\hat{\Omega}$ by the mapping
\cite{Bukac-Nitsche}
\begin{equation*}
	\begin{bmatrix} x \\ y \end{bmatrix}  
        = \begin{bmatrix} \hat x \\ 5\cos(\frac{\hat{x}+\hat{y}}{100})\cos(\frac{\pi \hat{x}+\hat{y}}{100})^2+\hat{y}/2-\hat{x}/10\end{bmatrix}.
\end{equation*}
The external boundary of $\Omega_f$ is denoted as
$\Gamma_{f,inflow}$, while the external boundary of
$\Omega_p$ is split into $\Gamma_{p,\star},$ where 
$\star \in \{left,\,right,\,top,\,bottom \}$.

The next example is focused on modeling the interaction between a
stationary fracture filled with fluid and the surrounding poroelastic
reservoir. We are interested in the solution on the physical domain
$\Omega$. The physical units are meters for length, seconds for time,
and KPa for pressure. The boundary conditions are chosen to be
\begin{align*}
    &\text{Injection:} && \u_f \cdot \n_f = 10, \quad \u_f \cdot \btau_f = 0 && \text{on } \Gamma_{f,inflow}, \\
	&\text{No flow:}  & & \u_p \cdot \n_p = 0 & &\text{on } \Gamma_{p,left}, \\
	&\text{Pressure:} & & p_p = 1000 & &\text{on } \Gamma_{p,bottom}\cup \Gamma_{p,right}\cup \Gamma_{p,top}, \\
	&\text{Normal displacement:} & & \bbeta_p \cdot \n_p = 0 & &\text{on } 
         \Gamma_{p,top} \cup \Gamma_{p, right} \cup \Gamma_{p, bottom}, \\
	&\text{Shear traction:} & & (\bs_p \n_p) \cdot \btau_p = 0 
        & &\text{on } \Gamma_{p,top} \cup \Gamma_{p, right} \cup\Gamma_{p, bottom}.
\end{align*}
The initial conditions are set accordingly to $\bbeta_p(0) = 0$ m and
$p_p(0) = 10^3$ KPa.  The total simulation time is $T=300$ s and the
time step is $\Delta t = 1$ s. The model parameters are given in
Table~\ref{T3}.  These parameters are realistic for hydraulic
fracturing and are similar to the ones used in \cite{Girault-2015}.
The Lam\'{e} coefficients are determined from the Young's modulus $E$
and the Poisson's ratio $\nu$ via the relationships $\lambda_p =
E\nu/[(1+\nu)(1-2\nu)]$, $\mu_p = E/[2(1+\nu)]$. We note that this is
a challenging computational test due to the large variation in
parameter values.

For this and the rest of the test cases we use the Taylor-Hood $\mathcal{P}_2 -
\mathcal{P}_1$ \cite{taylor1973numerical} elements for the fluid
velocity and pressure in the fracture region, the Raviart-Thomas
$\mathcal{RT}_1-\mathcal{P}_1^{dc}$ elements
for the Darcy velocity and pressure, the continuous Lagrangian
$\mathcal{P}_1$ elements for the structure displacement, and the 
$\mathcal{P}_1^{dc}$ elements for the Lagrange multiplier.
\begingroup
\def\arraystretch{1.1}
\begin{table}[ht!]
\begin{center}
%{\scriptsize
\begin{tabular}{|l l l l |}
	\hline
	Parameter                          & Symbol      & Units         & Values                        \\ \hline\hline
	Young's modulus                    & $E$         & (KPa)         & $10^7$                        \\
	Poisson's ratio                    & $\nu$         &               & $0.2$                        \\
	Lam\'{e} coefficient                   & $\lambda_p$ & (KPa)         & $5/18\times 10^{7}$ \\
	Lam\'{e} coefficient                   & $\mu_p$     & (KPa)         & $5/12\times 10^{7}$ \\
	Dynamic viscosity                  & $\mu$       & (KPa s)       & $10^{-6}$                     \\
	Permeability                       & $K$         & (m$^2$)       & $diag(200,50) \times 10^{-12}$ \\
	Mass storativity                   & $s_0$       & (KPa$^{-1}$)  & $6.89\times 10^{-2}$          \\
	Biot-Willis constant               & $\alpha$    &               & 1.0                             \\
	Beavers-Joseph-Saffman coefficient & $\alpha_{BJS}$     &        & 1.0 \\
	Total time                         & T           & (s)           & 300                           \\ \hline
\end{tabular}
%}
\end{center}
\caption{Poroelasticity and fluid parameters in Example 2.}
\label{T3}
\end{table}
\endgroup
\begin{figure}[ht!]
	\centering
	\begin{subfigure}[b]{0.44\textwidth}
		\includegraphics[width=\textwidth]{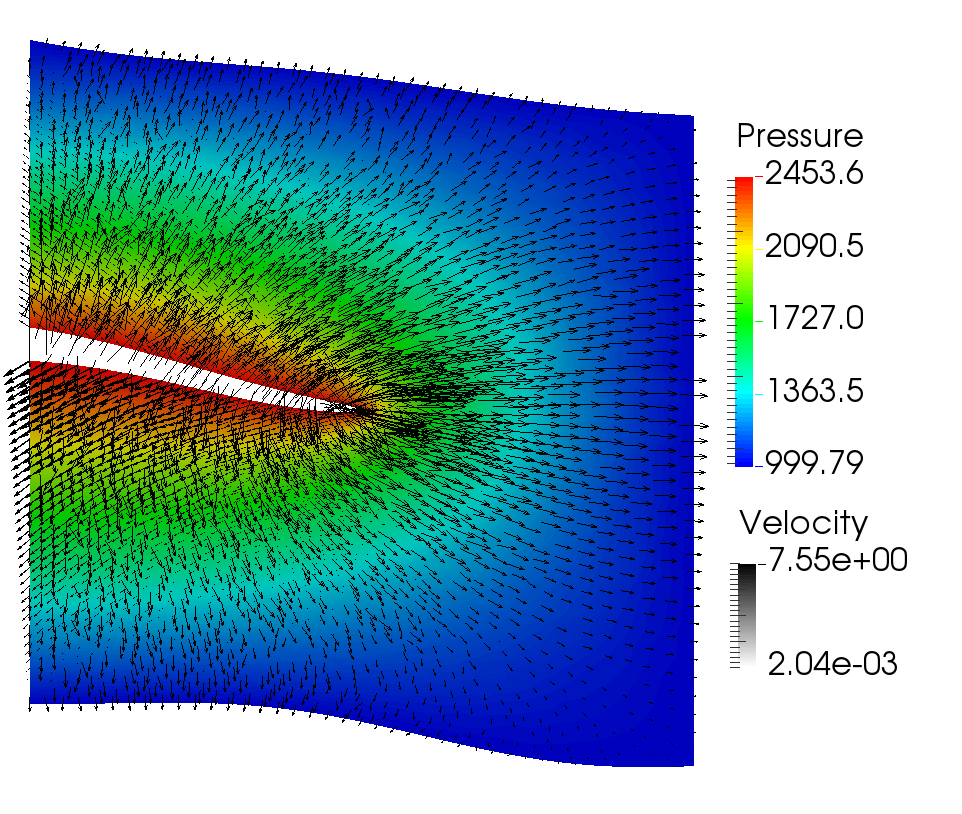}
		\caption{Darcy velocity field (m/s) over pressure (KPa)}
		\label{fig:2_1}
	\end{subfigure}
	\begin{subfigure}[b]{0.44\textwidth}
		\includegraphics[width=\textwidth]{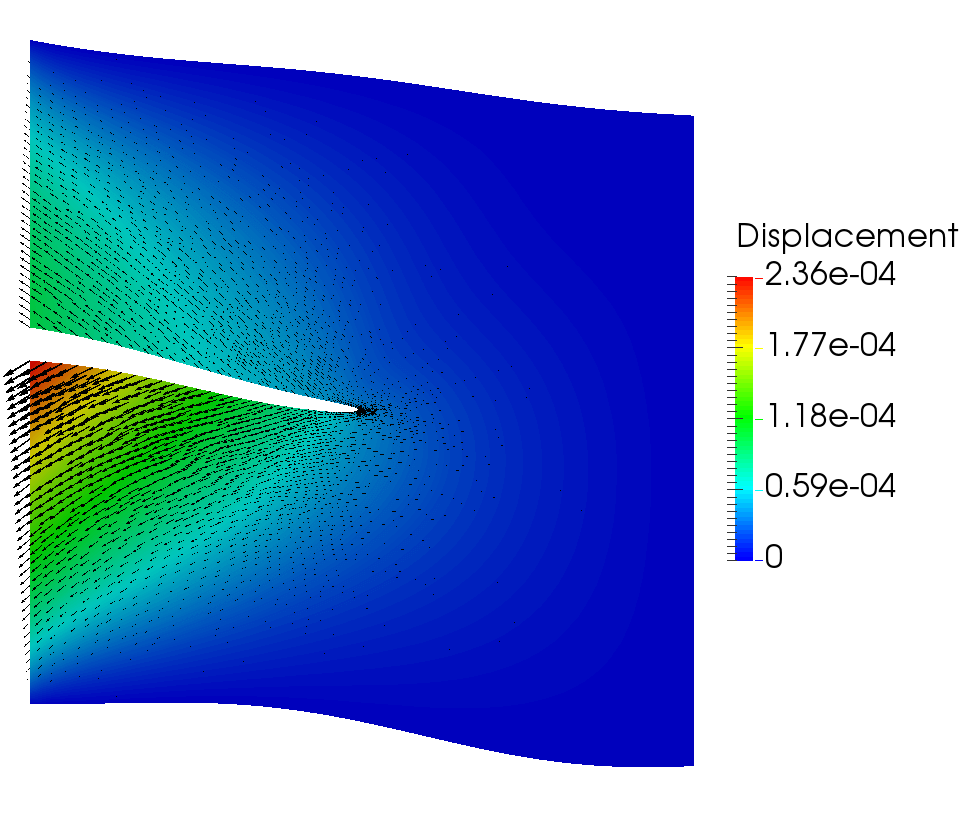}
		\caption{Structure displacement field (m)}
		\label{fig:2_2}
	\end{subfigure}
	\begin{subfigure}[b]{0.44\textwidth}
		\includegraphics[width=\textwidth]{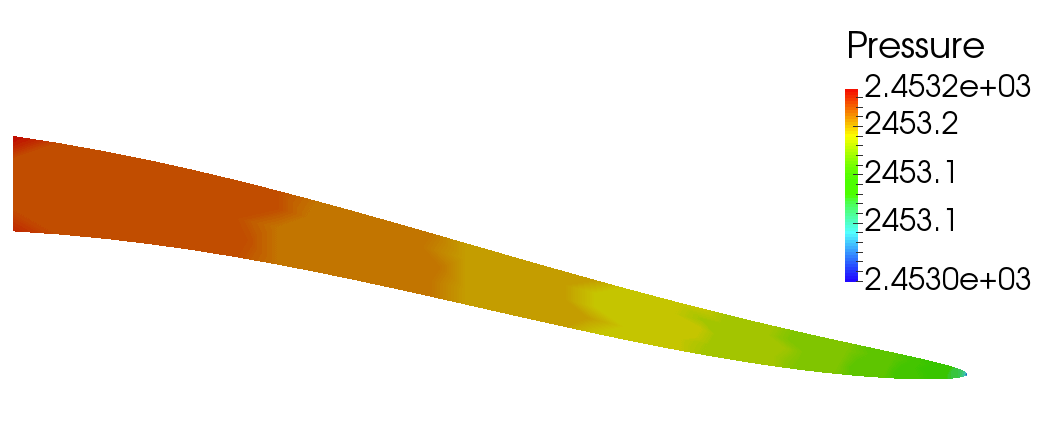}
		\caption{Fluid pressure (KPa) in the fracture}
		\label{fig:2_3}
	\end{subfigure}
	\begin{subfigure}[b]{0.44\textwidth}
		\includegraphics[width=\textwidth]{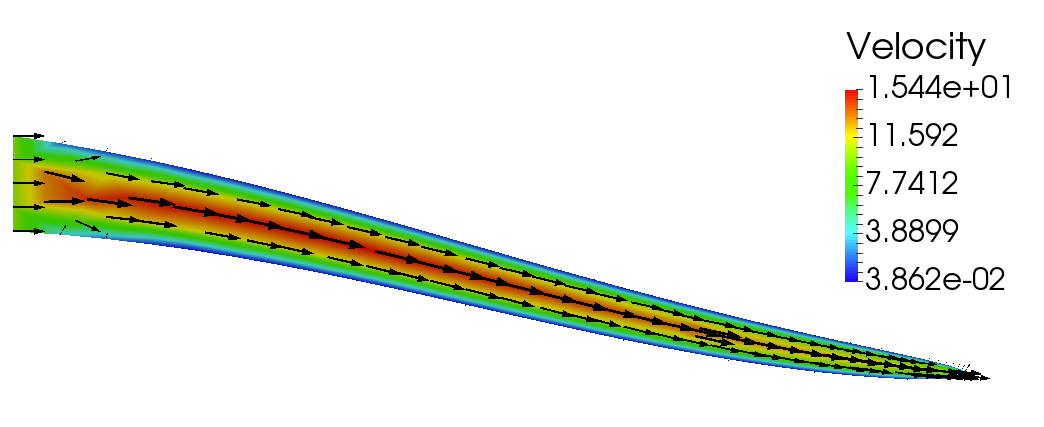}
		\caption{Fluid velocity field (m/s) in the fracture}
		\label{fig:2_4}
	\end{subfigure}
	\caption{Example 2: fluid flow in a fractured reservoir, $t=300$ s.}\label{fig:2}
\end{figure}

Figure \ref{fig:2} shows the structure (top) and fracture (bottom)
regions for the described test case at the final time $T=300$ s. The
grayscale velocity legend in Figure \ref{fig:2_1} is included to show
the range of the Darcy velocity magnitude. We observe channel-like
flow in the fracture region, which concentrates at the tip, and
leak-off into the reservoir. The fluid pressure in the reservoir has
increased in the vicinity of the fracture from the initial value of
1000 KPa to approximately 2450 KPa, which is close to the pressure in
the fracture. The pressure drop in the reservoir in the direction away
from the fracture is significant, but the resulting Darcy velocity is
relatively small, due to the very low permeability. The displacement
field shows that the fracture tends to open as the fluid is being
injected, with the deformation of the rock being largest around the
fracture and quickly approaching zero away from the it, which is
expected due to large stiffness of the rock. This example demonstrates
the ability of the proposed method to handle irregularly shaped
domains with a computationally challenging set of parameters, which
are realistic for hydrualic fracturing in tight rock formations.

\subsection{Flow through fractured reservoir with heterogeneous permeability}

In this example we illustrate the ability of the method to handle
heterogeneous permeability and Young's modulus. For this simulation we
use the reference domain $\hat{\Omega}$, see Figure~\ref{fig:1_2}. The
same boundary and initial conditions as in the previous test case are
specified, and the same physical parameters from Table~\ref{T3} are
used, except for the permeability $K$ and the Young's modulus $E$. The
permeability and porosity data is taken from a two-dimensional
cross-section of the data provided by the Society of Petroleum
Engineers (SPE) Comparative Solution
Project\footnote{www.spe.org/web/csp}.  The SPE data, which is given
on a rectangular $60\times 220$ grid is projected onto the triangular
grid on the reference domain $\hat \Omega$, and visualized in
Figure~\ref{fig:3}. We note that the permeability tensor is isotropic
in this example.  Given the porosity $\phi$ the Young's modulus is
determined from the law
$$ 
E = 10^7 \left(1 - \frac{\phi}{c}\right)^{2.1},
$$
where the constant $c = 0.5$ refers to the porosity at which the
effective Young's modulus becomes zero. This constant is chosen in
general based on the properties of the porous medium. The
justification for this law can be found in
\cite{kovavcik1999correlation}.
\begin{figure}[ht!]
	\centering
	\begin{subfigure}[b]{0.32\textwidth}
		\includegraphics[width=\textwidth]{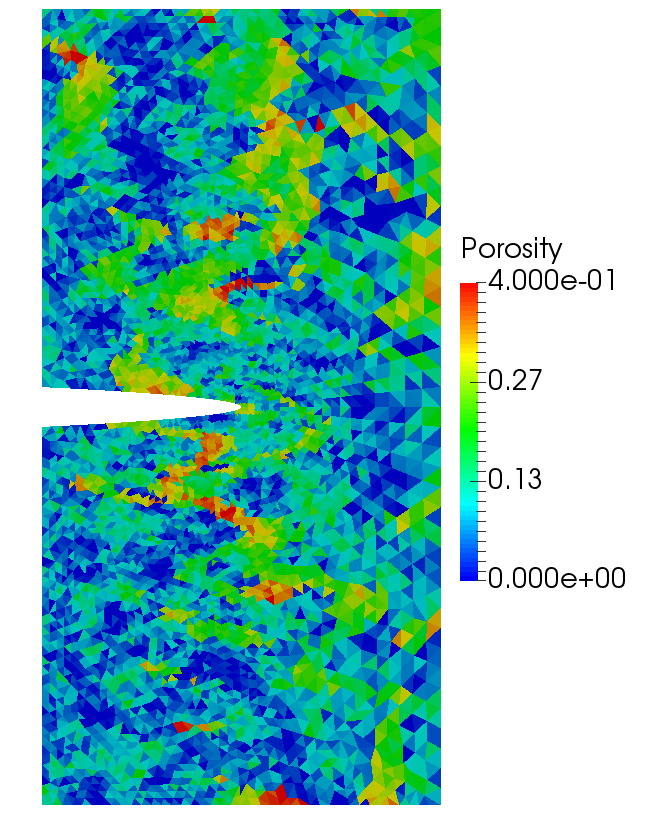}
		\caption{Porosity}
		\label{fig:3_1}
	\end{subfigure}
	\begin{subfigure}[b]{0.32\textwidth}
		\includegraphics[width=\textwidth]{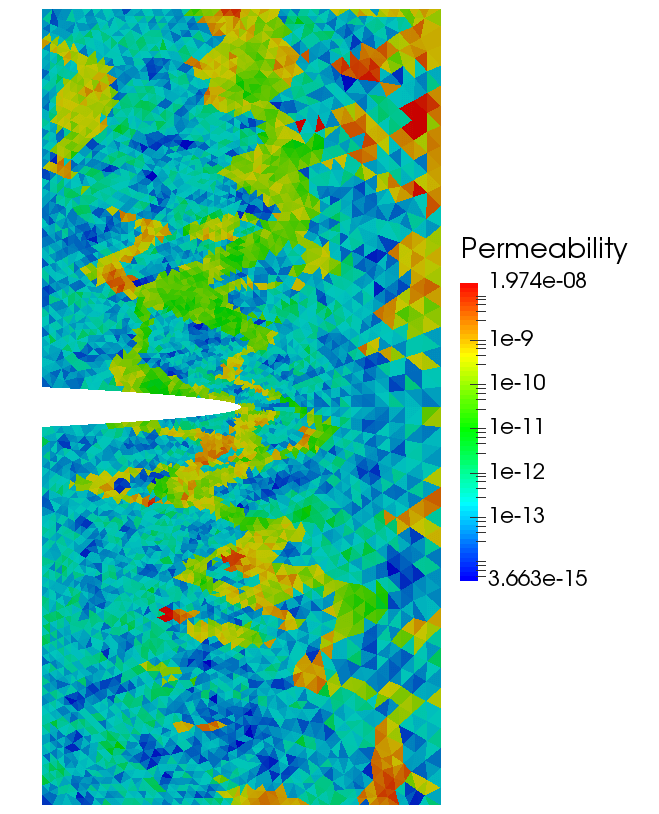}
		\caption{Permeabiltiy}
		\label{fig:3_2}
	\end{subfigure}
	\begin{subfigure}[b]{0.32\textwidth}
		\includegraphics[width=\textwidth]{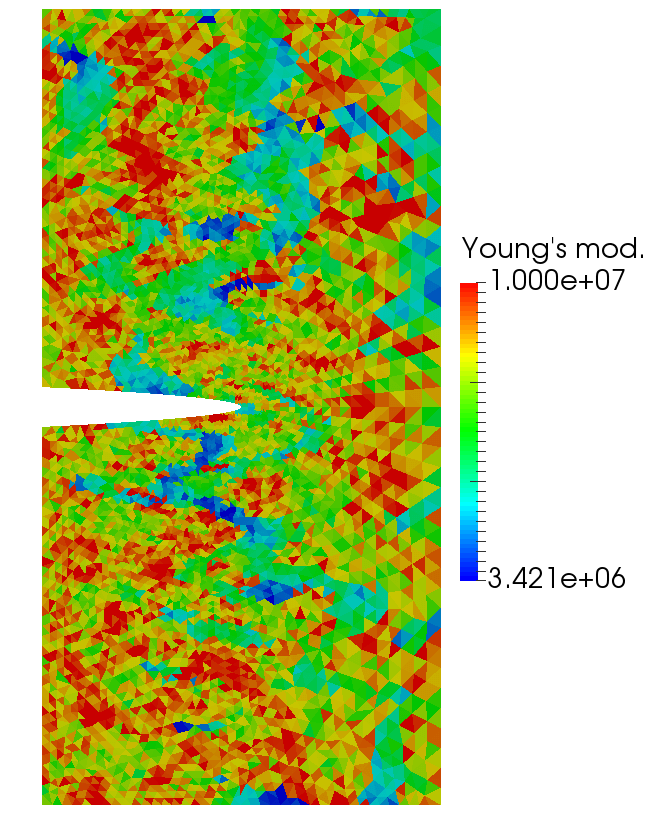}
		\caption{Young's modulus}
		\label{fig:3_3}
	\end{subfigure}
	\caption{Heterogeneous material coefficients in Example 3.}
\label{fig:3}
\end{figure}
\begin{figure}[ht!]
	\centering
	\begin{subfigure}[b]{0.32\textwidth}
		\includegraphics[width=\textwidth]{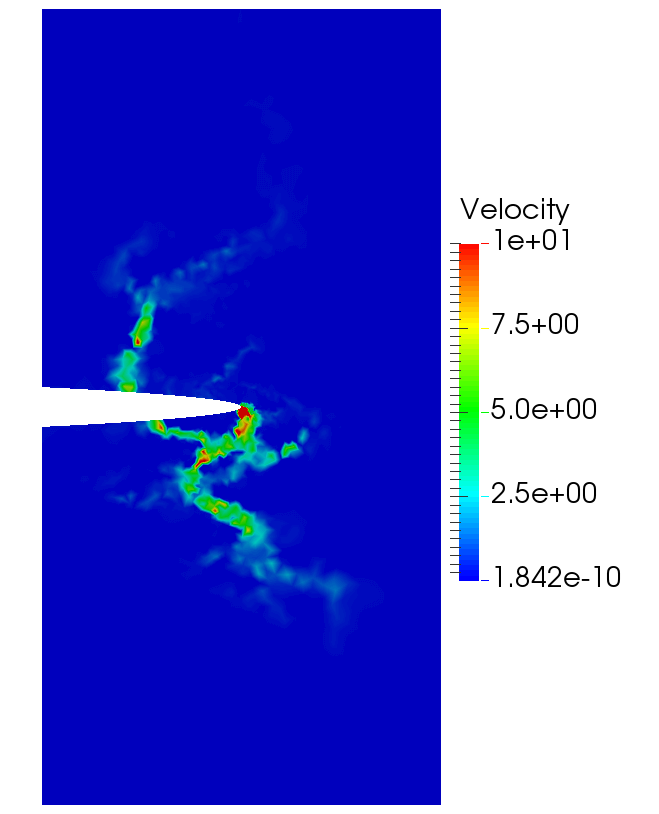}
		\caption{Darcy velocity magnitude (m/s)}
		\label{fig:4_1}
	\end{subfigure}
	\begin{subfigure}[b]{0.32\textwidth}
		\includegraphics[width=\textwidth]{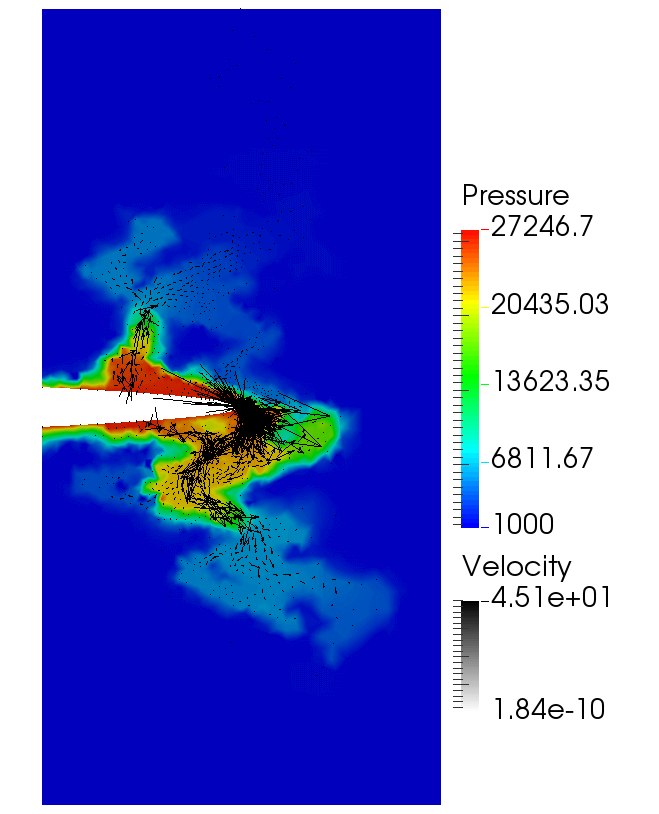}
		\caption{Velocity over pressure (KPa)}
		\label{fig:4_2}
	\end{subfigure}
	\begin{subfigure}[b]{0.32\textwidth}
		\includegraphics[width=\textwidth]{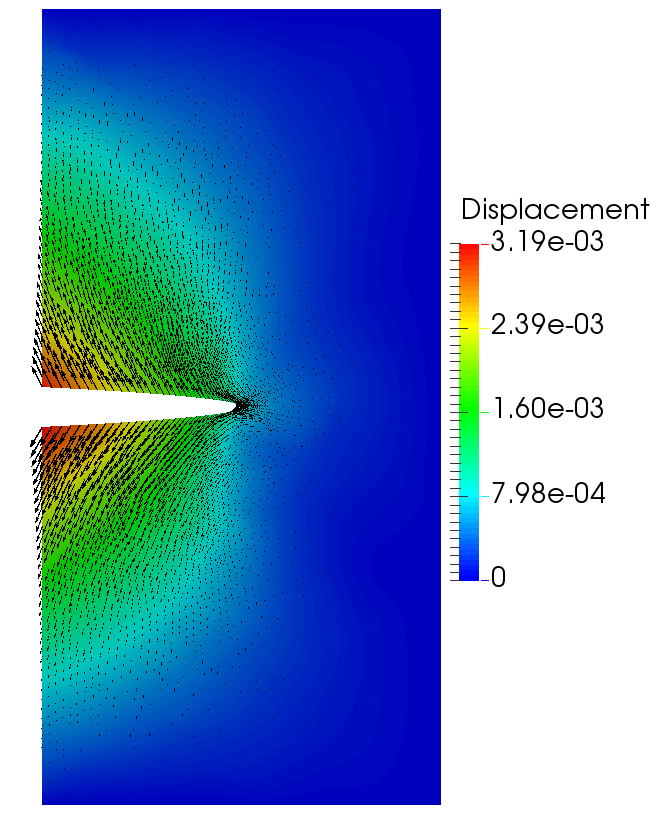}
		\caption{Displacement field (m)}
		\label{fig:4_3}
	\end{subfigure}
	\caption{Example 3: fluid flow in a fractured reservoir with heterogeneous permeability and Young's modulus, $t=300$ s.}\label{fig:4}
\end{figure}

The simulation results at the final time $T = 300s$ are shown in
Figure~\ref{fig:4}. Figures \ref{fig:4_1} and \ref{fig:4_2} show that
the propagation of the fluid in the Darcy region, as evidenced by the
variation in the velocity and pressure, follows the contours of
regions of higher permeability seen in Figure~\ref{fig:3_2}).  As in
the previous test case, the highest velocity in the reservoir is near
the fracture tip. However, the leak-off along the fracture is less
uniform, with a significant leak-off near the middle-top of the
fracture due to the region of relatively high permeability located
there. The last Figure \ref{fig:4_3} depicts the nonuniform
displacement field in the reservoir caused by the heterogeneous
Young's modulus. We note that the effect of heterogeneity of the
elastic coefficients is less pronounced due to the large stiffness of
the rock. The general displacement profile is similar to the
homogeneous case.

\subsection{Sensitivity analysis}
The goal of this section is to investigate how the developed model
behaves when the parameters are modified, moving from mild
non-physical values toward more realistic values that resemble the
ones used in the hydraulic fracturing examples. We progressively
update the parameters $K$, $s_0$ and $E$ as shown in Table~\ref{T4},
while the rest of the parameters are taken from Table~\ref{T3}. All test
cases in this section are governed by the same boundary and initial 
conditions as in the previous two examples. 
\begingroup
\def\arraystretch{1.1}
\begin{table}[ht!]
	\begin{center}
		\begin{tabular}{| l | c  c  c  |}
			\hline
			   & $K$ (m$^2$)       & $s_0$ (KPa$^{-1}$) & $E$ (KPa) \\ \hline\hline
			A  & $\bI\times 10^{-6}$           & 1.0                  & $10^3$    \\
			B  & $diag(200,50)\times 10^{-12}$ & 1.0                  & $10^3$    \\
			C  & $diag(200,50)\times 10^{-12}$ & $10^{-2}$          & $10^3$    \\
			D  & $diag(200,50)\times 10^{-12}$ & $10^{-2}$          & $10^{10}$ \\ \hline
		\end{tabular}
	\end{center}
	\caption{Set of parameters for the sensitivity analysis in Example 4.}
	\label{T4}
\end{table}
\endgroup
%%%
\begin{figure}[ht!]
	\centering
	\begin{subfigure}[b]{0.4\textwidth}
		\includegraphics[width=5.65cm,height=5.65cm,keepaspectratio]{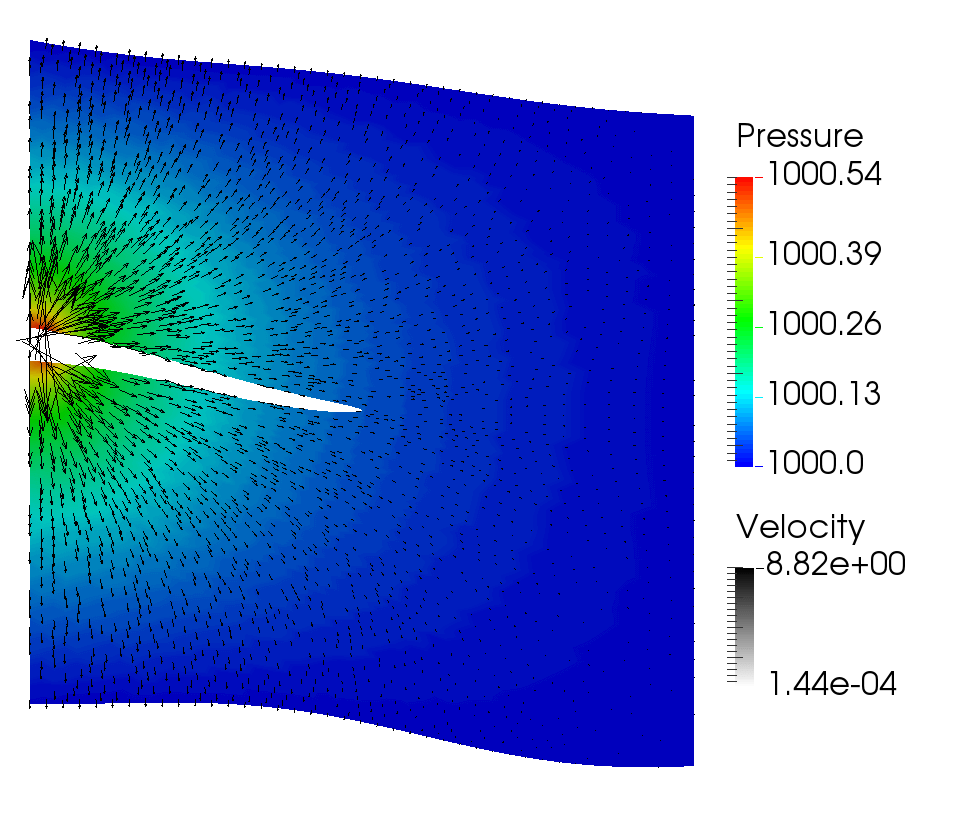}
		\label{fig:5_1}
	\end{subfigure}
	\begin{subfigure}[b]{0.4\textwidth}
		\includegraphics[width=5.65cm,height=5.65cm,keepaspectratio]{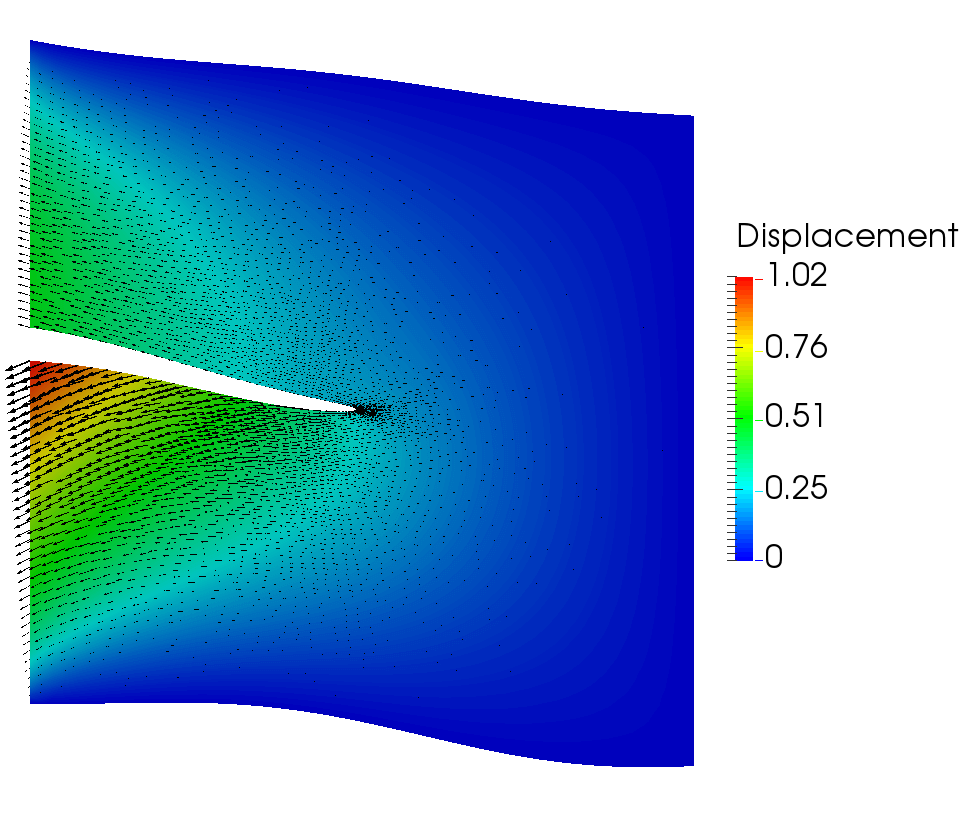}
		\label{fig:5_2}
	\end{subfigure}
	\begin{subfigure}[b]{0.4\textwidth}
		\includegraphics[width=5.65cm,height=5.65cm,keepaspectratio]{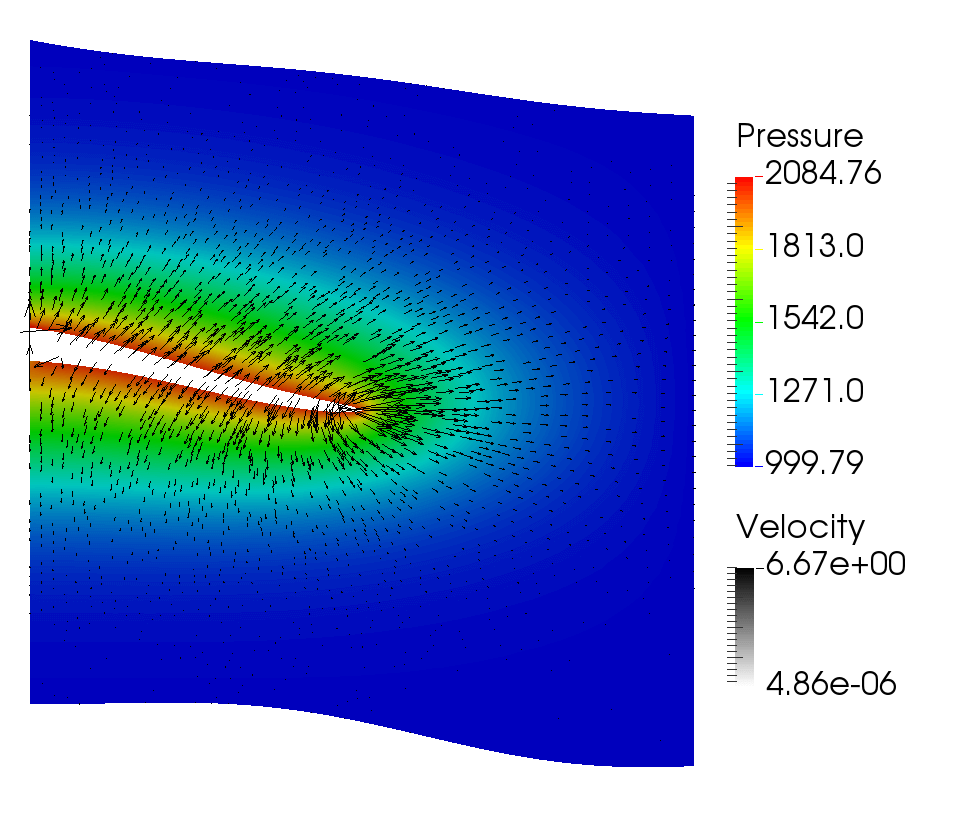}
		\label{fig:5_3}
	\end{subfigure}
	\begin{subfigure}[b]{0.4\textwidth}
		\includegraphics[width=5.65cm,height=5.65cm,keepaspectratio]{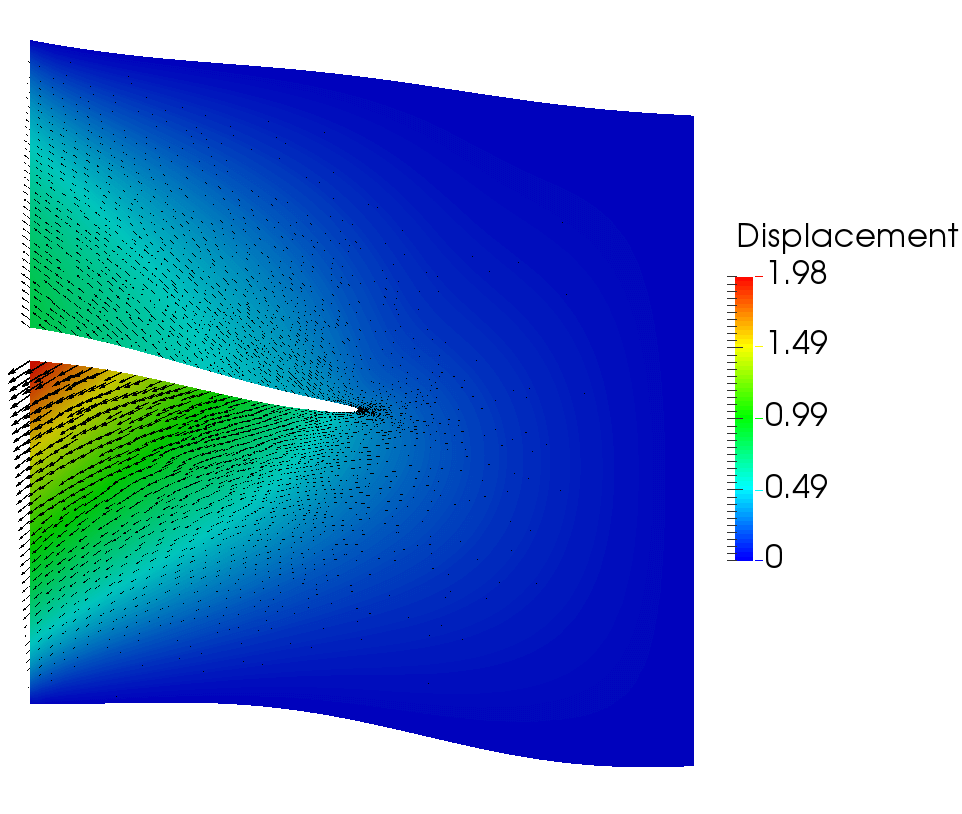}
		\label{fig:5_4}
	\end{subfigure}
	\begin{subfigure}[b]{0.4\textwidth}
		\includegraphics[width=5.65cm,height=5.65cm,keepaspectratio]{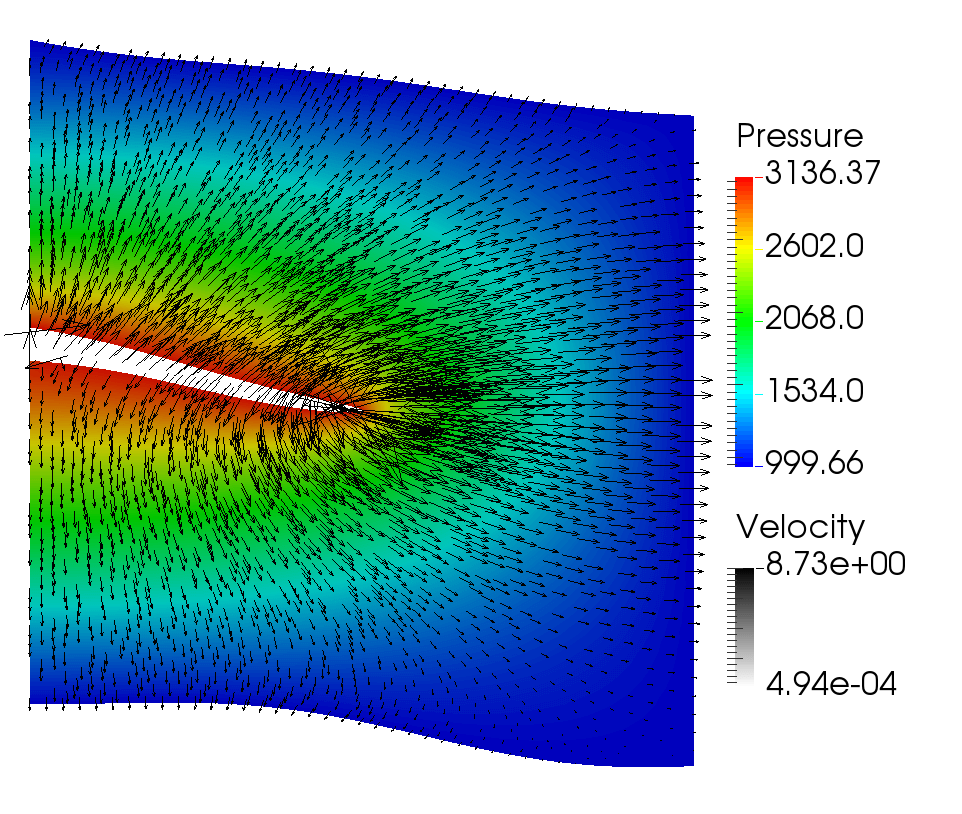}
		\label{fig:5_5}
	\end{subfigure}
	\begin{subfigure}[b]{0.4\textwidth}
		\includegraphics[width=5.65cm,height=5.65cm,keepaspectratio]{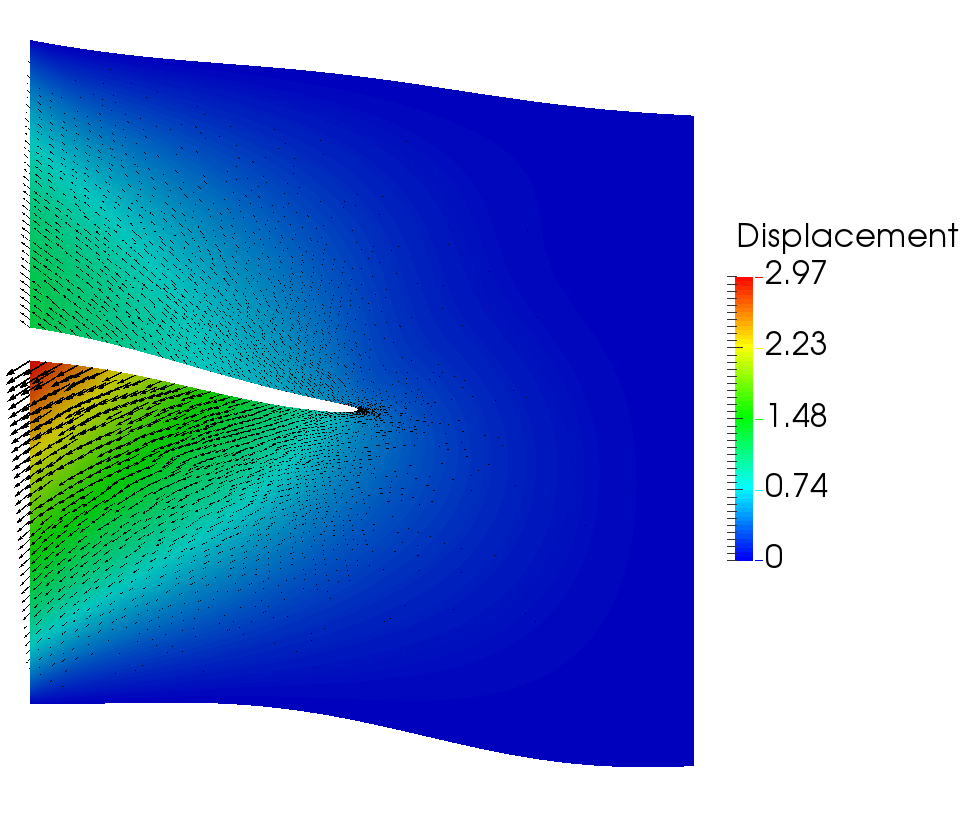}
		\label{fig:5_6}
	\end{subfigure}
	\begin{subfigure}[b]{0.4\textwidth}
		\includegraphics[width=5.65cm,height=5.65cm,keepaspectratio]{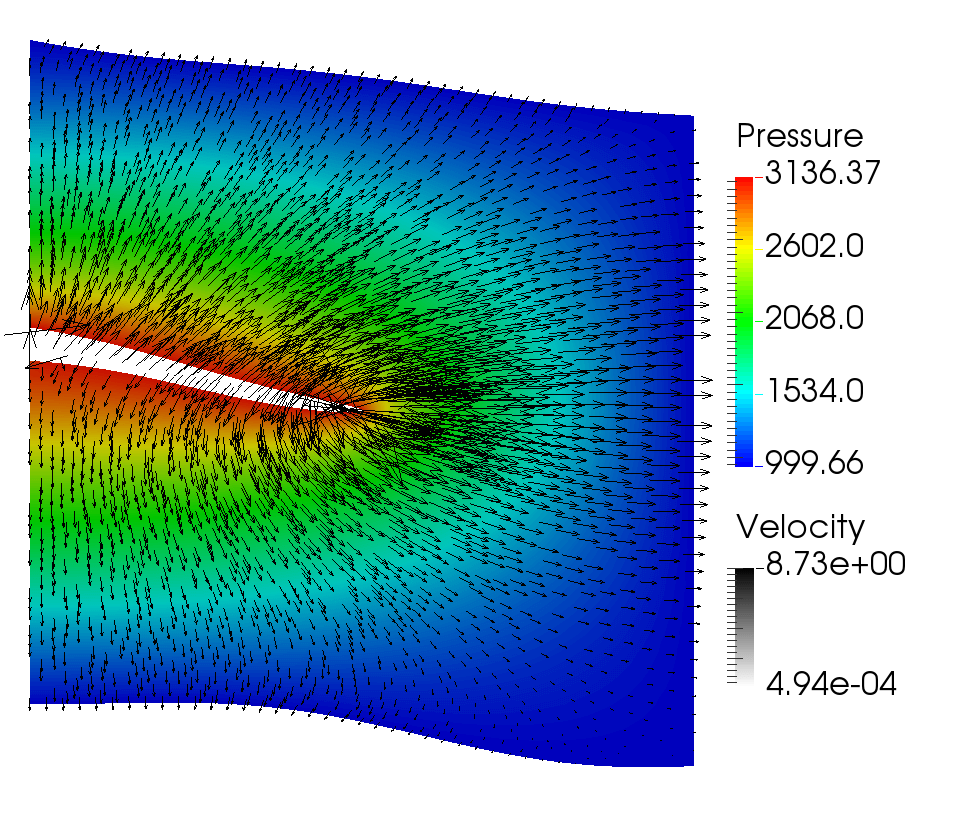}
		\label{fig:5_7}
	\end{subfigure}
	\begin{subfigure}[b]{0.4\textwidth}
		\includegraphics[width=5.65cm,height=5.65cm,keepaspectratio]{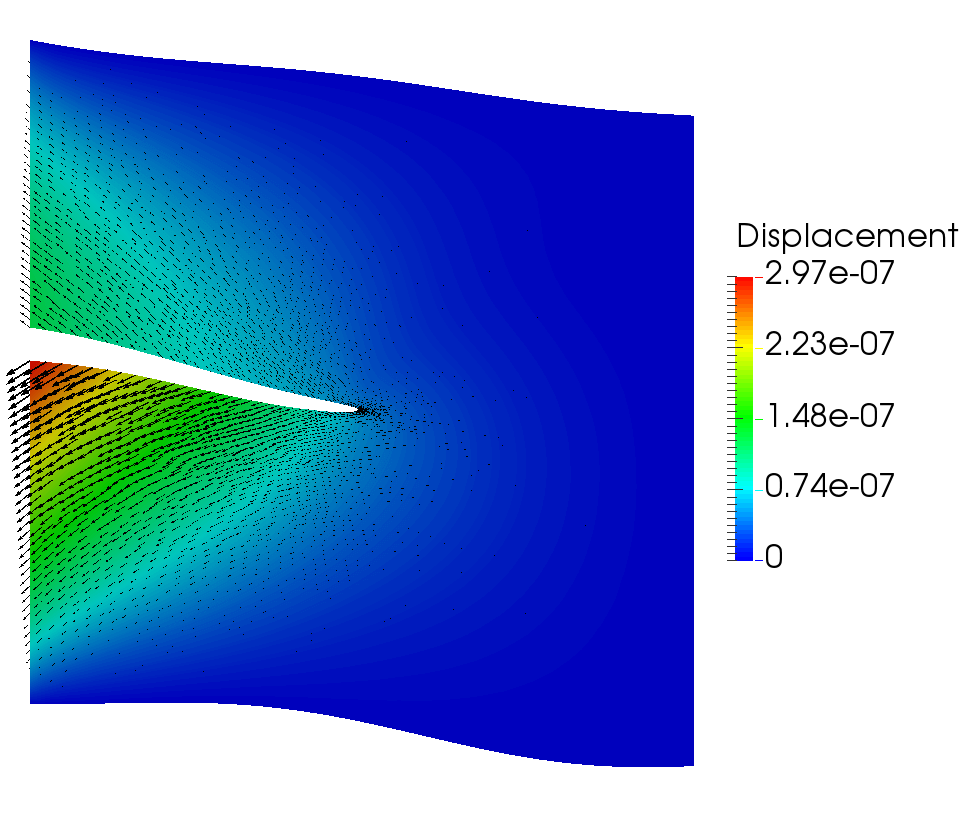}
		\label{fig:5_8}
	\end{subfigure}
	\caption{Sensitivity analysis simulations, $t=300$ s. Cases A
          to D are shown from top to bottom. The left figures show the
          Darcy velocity (m/s) superimposed with contour plot for the
          pressure (KPa). The right figures show the structure
          displacement field (m) over the displacement magnitude
          contour plot. The grayscale velocity legend shows the range
          of velocity magnitude.}
\label{fig:5}
\end{figure}

\begin{description}
	\item[Case A:] The pressure gradient is small as seen from the
          contour plot, this is due to the large permeability.  Also,
          from continuity of flux across the interface, one
          would expect to see that the magnitude of the Darcy velocity
          is close to the magnitude of the Stokes velocity, which we
          indeed observe in all the simulations.
	\item[Case B:] The permeability now is 4 orders of magnitude
          smaller, resulting in a larger pressure gradient, which is consistent
          with Darcy's law \eqref{eq:biot1}. Also, more
          flow is going toward the tip of the fracture, since its
          walls are now much less permeable. The displacement
          magnitude is also larger, while keeping the same profile.
	\item[Case C:] This case shows how the model reacts to
          decrease in mass storativity - which is by exhibiting larger
          pressure gradient and displacement magnitude while keeping
          the overall behavior as in case B.
	\item[Case D:] The last case is to show the effect of a
          significant change in Young's modulus. Increasing it by 7
          orders of magnitude, which makes the material much stiffer,
          results in the displacement being decreased by 7 orders of
          magnitude as expected.
\end{description}

The above results show that the displacement magnitude directly
increases with the magnitude of the pressure, while the profile of the
displacement field stays the same. This is consistent with the
dependence of the poroelastic stress on the fluid pressure, see
\eqref{stress-defn}.  In addition, the displacement magnitude is
inversely proportional to the Young’s modulus, which is consistent
with the constitutive law for the elastic stress in
\eqref{stress-defn}.

\section{Conclusions}
We have studied the interaction of a free fluid with a fluid within a
poroelastic medium. After stating the governing equations and
discussing the appropriate boundary and interface conditions we
considered a numerical discretization of the problem using a mixed
finite element method. A Lagrange multiplier is used to impose weakly
the continuity of mormal velocity interface condition, which is of
essesntial type in the mixed Darcy formulation. We show that the
method is stable and convergent of optimal order, even in the case of
non-matching grids across the interface. Computational experiments
illustrate that this method is an effective approach for simulating
fluid-poroelastic structure interaction with a wide range of physical
parameters, including heterogeneous media. The Lagrange multiplier
formulation is suitable for parallel non-overlapping domain
decomposition algorithms and multiscale approximations via coarse
mortar spaces. These topics will be explored in future research.

% % % % % % % % % % % % % % % % % % % % % % % % % % % % % % % % % % % % % % % % % % %
\section{Appendix: fully discrete analysis}
In this section we provide a detailed analysis of the stability and 
convergence of the fully discrete method 
\eqref{disc-h-weak-1}-\eqref{disc-h-b-gamma}. We will utilize the following 
discrete Gronwall inequality \cite{Quarteroni-Valli}.
\begin{lemma}[Discrete Gronwall lemma] \label{disc-Gronwall}
	Let $\tau > 0$, $B \ge 0$, and let $a_n,b_n,c_n,d_n$, $n \geq 0$, be non-negative sequences 
such that $a_0 \leq B$ and
	\begin{align*}
	a_n + \tau \sum_{l=1}^n b_l \leq \tau \sum_{l=1}^{n-1} d_la_l  + 
         \tau \sum_{l=1}^n c_l +B, \quad n\geq 1.
	\end{align*}
Then,
	\begin{align*}
	a_n + \tau \sum_{l=1}^n b_l \leq \exp(\tau \sum_{l=1}^{n-1}d_l)
       \left( \tau \sum_{l=1}^n c_l + B \right) , \quad n \geq 1.
	\end{align*}
	\\
\end{lemma}

\begin{proof}[Proof of Theorem \ref{stability-fully-discrete}] 
	We choose 
	\[(\v_{f,h},w_{f,h},\v_{p,h},w_{p,h},\bxi_{p,h},\mu_h) = (\u^{n}_{f,h},p^{n}_{f,h},\u^{n}_{p,h},p^{n}_{p,h},d_{\tau} \bbeta^{n}_{p,h},\lam_h)\] 
	in \eqref{disc-h-weak-1}--\eqref{disc-h-b-gamma} and use the discrete analog 
of \eqref{int-parts}:
	\begin{align}
	\int_{S}u^n d_{\tau}\phi^n = \frac12 d_{\tau}\|\phi^n\|^2_{L^2(S)} 
        + \frac12 \tau \|d_{\tau}\phi^n\|^2_{L^2(S)} \label{sum-parts}
	\end{align}
	to obtain the energy equality
	\begin{align}
	& \frac12 d_{\tau}\left(s_0 \|p^n_{p,h}\|^2_{L^2(\O_p)} 
        + a^e_{p}(\bbeta^n_{p,h},\bbeta^n_{p,h})\right)
	+ \frac{\tau}{2} \left( s_0 \|d_{\tau}p^n_{p,h}\|^2_{L^2(\O_p)} 
	+a^e_{p}(d_{\tau}\bbeta^n_{p,h},d_{\tau}\bbeta^n_{p,h}) \right)  \nonumber \\
& \qquad
	+ a_{f}(\u^n_{f,h},\u^n_{f,h}) + a^d_{p}(\u^n_{p,h},\u^n_{p,h}) 
        +  |\u^n_{f,h} - d_{\tau}\bbeta^n_{p,h} |^2_{a_{BJS}} 
	= \mathcal{F} (t_n). \label{disc-energy-eq}
	\end{align}
	The right-hand side can be bounded as follows, using inequalities \eqref{CS} and \eqref{Young},
	\begin{align}
	\mathcal{F}(t_n) &= (\f_{f}(t_n),\u^n_{f,h}) + (\f_{p}(t_n),\dt \bbeta^n_{p,h}) 
         +(q_{f}(t_n),p^n_{f,h}) + (q_{p}(t_n),p^n_{p,h}) \nonumber \\
	&\leq (\f_{p}(t_n),\dt \bbeta^n_{p,h}) 
        + \frac{\epsilon_1}{2} \left(\|\u^n_{f,h}\|^2_{L^2(\O_f)} + \|p^n_{f,h}\|^2_{L^2(\O_f)} 
        + \|p^n_{p,h}\|^2_{L^2(\O_p)}\right) \nonumber \\
	&\quad + C \frac{1}{2\epsilon_1}\left(\|\f_{f}(t_n)\|^2_{L^2(\O_f)} + \|q_{f}(t_n)\|^2_{L^2(\O_f)} + \|q_{p}(t_n)\|^2_{L^2(\O_p)} \right). \label{disc-rhs-bound}
	\end{align}
	Combining \eqref{disc-energy-eq} and \eqref{disc-rhs-bound}, summing up over the 
        time index $n=1,...,N$, multiplying by $\tau$ and using the coercivity of the 
        bilinear forms \eqref{C1}--\eqref{C3}, we obtain
	\begin{align}
	&s_0 \|p^N_{p,h}\|^2_{L^2(\O_p)} + \|\bbeta^N_{p,h}\|^2_{H^1(\O_p)}
	+ \tau\sum_{n=1}^{N}\left(\|\u^n_{f,h}\|^2_{H^1(\O_{f})}  
        + \|\u^n_{p,h}\|^2_{L^2(\O_{p})}+|\u^n_{f,h} - \dt \bbeta^n_{p,h}|^2_{a_{BJS}} \right)
	\nonumber \\
	&\quad + \tau^2 \sum_{n=1}^{N} 
        \left(  s_0\|\dt p^n_{p,h}\|^2_{L^2(\O_p)} + \|\dt \bbeta^n_{p,h}\|^2_{H^1(\O_p)}\right)
	\nonumber \\
	&\leq C\Bigg(s_0 \|p^0_{p,h}\|^2_{L^2(\O_p)}  + \|\bbeta^0_{p,h}\|^2_{H^1(\O_p)}  
         + \epsilon_1 \tau\sum_{n=1}^{N} \left( \|\u^n_{f,h}\|^2_{L^2(\O_f)} 
         + \|p^n_{f,h}\|^2_{L^2(\O_f)} +\|p^n_{p,h}\|^2_{L^2(\O_p)}\right)  
        \nonumber \\
	&\quad 
         + \epsilon_1^{-1}\tau \sum_{n=1}^{N}\left( \|\f_{f}(t_n)\|^2_{L^2(\O_f)} 
         + \|q_{f}(t_n)\|^2_{L^2(\O_f)} + \|q_{p}(t_n)\|^2_{L^2(\O_p)} \right)
         + \tau\sum_{n=1}^{N}(\f_{p}(t_n),\dt \bbeta^n_{p,h}) \Bigg).
        \label{disc-energy-ineq1}
	\end{align}
To bound the last term on the right we use summation by parts:
	\begin{align}
	&\tau\sum_{n=1}^{N}(\f_{p}(t_n),\dt \bbeta^n_{p,h}) 
         = (\f_p(t_N),\bbeta_{p,h}^N)- (\f_p(0),\bbeta_{p,h}^0) 
         - \tau\sum_{n=1}^{N-1}(\dt \f_{p}^n,\bbeta^n_{p,h}) \nonumber \\
	& \qquad\leq \frac{\epsilon_1}{2} \| \bbeta_{p,h}^N\|^2_{L^2(\O_p)} 
        +  \frac{1}{2\epsilon_1}\|\f_p(t_N)\|^2_{L^2(\Omega_p)} 
        + \frac{\tau}{2}\sum_{n=1}^{N-1}\|\bbeta^n_{p,h}\|^2_{L^2(\O_p)}
\nonumber \\
& \qquad\qquad        
        + \frac12 \left( \| \bbeta_{p,h}^0\|^2_{L^2(\O_p)} + \| \f_{p}(0)\|^2_{L^2(\O_p)} 
        +  \tau\sum_{n=1}^{N-1}\|\dt \f_{p}^n\|^2_{L^2(\O_p)}\right). \label{sum-by-parts} 
	\end{align}
Next using the inf-sup condition \eqref{inf-sup} for $(p^n_{f,h},p^n_{p,h}, \lam^n_h)$ we obtain, 
in a similar way to \eqref{p-bound},
\begin{align}
    & \epsilon_2\tau\sum_{n=1}^N \Big(\|p^n_{f,h}\|^2_{L^2(\O_f)} + \|p^n_{p,h}\|^2_{L^2(\O_p)} 
      + \|\lam^n_h\|^2_{ \Lambda_h} \Big) \nonumber \\
    &\quad\quad\leq C \epsilon_2\tau\sum_{n=1}^N 
    \Big( \|\f_{f}(t_n)\|^2_{L^2(\O_f)} + \|\f_{p}(t_n)\|^2_{L^2(\O_p)} 
    + \|\u^n_{f,h}\|^2_{H^1(\O_f)} + \|\u^n_{p,h}\|^2_{L^2(\O_p)} \nonumber \\
    &\quad\quad\quad+ \|\bbeta^n_{p,h}\|^2_{H^1(\O_p)} 
    + |\u^n_{f,h}-\dt \bbeta^n_{p,h}|^2_{a_{BJS}}\Big). \label{disc-inf-sup}
\end{align}
Combining \eqref{disc-energy-ineq1}--\eqref{disc-inf-sup}, and taking $\epsilon_2$ 
small enough, and then $\epsilon_1$ small enough, and using Lemma \ref{disc-Gronwall} 
with $a_n = \|\bbeta^n_{p,h}\|^2_{H^1(\O_p)}$, gives
\begin{align}
	&s_0 \|p^N_{p,h}\|^2_{L^2(\O_p)} + \|\bbeta^N_{p,h}\|^2_{H^1(\O_p)}+\tau \sum_{n=1}^{N} \left[\|\u^n_{f,h}\|^2_{H^1(\O_{f})} +\|\u^n_{p,h}\|^2_{L^2(\O_{p})}+|\u^n_{f,h} - \dt \bbeta^n_{p,h}|^2_{a_{BJS}} \right]\nonumber \\
	&\quad+ \tau^2\sum_{n=1}^{N} \left[ s_0\|\dt p^n_{p,h}\|^2_{L^2(\O_p)} + \|\dt \bbeta^n_{p,h}\|^2_{H^1(\O_p)}  \right] + \tau \sum_{n=1}^{N} \left[\|p^n_{p,h}\|^2_{L^2(\O_p)} +\|p^n_{f,h}\|^2_{L^2(\O_f)}+\|\lam^n_{h}\|^2_{\Lambda_h} \right] \nonumber\\
	&\leq C\exp(T)\Big(s_0 \|p^0_{p,h}\|^2_{L^2(\O_p)}  + \|\bbeta^0_{p,h}\|^2_{H^1(\O_p)}+ \| \f_{p}(0)\|^2_{L^2(\O_p)} \nonumber\\
	&\quad+\tau \sum_{n=1}^{N} \Big[ \|\f_{f}(t_n)\|^2_{L^2(\O_f)}+\|\f_{p}(t_n)\|^2_{L^2(\O_p)} + \|q_{f}(t_n)\|^2_{L^2(\O_f)} + \|q_{p}(t_n)\|^2_{L^2(\O_p)}+\|\dt \f_{p}\|^2_{L^2(\O_p)}\Big] \Big), \nonumber 
\end{align}
which implies the statement of the theorem using the appropriate space-time norms.
\end{proof}

For the sake of space, we do not present the proof of Theorem~\ref{error-fully-discrete}.
The error equations are obtained by subtracting the first two equations of the 
fully discrete formulation \eqref{disc-h-weak-1}--\eqref{disc-h-weak-2} from the their 
continuous counterparts \eqref{h-cts-1}--\eqref{h-cts-2}:
\begin{align}
	&  a_{f}(\e^n_f,\v_{f,h})
	+ a^d_{p}(\e^n_p,\v_{p,h})
	+ a^e_{p}(\e^n_s,\bxi_{p,h}) 
	+a_{BJS}(\e^n_f,\dt\e^n_{s};\v_{f,h},\bxi_{p,h})+ b_f(\v_{f,h},e^n_{fp}) + b_p(\v_{p,h},e^n_{pp})\nonumber \\ 
	&\quad 
	+\alpha b_p(\bxi_{p,h},e^n_{pp})  
	+ b_{\Gamma}(\v_{f,h},\v_{p,h},\bxi_{p,h};e^n_\lam) +\left(s_0\,\dt e^n_{pp},w_{p,h}\right) - \alpha b_p(\dt e^n_{s},w_{p,h})
	- b_p(\e^n_{p},w_{p,h})  \nonumber \\
	&\quad
	- b_f(\e^n_{f},w_{f,h})
	= (s_0 r_n(p_p),w_{p,h})+a_{BJS}(0,r_n(\bbeta_{p});\v_{f,h},\bxi_{p,h})   - \alpha b_p(r_n(\bbeta_p),w_{p,h}),
\end{align}
where $r_n$ denotes the difference between the time derivative and its discrete analog:
\begin{align*}
r_n(\theta) &= \d_t \theta(t_n) -\dt \theta^n.
\end{align*}
It is easy to see that \cite[Lemma 4]{BYZ} for sufficiently smooth $\theta$, 
$$
\tau \sum_{n=1}^N\|r_n(\theta)\|^2_{H^k(S)} \leq C\tau^2\|\d_{tt}\theta\|^2_{L^2(0,T;H^k(S))}.
$$
The proof of Theorem~\ref{error-fully-discrete} follows the structure of 
the proof of Theorem~\ref{thm:error-semi-discrete}, using discrete-in-time arguments 
as in the proof of Theorem~\ref{stability-fully-discrete}.

%>>>>>>>>>>>>>>>>>>>>>>>>>>>>>>>>>>>>>>>>>>>>>>>>>>>>>>>>>>>>>>>>>>>>>>>>>>>>>>>>>>>>>>>>>>>>>>>>>>>>>>>>>>>>>>>>>>>>>>>>>>>>>>>>>
%\bibliographystyle{splncs03}
%\bibliographystyle{ieeetr}
%\forgetcommand{\v}
\bibliographystyle{plain}
\bibliography{fracture}

\end{document}